\documentclass[12pt]{article}
\usepackage{amsmath, amsthm, amssymb}

\title{Canonical matrices for
linear matrix
problems\footnotetext{This is the author's version of a work that was published in Linear Algebra Appl. 317 (2000) 53--102.}}
\author{Vladimir V. Sergeichuk\\
Institute of Mathematics \\ Tereshchenkivska 3,
Kiev, Ukraine\\sergeich@imath.kiev.ua}
\date{}

\begin{document}

\maketitle

\begin{abstract}
We consider a large class of matrix problems,
which includes the problem of classifying
arbitrary systems of linear mappings. For every
matrix problem from this class, we construct
Belitski\u\i's algorithm for reducing a matrix to
a canonical form, which is the generalization of
the Jordan normal form, and study the set
$C_{mn}$ of indecomposable canonical $m\times n$
matrices. Considering $C_{mn}$ as a subset in the
affine space of $m$-by-$n$ matrices, we prove
that either $C_{mn}$ consists of a finite number
of points and straight lines for every $m\times
n$, or $C_{mn}$ contains a 2-dimensional plane
for a certain $m\times n$.


 \end{abstract}

\def\newpic#1{%
 \def\emline##1##2##3##4##5##6{%
 \put(##1,##2){\special{em:point #1##3}}%
  \put(##4,##5){\special{em:point #1##6}}%
 \special{em:line #1##3,#1##6}}}
\newpic{}

\newcommand{\matr}[4]%
{\left[\genfrac{}{}{0pt}{}{#1}{#3}\,
\genfrac{}{}{0pt}{}{#2}{#4}\right]}

\newcommand{\im}{\mathop{\rm Im}\nolimits}
\newcommand{\Ker}{\mathop{\rm Ker}\nolimits}
\newcommand{\rad}{\mathop{\rm rad}\nolimits}
\newcommand{\Hom}{\mathop{\rm Hom}\nolimits}
\newcommand{\diag}{\mathop{\rm diag}\nolimits}
\newcommand{\rank}{\mathop{\rm rank}\nolimits}

\renewcommand{\le}{\leqslant}
\renewcommand{\ge}{\geqslant}

\newtheorem{theorem}{Theorem}[section]
\newtheorem{lemma}{Lemma}[section]
\newtheorem*{corollary}{Corollary}

\theoremstyle{remark}
\newtheorem{remark}{Remark}[section]
\newtheorem{example}{Example}[section]
\newtheorem{step}{Step}

\theoremstyle{definition}
\newtheorem{definition}{Definition}[section]

All matrices are considered over an algebraically
closed field $k$; $k^{m\times n}$ denotes the set
of $m$-by-$n$ matrices over $k$. The article
consists of three sections.

In Section \ref{sss1} we present Belitski\u{\i}'s
algorithm \cite{bel} (see also \cite{bel1}) in a
form, which is convenient for linear algebra. In
particular, the algorithm permits to reduce pairs
of $n$-by-$n$ matrices to a canonical form by
transformations of simultaneous similarity:
$(A,B)\mapsto (S^{-1}AS, S^{-1}BS)$; another
solution of this classical problem was given by
Friedland \cite{fri}. This section uses
rudimentary linear algebra (except for the proof
of Theorem \ref{r1.0}) and may be interested for
the general reader.

In Section \ref{sss2} we determine a broad class
of matrix problems, which includes the problems
of classifying representations of quivers,
partially ordered sets and finite dimensional
algebras. In Section \ref{sss3} we get the
following geometric characterization of the set
of canonical matrices in the spirit of
\cite{gab_vos}: if a matrix problem does not
`contain' the canonical form problem for pairs of
matrices under simultaneous similarity, then its
set of indecomposable canonical $m \times n$
matrices in the affine space $k^{m \times n}$
consists of a finite number of points and
straight lines (contrary to \cite{gab_vos}, these
lines are unpunched).

A detailed introduction is given at the beginning
of every section. Each introduction may be read
independently.


\section{Belitski\u{\i}'s algorithm}   \label{sss1}

\subsection{Introduction}     \label{sec1.1}

Every matrix problem is given by a set of
admissible transformations that determines an
equivalence relation on a certain set of matrices
(or sequences of matrices). The question is to
find a {\it canonical form}---i.e., determine a
`nice' set of canonical matrices such that each
equivalence class contains exactly one canonical
matrix. Two matrices are then equivalent if and
only if they have the same canonical form.

Many matrix problems can be formulated in terms
of quivers and their representations, introduced
by Gabriel \cite{gab} (see also \cite{gab_roi}).
A {\it quiver} is a directed graph, its {\it
representation} $A$ is given by assigning to each
vertex $i$ a finite dimensional vector space
$A_i$ over $k$ and to each arrow $\alpha : i \to
j$  a~linear mapping $A_{\alpha}: A_i \to A_j$.
For example, the diagonalization theorem, the
Jordan normal form, and the matrix pencil theorem
give the solution of the canonical form problem
for representations of the quivers, respectively,
$$
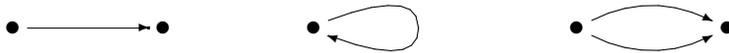
\special{em:linewidth 0.4pt} \unitlength 1.00mm
\linethickness{0.4pt}
\begin{picture}(95.00,15.50)
\put(0.00,10.00){\makebox(0,0)[cc]{$\bullet$}}
\put(20.00,10.00){\makebox(0,0)[cc]{$\bullet$}}
\put(18.00,10.00){\vector(1,0){0.2}}
\emline{2.00}{10.00}{1}{18.00}{10.00}{2}
\put(40.00,10.00){\makebox(0,0)[cc]{$\bullet$}}
\put(75.00,10.00){\makebox(0,0)[cc]{$\bullet$}}
\put(95.00,10.00){\makebox(0,0)[cc]{$\bullet$}}
\put(93.00,11.00){\vector(2,-1){0.2}}
\emline{77.00}{11.00}{3}{79.22}{11.96}{4}
\emline{79.22}{11.96}{5}{81.44}{12.60}{6}
\emline{81.44}{12.60}{7}{83.67}{12.94}{8}
\emline{83.67}{12.94}{9}{85.89}{12.98}{10}
\emline{85.89}{12.98}{11}{88.11}{12.70}{12}
\emline{88.11}{12.70}{13}{90.33}{12.11}{14}
\emline{90.33}{12.11}{15}{93.00}{11.00}{16}
\put(93.00,9.00){\vector(2,1){0.2}}
\emline{77.00}{9.00}{17}{79.22}{8.04}{18}
\emline{79.22}{8.04}{19}{81.44}{7.40}{20}
\emline{81.44}{7.40}{21}{83.67}{7.06}{22}
\emline{83.67}{7.06}{23}{85.89}{7.02}{24}
\emline{85.89}{7.02}{25}{88.11}{7.30}{26}
\emline{88.11}{7.30}{27}{90.33}{7.89}{28}
\emline{90.33}{7.89}{29}{93.00}{9.00}{30}
\emline{42.00}{11.00}{31}{45.10}{12.06}{32}
\emline{45.10}{12.06}{33}{47.74}{12.73}{34}
\emline{47.74}{12.73}{35}{49.92}{13.01}{36}
\emline{49.92}{13.01}{37}{51.63}{12.91}{38}
\emline{51.63}{12.91}{39}{52.88}{12.43}{40}
\emline{52.88}{12.43}{41}{53.67}{11.56}{42}
\emline{53.67}{11.56}{43}{54.00}{10.00}{44}
\put(42.00,9.00){\vector(-3,1){0.2}}
\emline{54.00}{10.00}{45}{53.77}{8.67}{46}
\emline{53.77}{8.67}{47}{53.07}{7.72}{48}
\emline{53.07}{7.72}{49}{51.92}{7.15}{50}
\emline{51.92}{7.15}{51}{50.30}{6.98}{52}
\emline{50.30}{6.98}{53}{48.21}{7.18}{54}
\emline{48.21}{7.18}{55}{45.67}{7.78}{56}
\emline{45.67}{7.78}{57}{42.00}{9.00}{58}
\end{picture}
$$\\[-11mm]
(Analogously, one may study systems of forms
and linear mappings as representations of a
partially directed graph $G$, assigning a
bilinear form to an undirected edge. As was
proved in \cite{roi1,ser2}, the problem of
classifying representations of $G$ is reduced to
the problem of classifying representations of a
certain quiver $\bar G$. The class of studied
matrix problems may be extended by considering
quivers with relations \cite{gab_roi,rin} and
partially directed graphs with relations
\cite{ser2}.)

The canonical form problem was solved only for
the quivers of so called tame type by Donovan and
Freislich \cite{don1} and Nazarova \cite{naz},
this problem is considered as hopeless for the
other quivers (see Section \ref{sss2}).
Nevertheless, the matrices of each individual
representation of a quiver may be reduced to a
canonical form by Belitski\u{\i}'s algorithm (see
\cite{bel} and its extended version \cite{bel1}).
This algorithm and the better known Littlewood
algorithm \cite{lit} (see also \cite{ser,sha})
for reducing matrices to canonical form under
unitary similarity have the same conceptual
sketch: The matrix is partitioned and successive
admissible transformations are applied to reduce
the submatrices to some nice form. At each stage,
one refines the partition and restricts the set
of permissible transformations to those that
preserve the already reduced blocks. The process
ends in a finite number of steps, producing the
canonical form.

We will apply Belitski\u{\i}'s algorithm to the
canonical form problem for matrices under
$\Lambda$-similarity, which is defined as
follows. Let $\Lambda$ be an algebra of $n \times
n$ matrices (i.e., a subspace of $k^{n \times n}$
that is closed with respect to multiplication and
contains the identity matrix $I$) and let
$\Lambda ^*$ be the set of its nonsingular
matrices. We say that two ${n \times n}$ matrices
$M$ and $N$ are $\Lambda$-{\it similar} and write
$M \sim_{\Lambda} N$ if there exists $S \in
\Lambda ^*$ such that $S^{-1}MS = N$
($\sim_{\Lambda}$ is an equivalence relation; see
the end of Section \ref{sec2}).

\begin{example}                         \label{e0.1}
The problem of classifying representations of
each quiver can be formulated in terms of
$\Lambda$-similarity, where $\Lambda$ is an
algebra of block-diagonal matrices in which some
of the diagonal blocks are required to be equal.
For instance, the problem of classifying
representations of the quiver
\begin{equation}       \label{0.01}
\special{em:linewidth 0.4pt} \unitlength 0.40mm
\linethickness{0.4pt}
\begin{picture}(139.00,22.00)
(0,17.3)
\emline{116.00}{11.00}{1}{119.51}{11.83}{2}
\emline{119.51}{11.83}{3}{122.64}{12.43}{4}
\emline{122.64}{12.43}{5}{125.39}{12.80}{6}
\emline{125.39}{12.80}{7}{127.76}{12.94}{8}
\emline{127.76}{12.94}{9}{129.75}{12.85}{10}
\emline{129.75}{12.85}{11}{131.37}{12.54}{12}
\emline{131.37}{12.54}{13}{132.61}{12.00}{14}
\emline{132.61}{12.00}{15}{133.47}{11.23}{16}
\emline{133.47}{11.23}{17}{134.00}{10.00}{18}
\put(116.00,9.00){\vector(-4,1){0.2}}
\emline{134.00}{10.00}{19}{133.59}{8.89}{20}
\emline{133.59}{8.89}{21}{132.81}{8.02}{22}
\emline{132.81}{8.02}{23}{131.65}{7.40}{24}
\emline{131.65}{7.40}{25}{130.11}{7.02}{26}
\emline{130.11}{7.02}{27}{128.19}{6.89}{28}
\emline{128.19}{6.89}{29}{125.89}{7.00}{30}
\emline{125.89}{7.00}{31}{123.22}{7.36}{32}
\emline{123.22}{7.36}{33}{120.16}{7.96}{34}
\emline{120.16}{7.96}{35}{116.00}{9.00}{36}
\emline{24.00}{11.00}{37}{20.49}{11.83}{38}
\emline{20.49}{11.83}{39}{17.36}{12.43}{40}
\emline{17.36}{12.43}{41}{14.61}{12.80}{42}
\emline{14.61}{12.80}{43}{12.24}{12.94}{44}
\emline{12.24}{12.94}{45}{10.25}{12.85}{46}
\emline{10.25}{12.85}{47}{8.63}{12.54}{48}
\emline{8.63}{12.54}{49}{7.39}{12.00}{50}
\emline{7.39}{12.00}{51}{6.53}{11.23}{52}
\emline{6.53}{11.23}{53}{6.00}{10.00}{54}
\put(24.00,9.00){\vector(4,1){0.2}}
\emline{6.00}{10.00}{55}{6.41}{8.89}{56}
\emline{6.41}{8.89}{57}{7.19}{8.02}{58}
\emline{7.19}{8.02}{59}{8.35}{7.40}{60}
\emline{8.35}{7.40}{61}{9.89}{7.02}{62}
\emline{9.89}{7.02}{63}{11.81}{6.89}{64}
\emline{11.81}{6.89}{65}{14.11}{7.00}{66}
\emline{14.11}{7.00}{67}{16.78}{7.36}{68}
\emline{16.78}{7.36}{69}{19.84}{7.96}{70}
\emline{19.84}{7.96}{71}{24.00}{9.00}{72}
\put(28.00,10.00){\makebox(0,0)[cc]{1}}
\put(112.00,10.00){\makebox(0,0)[cc]{3}}
\put(108.00,7.00){\vector(1,0){0.2}}
\emline{32.00}{7.00}{79}{108.00}{7.00}{80}
\put(108.00,10.00){\vector(1,0){0.2}}
\emline{32.00}{10.00}{73}{108.00}{10.00}{74}
\put(70.00,34.00){\makebox(0,0)[cc]{2}}
\put(1.00,10.00){\makebox(0,0)[cc]{$\alpha$}}
\put(139.00,10.00){\makebox(0,0)[cc]{$ \zeta$}}
\put(70.00,15.00){\makebox(0,0)[cc]{$\gamma$}}
\put(48.00,25.00){\makebox(0,0)[cc]{$\beta$}}
\put(91.00,25.00){\makebox(0,0)[cc]{$\varepsilon
$}}
\put(66.00,29.00){\vector(2,1){0.2}}
\emline{32.00}{12.00}{75}{66.00}{29.00}{76}
\put(108.00,12.00){\vector(2,-1){0.2}}
\emline{74.00}{29.00}{77}{108.00}{12.00}{78}
\put(70.00,3.00){\makebox(0,0)[cc]{$ \delta$}}
\end{picture}
\end{equation}\\
is the canonical form problem for matrices of the
form $$
\begin{bmatrix}
A_{\alpha}&0&0&0\\ A_{\beta}&0&0&0  \\
A_{\gamma}&0&0&0  \\ A_{\delta}& A_{\varepsilon}&
A_{\zeta}&0
\end{bmatrix}$$
under $\Lambda$-similarity, where
$\Lambda$ consists of
block-diagonal matrices of the form $S_1 \oplus
S_2 \oplus S_3\oplus S_3$.
\end{example}

\begin{example}                         \label{e0.1a}
By the definition of Gabriel and Roiter
\cite{gab_roi}, a linear matrix problem of size
$m\times n$ is given by a pair $(D^*,\cal{M})$,
where $D$ is a subalgebra of $k^{m\times m}\times
k^{n\times n}$ and $\cal M$ is a subset of
$k^{m\times n}$ such that $SAR^{-1}\in\cal M$
whenever $A\in\cal M$ and $(S,R)\in { D}^*$. The
question is to classify the orbits of $\cal M$
under the action $(S,R):A\mapsto SAR^{-1}$.
Clearly, two $m\times n$ matrices $A$ and $B$
belong to the same orbit if and only if
$\matr{0}{A}{0}{0}$ and $\matr{0}{B}{0}{0}$ are
$\Lambda$-similar, where $\Lambda:=\{S\oplus
R\,|\, (S,R)\in{ D}\}$ is an algebra of
$(m+n)\times (m+n)$ matrices.
\end{example}

In Section \ref{sec2} we prove that for every
algebra $\Lambda \subset k^{n \times n}$ there
exists a nonsingular matrix $P$ such that the
algebra $P^{-1}\Lambda
P:=\{P^{-1}AP\,|\,A\in\Lambda\}$ consists of
upper block-triangular matrices, in which some of
the diagonal blocks must be equal and
off-diagonal blocks satisfy a system of linear
equations. The algebra $P^{-1}\Lambda P$ will be
called a {\it reduced matrix algebra.} The
$\Lambda$-similarity transformations with a
matrix $M$ correspond to the $P^{-1} \Lambda
P$-similarity transformations with the matrix
$P^{-1}MP$ and hence it suffices to study
$\Lambda$-similarity transformations given by a
reduced matrix algebra $\Lambda$.

In Section \ref{ss2}, for every Jordan matrix $J$
we construct a matrix $J^{\#}=P^{-1}JP$ ($P$ is a
permutation matrix) such that all matrices
commuting with it form a reduced algebra.
Following Shapiro \cite{sha1}, we call $J^{\#}$ a
{\it Weyr matrix} since its form is determined by
the set of its Weyr characteristics
(Belitski\u{\i} \cite{bel} calls $J^{\#}$ a
modified Jordan matrix; it plays a central role
in his algorithm).

In Section \ref{sec3} we construct an algorithm
(which is a modification of Belitski\u\i's
algorithm \cite{bel}, \cite{bel1}) for reducing
matrices to canonical form under
$\Lambda$-similarity with a reduced matrix
algebra $\Lambda$. In Section \ref{sec4} we study
the construction of the set of canonical
matrices.


\subsection{Reduced matrix algebras}   \label{sec2}

In this section we prove that for every matrix
algebra $\Lambda\subset k^{n \times n}$ there
exists a nonsingular matrix $P$ such that the
algebra $P^{-1}\Lambda P$ is a reduced matrix
algebra in the sense of the following definition.

A block matrix $M=[M_{ij}]$, $M_{ij}\in
k^{m_i\times n_j}$, will be called an
$\underline{m}\times\underline{n}$ {\it matrix},
where $\underline{m}=(m_1,m_2,\ldots),$
$\underline{n}=(n_1,n_2,\ldots)$ and $m_i,n_j\in
\{0,1,2,\ldots\}$ (we take into consideration
blocks without rows or columns).

\begin{definition}    \label{d1.1}
An algebra $\Lambda$ of $\underline{n}
\times\underline{n}$ matrices,
$\underline{n}=(n_1,\dots,n_t)$, will be called a
{\it reduced} $\underline{n}\times\underline{n}$
{\it algebra} if there exist
\begin{itemize}
\item [(a)] an equivalence relation
\begin{equation}  \label{0}
\sim \ \ \text{in}\ \ T=\{1,\dots,t\},
\end{equation}

\item[(b)] a family of systems of linear equations
\begin{equation}       \label{00}
\Bigl\{\sum_{{\cal I} \ni i < j \in {\cal J}}
c_{ij}^{(l)}x_{ij} = 0, \quad 1 \le l \le
q_{_{\cal IJ}} \Bigr\}_{{\cal I,J}\in
T/\!\sim}\,,
\end{equation}
indexed by pairs of equivalence classes, where
$c_{ij}^{(l)} \in k$ and $q_{_{\cal IJ}} \geq 0$,
\end{itemize}
such that $\Lambda$ consists of all upper
block-triangular
$\underline{n}\times\underline{n}$ matrices
\begin{equation}         \label{1}
   S=\begin{bmatrix}
S_{11}&S_{12}&\cdots& S_{1t}\\
&S_{22}&\ddots&\vdots\\ &&\ddots&S_{t-1,t}\\
 \mbox{\LARGE 0} &&& S_{tt}  \end{bmatrix},\quad
   S_{ij}\in k^{n_i\times n_j},
\end{equation}
in which diagonal blocks satisfy the condition
\begin{equation}         \label{2}
S_{ii} = S_{jj} \quad {\rm whenever}\quad i \sim
j\,,
\end{equation}
 and off-diagonal blocks satisfy the equalities
 \begin{equation}         \label{3}
\sum_{{\cal I} \ni i < j \in {\cal J}}
c_{ij}^{(l)} S_{ij} = 0, \quad  1 \le l \le
q_{_{\cal IJ}}\,,
\end{equation}
for each pair ${\cal I,J}\in T/\!\sim$\,.
\end{definition}

Clearly, the sequence
$\underline{n}=(n_1,\dots,n_t)$ and the
equivalence relation $\sim$ are uniquely
determined by $\Lambda$; moreover, $n_i= n_j$ if
$i\sim j$.

\begin{example}                   \label{e1.2}
Let us consider the classical canonical form
problem for pairs of matrices $(A,B)$ under
simultaneous  similarity (i.e., for
representations of the quiver \!\!
\unitlength 0.5mm \linethickness{0.4pt}
\begin{picture}(15,0)(-7,-1.5)
\put(0.00,0.00){\circle*{0.8}}
\put(-5.53,0.00){\oval(4.00,4.00)[l]}
\bezier{16}(-5.47,1.93)(-3.67,2.00)(-1.27,0.93)
\put(-1.20,-0.67){\vector(2,1){0.2}}
\bezier{20}(-5.47,-2.00)(-3.20,-1.87)(-1.20,-0.67)
\put(5.53,0.00){\oval(4.00,4.00)[r]}
\bezier{16}(5.47,1.93)(3.67,2.00)(1.27,0.93)
\put(1.20,-0.67){\vector(-2,1){0.2}}
\bezier{20}(5.47,-2.00)(3.20,-1.87)(1.20,-0.67)
\end{picture}
\!). Reducing $(A,B)$ to the form $(J,C)$, where
$J$ is a Jordan matrix, and restricting the set
of permissible transformations to those that
preserve $J$, we obtain the canonical form
problem for $C$ under $\Lambda$-similarity, where
$\Lambda$ consists of all matrices commuting with
$J$. In the next section, we modify $J$ such that
$\Lambda$ becomes a reduced matrix algebra.
\end{example}

\begin{theorem}                \label{r1.0}
For every matrix algebra $\Lambda\subset k^{n
\times n}$, there exists a nonsingular matrix $P$
such that $P^{-1}\Lambda P$ is a reduced matrix
algebra.
\end{theorem}

\begin{proof}
Let $V$ be a vector space over $k$ and $\Lambda
\subset {\rm End}_k(V)$ be an algebra of linear
operators. We prove briefly that their matrices
in a certain basis of $V$ form a reduced algebra
(this fact is used only in Section \ref{s3.4};
the reader may omit the proof if he is not
familiar with the theory of algebras).

Let $R$ be the radical of $\Lambda$. By the
Wedderburn-Malcev theorem \cite{dr_ki}, there
exists a subalgebra $\bar{\Lambda}
\subset\Lambda$ such that $\bar{\Lambda}
\simeq\Lambda/R$ and $\bar{\Lambda}\cap R=0$. By
the Wedderburn-Artin theorem \cite{dr_ki},
$\bar{\Lambda}\simeq k^{m_1\times m_1}\times
\dots \times k^{m_q\times m_q}$. We denote by
$e_{ij}^{(\alpha)} \in \bar{\Lambda}$ $(i,j\in \{
1,\dots, m_{\alpha}\},\ 1\le\alpha\le q)$ the
elements of $\Lambda$ that correspond to the
matrix  units of $k^{m_{\alpha}\times
m_{\alpha}}$. Put $e_{\alpha}=e_{11}^{(\alpha)},$
$e=e_1+\dots+e_q,$ and $V_0=eV.$

We consider $\Lambda_0:=e\Lambda e$ as a
subalgebra of ${\rm End}_k(V_0)$, its radical is
$R_0:=R\cap\Lambda_0$ and $\Lambda_0/ R_0\simeq
k\times \dots\times k.$ Let $R_0^{m-1}\neq
0=R_0^m.$ We choose a basis of $R_0^{m-1}V_0$
formed by vectors $v_1,\dots,v_{t_1} \in
\bigcup_{\alpha}e_{\alpha}V_0$, complete it to a
basis of $R_0^{m-2}V_0$  by vectors
$v_{t_1+1},\dots,v_{t_2} \in
\bigcup_{\alpha}e_{\alpha}V_0$, and so on, until
we obtain a basis  $v_1,\dots,v_{t_m}$ of $V_0$.
All its vectors have the form
$v_i=e_{\alpha_i}v_i$; put ${\cal I}_{\alpha}=
\{i\,|\,\alpha_i= \alpha \}$ for $1\le \alpha\le
q$.

Since $e_{\alpha}e_{\beta}=0$ if
$\alpha\neq\beta$, $e_{\alpha}^2= e_{\alpha}$,
and $e$ is the unit of $\Lambda_0$, the vector
space of $\Lambda_0$ is the  direct sum of all
$e_{\alpha}\Lambda_0 e_{\beta}.$ Moreover,
$e_{\alpha}\Lambda_0 e_{\beta}=e_{\alpha}R_0
e_{\beta}$ for $\alpha\neq\beta$ and
$e_{\alpha}\Lambda_0 e_{\alpha}=
ke_{\alpha}\oplus e_{\alpha}R_0e_{\alpha},$ hence
$\Lambda_0= (\bigoplus_{\alpha}ke_{\alpha})
\oplus (\bigoplus_{\alpha,\beta}e_{\alpha}R_0e_{
\beta}).$ The matrix of every linear operator
from $ e_{\alpha}R_0e_{ \beta}$ in the basis
$v_1,\dots,v_{t_m}$ has the form
$[a_{ij}]_{i,j=1}^ {t_m}$, where $a_{ij}\ne 0$
implies $i<j$ and $(i,j)\in{\cal
I}_{\alpha}\times {\cal I}_{\beta}$. Therefore,
the set of matrices $[a_{ij}]$ of linear
operators from $\Lambda_0$ in the basis
$v_1,\dots,v_{t_m}$ may be given by a system of
linear equations of the form
 $$
a_{ij}=0 \ (i>j), \quad a_{ii}=a_{jj} \
(\{i,j\}\subset {\cal I}_{\alpha}), \quad
\sum_{{\cal I}_{\alpha} \ni i < j \in {\cal
I}_{\beta}} c_{ij}^{(l)} a_{ij} = 0 \  ( 1 \le l
\le q_{\alpha\beta}).
 $$

The matrices of linear operators from $\Lambda$
in the basis $e_{11}^{(\alpha_1)}v_1, \dots,
e_{m_{\alpha_1} 1}^{(\alpha_1)}v_1,
$\linebreak[0]$ \ e_{11}^{(\alpha_2)}v_2,\dots,$
$ e_{m_{\alpha_2} 1}^{(\alpha_2)}v_2,\dots $ of
$V$ have the form (\ref{1}) and are given by the
system of relations (\ref{2})--(\ref{3}). Hence
their set is a reduced matrix algebra.
\end{proof}

For every matrix algebra $\Lambda \subset
k^{n\times n}$, the set $\Lambda^*$ of its
nonsingular matrices is a group and hence the
$\Lambda$-similarity is an equivalence relation.
Indeed, we may assume that $\Lambda$ is a reduced
matrix algebra. Then every $S\in\Lambda^*$ can be
written in the form $D(I-C)$, where
$D,\,C\in\Lambda$ such that $D$ is a
block-diagonal and all diagonal blocks of  $C$
are zero. Since $C$ is nilpotent,
$S^{-1}=(I+C+C^2+\cdots)D^{-1}\in\Lambda^*$.

Note also that every finite dimensional algebra
is isomorphic to a matrix algebra and hence, by
Theorem \ref{r1.0}, it is isomorphic to a reduced
matrix algebra.

\subsection{Weyr matrices}     \label{ss2}

Following Belitski\u{\i} \cite{bel}, for every
Jordan matrix $J$ we define a matrix
$J^{\#}=P^{-1}JP$ ($P$ is a permutation matrix)
such that all matrices commuting with it form a
reduced algebra. We will fix a linear order
$\prec$ in $k$ (if $k$ is the field of complex
numbers, we may use the lexicographic ordering:
$a+bi\prec c+di$ if either $a=c$ and $b<d$, or
$a<c$).

\begin{definition}    \label{d2'.1}
A {\it Weyr matrix} is a matrix of the form
\begin{equation}       \label{2'.1}
W=W_{\{\lambda_1\}}\oplus\dots\oplus
W_{\{\lambda_r\}},\quad \lambda_1\prec\dots\prec
\lambda_r,
\end{equation}
where
\begin{equation*}       \label{2'.5'}
W_{\{\lambda_i\}}= \left[\begin{tabular}{cccc}
$\lambda_iI_{m_{i1}}$&$W_{i1}$&&{\Large 0}\\
&$\lambda_iI_{m_{i2}}$&$\ddots$&\\
&&$\ddots$&$W_{i,k_i-1}$\\ {\Large
0}&&&$\lambda_iI_{m_{ik_i}}$
\end{tabular}\right],\quad
W_{ij}=
\begin{bmatrix}
I\\0
\end{bmatrix},
\end{equation*}
$m_{i1}\ge\dots\ge m_{ik_i}$. The {\it standard
partition} of $W$ is the ${\underline n}\times
{\underline n}$ partition, where $\underline
n=({\underline n}_1,\dots, {\underline n}_r)$ and
${\underline n}_i$ is the sequence
$m_{i1}-m_{i2}, m_{i2}-m_{i3},\dots,
m_{i,k_i-1}-m_{ik_i}, m_{ik_i};$
$m_{i2}-m_{i3},\dots, m_{i,k_i-1}-m_{ik_i},
m_{ik_i};\dots; m_{i,k_i-1}-m_{ik_i}, m_{ik_i};
m_{ik_i}$ from which all zero components are
removed.
\end{definition}

The standard partition of $W$ is the most coarse
partition for which all diagonal blocks have the
form $\lambda_i I$ and all off-diagonal blocks
have the form $0$ or $I$.

The matrix $W$ is named a `Weyr matrix' since
$(m_{i1}, m_{i2},\dots,m_{ik_i})$ is the Weyr
characteristic of $W$ (and of every matrix that
is similar to $W$) for $\lambda_i$. Recall (see
\cite{sha}, \cite{sha1}, \cite{wey}) that the
{\it Weyr characteristic} of a square matrix $A$
for an eigenvalue $\lambda$ is the decreasing
list $(m_1,m_2,\ldots)$, where $m_i:=
{\rank}(A-\lambda I)^{i-1}- {\rank}(A-\lambda
I)^i$. Clearly, $m_i$ is the number of Jordan
cells $J_l(\lambda ),\ l\ge i$, in the Jordan
form of $A$ (i.e., $m_i-m_{i+1}$ is the number of
$J_i(\lambda )$), so the Jordan form is uniquely,
up to permutation of Jordan cells, determined by
the set of eigenvalues of $A$ and their Weyr
characteristics. Taking into account the
inequality at the right-hand side of
\eqref{2'.1}, we get the first statement of the
following theorem:

\begin{theorem}      \label{t2'.1}
Every square matrix $A$ is similar to exactly one
Weyr matrix $A^{\#}$. The matrix $A^{\#}$ is
obtained from the Jordan form of $A$ by
simultaneous permutations of its rows and
columns. All matrices commuting with $A^{\#}$
form a reduced matrix algebra $\Lambda(A^{\#})$
of ${\underline n}\times {\underline n}$ matrices
(\ref{1}) with equalities (\ref{3}) of the form
$S_{ij}=S_{i'j'}$ and $S_{ij}=0$, where
${\underline n}\times {\underline n}$ is the
standard partition of $A^{\#}$.
\end{theorem}

To make the proof of the second and the third
statements clearer, we begin with an example.

\begin{example}          \label{e2'.1}
Let us construct the Weyr form
$J_{\{\lambda\}}^{\#}$ of the Jordan matrix
$$
J_{\{\lambda\}}:= \underbrace{J_4(\lambda)\oplus
\dots \oplus J_4(\lambda)}_{\mbox{$p$
times}}\oplus \underbrace{J_2(\lambda)\oplus
\dots \oplus J_2(\lambda)}_{\mbox{$q$ times}}
$$
with a single eigenvalue $\lambda$. Gathering
Jordan cells of the same size, we first reduce
$J_{\{\lambda\}}$ to $J_{\{\lambda\}}^{\text{\it
+}}= J_4(\lambda I_p) \oplus J_2(\lambda I_q)$.
The matrix $J_{\{\lambda\}}^{\text{\it +}}$ and
all matrices commuting with it have the form,
respectively, $$
\special{em:linewidth 0.4pt} \unitlength 0.48mm
\linethickness{0.4pt}
\begin{picture}(280.00,65.00)
{\linethickness{1.2pt}
\put(0.00,0.00){\framebox(120.00,60.00)[cc]{}} }
\emline{80.00}{60.00}{1}{80.00}{0.00}{2}
\emline{120.00}{20.00}{3}{0.00}{20.00}{4}
\put(10.00,55.00){\makebox(0,0)[cc]{$\lambda
I_p$}}
\put(30.00,45.00){\makebox(0,0)[cc]{$\lambda
I_p$}}
\put(50.00,35.00){\makebox(0,0)[cc]{$\lambda
I_p$}}
\put(70.00,25.00){\makebox(0,0)[cc]{$\lambda
I_p$}}
\put(90.00,15.00){\makebox(0,0)[cc]{$\lambda
I_q$}}
\put(110.00,5.00){\makebox(0,0)[cc]{$\lambda
I_q$}}
\put(100.00,40.00){\makebox(0,0)[cc]{\Large{0}}}
\put(40.00,10.00){\makebox(0,0)[cc] {\Large{0}}}
\put(10.00,65.00){\makebox(0,0)[cc]
{$\scriptstyle (11)$}}
\put(30.00,65.00){\makebox(0,0)[cc]
{$\scriptstyle (21)$}}
\put(50.00,65.00){\makebox(0,0)[cc]
{$\scriptstyle (31)$}}
\put(70.00,65.00){\makebox(0,0)[cc]
{$\scriptstyle (41)$}}
\put(90.00,65.00){\makebox(0,0)[cc]
{$\scriptstyle (12)$}}
\put(110.00,65.00){\makebox(0,0)[cc]
{$\scriptstyle (22)$}}
\put(127.00,55.00){\makebox(0,0)[cc]
{$\scriptstyle (11)$}}
\put(127.00,45.00){\makebox(0,0)[cc]
{$\scriptstyle (21)$}}
\put(127.00,35.00){\makebox(0,0)[cc]
{$\scriptstyle (31)$}}
\put(127.00,25.00){\makebox(0,0)[cc]
{$\scriptstyle (41)$}}
\put(127.00,15.00){\makebox(0,0)[cc]
{$\scriptstyle (12)$}}
\put(127.00,5.00){\makebox(0,0)[cc]
{$\scriptstyle (22)$}} {\linethickness{1.2pt}
\put(150.00,0.00){\framebox(120.00,60.00)[cc]{}}
} \emline{270.00}{20.00}{5}{150.00}{20.00}{6}
\emline{230.00}{0.00}{7}{230.00}{60.00}{8}
\put(160.00,55.00){\makebox(0,0)[cc]{$A_1$}}
\put(180.00,45.00){\makebox(0,0)[cc]{$A_1$}}
\put(200.00,35.00){\makebox(0,0)[cc]{$A_1$}}
\put(220.00,25.00){\makebox(0,0)[cc]{$A_1$}}
\put(240.00,15.00){\makebox(0,0)[cc]{$D_1$}}
\put(260.00,5.00){\makebox(0,0)[cc]{$D_1$}}
\put(160.00,65.00){\makebox(0,0)[cc]
{$\scriptstyle (11)$}}
\put(180.00,65.00){\makebox(0,0)[cc]
{$\scriptstyle (21)$}}
\put(200.00,65.00){\makebox(0,0)[cc]
{$\scriptstyle (31)$}}
\put(220.00,65.00){\makebox(0,0)[cc]
{$\scriptstyle (41)$}}
\put(240.00,65.00){\makebox(0,0)[cc]
{$\scriptstyle (12)$}}
\put(260.00,65.00){\makebox(0,0)[cc]
{$\scriptstyle (22)$}}
\put(277.00,55.00){\makebox(0,0)[cc]
{$\scriptstyle (11)$}}
\put(277.00,45.00){\makebox(0,0)[cc]
{$\scriptstyle (21)$}}
\put(277.00,35.00){\makebox(0,0)[cc]
{$\scriptstyle (31)$}}
\put(277.00,25.00){\makebox(0,0)[cc]
{$\scriptstyle (41)$}}
\put(277.00,15.00){\makebox(0,0)[cc]
{$\scriptstyle (12)$}}
\put(277.00,5.00){\makebox(0,0)[cc]
{$\scriptstyle (22)$}}
\put(30.00,55.00){\makebox(0,0)[cc]{$I_p$}}
\put(50.00,45.00){\makebox(0,0)[cc]{$I_p$}}
\put(70.00,35.00){\makebox(0,0)[cc]{$I_p$}}
\put(110.00,15.00){\makebox(0,0)[cc]{$I_q$}}
\put(180.00,55.00){\makebox(0,0)[cc]{$A_2$}}
\put(200.00,55.00){\makebox(0,0)[cc]{$A_3$}}
\put(220.00,55.00){\makebox(0,0)[cc]{$A_4$}}
\put(220.00,45.00){\makebox(0,0)[cc]{$A_3$}}
\put(200.00,45.00){\makebox(0,0)[cc]{$A_2$}}
\put(220.00,35.00){\makebox(0,0)[cc]{$A_2$}}
\put(240.00,55.00){\makebox(0,0)[cc]{$B_1$}}
\put(260.00,55.00){\makebox(0,0)[cc]{$B_2$}}
\put(260.00,45.00){\makebox(0,0)[cc]{$B_1$}}
\put(260.00,15.00){\makebox(0,0)[cc]{$D_2$}}
\put(200.00,15.00){\makebox(0,0)[cc]{$C_1$}}
\put(220.00,15.00){\makebox(0,0)[cc]{$C_2$}}
\put(220.00,5.00){\makebox(0,0)[cc]{$C_1$}}
\end{picture}$$
Simultaneously permuting strips in these
matrices, we get the Weyr matrix
$J_{\{\lambda\}}^{\#}$ and all matrices commuting
with it (they form a reduced ${\underline
n}\times {\underline n}$ algebra
$\Lambda(J_{\{\lambda\}}^{\#})$ with equalities
(\ref{3}) of the form $S_{ij}=S_{i'j'}$,
$S_{ij}=0$, and with ${\underline
n}=(p,q,p,q,p,p)$): $$
\special{em:linewidth 0.4pt} \unitlength 0.48mm
\linethickness{0.4pt}
\begin{picture}(280.00,65.00)
{\linethickness{1.2pt}
\put(0.00,0.00){\framebox(120.00,60.00)[cc]{}} }
\put(10.00,55.00){\makebox(0,0)[cc]{$\lambda
I_p$}}
\put(30.00,45.00){\makebox(0,0)[cc]{$\lambda
I_q$}}
\put(50.00,35.00){\makebox(0,0)[cc]{$\lambda
I_p$}}
\put(70.00,25.00){\makebox(0,0)[cc]{$\lambda
I_q$}}
\put(90.00,15.00){\makebox(0,0)[cc]{$\lambda
I_p$}}
\put(110.00,5.00){\makebox(0,0)[cc]{$\lambda
I_p$}} \put(10.00,65.00){\makebox(0,0)[cc]
{$\scriptstyle (11)$}}
\put(30.00,65.00){\makebox(0,0)[cc]
{$\scriptstyle (12)$}}
\put(50.00,65.00){\makebox(0,0)[cc]
{$\scriptstyle (21)$}}
\put(70.00,65.00){\makebox(0,0)[cc]
{$\scriptstyle (22)$}}
\put(90.00,65.00){\makebox(0,0)[cc]
{$\scriptstyle (31)$}}
\put(110.00,65.00){\makebox(0,0)[cc]
{$\scriptstyle (41)$}}
\put(127.00,55.00){\makebox(0,0)[cc]
{$\scriptstyle (11)$}}
\put(127.00,45.00){\makebox(0,0)[cc]
{$\scriptstyle (12)$}}
\put(127.00,35.00){\makebox(0,0)[cc]
{$\scriptstyle (21)$}}
\put(127.00,25.00){\makebox(0,0)[cc]
{$\scriptstyle (22)$}}
\put(127.00,15.00){\makebox(0,0)[cc]
{$\scriptstyle (31)$}}
\put(127.00,5.00){\makebox(0,0)[cc]
{$\scriptstyle (41)$}}
{\linethickness{1.2pt}
\put(150.00,0.00){\framebox(120.00,60.00)[cc]{}}
} \put(160.00,55.00){\makebox(0,0)[cc]{$A_1$}}
\put(180.00,45.00){\makebox(0,0)[cc]{$D_1$}}
\put(200.00,35.00){\makebox(0,0)[cc]{$A_1$}}
\put(220.00,25.00){\makebox(0,0)[cc]{$D_1$}}
\put(240.00,15.00){\makebox(0,0)[cc]{$A_1$}}
\put(260.00,5.00){\makebox(0,0)[cc]{$A_1$}}
\put(160.00,65.00){\makebox(0,0)[cc]
{$\scriptstyle (11)$}}
\put(180.00,65.00){\makebox(0,0)[cc]
{$\scriptstyle (12)$}}
\put(200.00,65.00){\makebox(0,0)[cc]
{$\scriptstyle (21)$}}
\put(220.00,65.00){\makebox(0,0)[cc]
{$\scriptstyle (22)$}}
\put(240.00,65.00){\makebox(0,0)[cc]
{$\scriptstyle (31)$}}
\put(260.00,65.00){\makebox(0,0)[cc]
{$\scriptstyle (41)$}}
\put(277.00,55.00){\makebox(0,0)[cc]
{$\scriptstyle (11)$}}
\put(277.00,45.00){\makebox(0,0)[cc]
{$\scriptstyle (12)$}}
\put(277.00,35.00){\makebox(0,0)[cc]
{$\scriptstyle (21)$}}
\put(277.00,25.00){\makebox(0,0)[cc]
{$\scriptstyle (22)$}}
\put(277.00,15.00){\makebox(0,0)[cc]
{$\scriptstyle (31)$}}
\put(277.00,5.00){\makebox(0,0)[cc]
{$\scriptstyle (41)$}}
\emline{0.00}{40.00}{81}{120.00}{40.00}{82}
\emline{120.00}{20.00}{83}{0.00}{20.00}{84}
\emline{0.00}{10.00}{85}{120.00}{10.00}{86}
\emline{100.00}{0.00}{87}{100.00}{60.00}{88}
\emline{80.00}{60.00}{9}{80.00}{0.00}{10}
\emline{40.00}{0.00}{11}{40.00}{60.00}{12}
\emline{150.00}{40.00}{13}{270.00}{40.00}{14}
\emline{270.00}{20.00}{15}{150.00}{20.00}{16}
\emline{150.00}{10.00}{17}{270.00}{10.00}{18}
\emline{250.00}{0.00}{19}{250.00}{60.00}{20}
\emline{230.00}{60.00}{21}{230.00}{0.00}{22}
\emline{190.00}{0.00}{23}{190.00}{60.00}{24}
\put(50.00,55.00){\makebox(0,0)[cc]{$I_p$}}
\put(70.00,45.00){\makebox(0,0)[cc]{$I_q$}}
\put(90.00,35.00){\makebox(0,0)[cc]{$I_p$}}
\put(110.00,15.00){\makebox(0,0)[cc]{$I_p$}}
\put(180.00,55.00){\makebox(0,0)[cc]{$B_1$}}
\put(200.00,55.00){\makebox(0,0)[cc]{$A_2$}}
\put(220.00,55.00){\makebox(0,0)[cc]{$B_2$}}
\put(220.00,45.00){\makebox(0,0)[cc]{$D_2$}}
\put(240.00,55.00){\makebox(0,0)[cc]{$A_3$}}
\put(240.00,45.00){\makebox(0,0)[cc]{$C_1$}}
\put(260.00,55.00){\makebox(0,0)[cc]{$A_4$}}
\put(260.00,45.00){\makebox(0,0)[cc]{$C_2$}}
\put(240.00,35.00){\makebox(0,0)[cc]{$A_2$}}
\put(260.00,35.00){\makebox(0,0)[cc]{$A_3$}}
\put(260.00,25.00){\makebox(0,0)[cc]{$C_1$}}
\put(260.00,15.00){\makebox(0,0)[cc]{$A_2$}}
\put(220.00,35.00){\makebox(0,0)[cc]{$B_1$}}
\end{picture}
$$
\end{example}

\begin{proof}[Proof of Theorem \ref{t2'.1}]
We may suppose that $A$ is a Jordan matrix
\begin{equation*}
J=J_{\{\lambda_1\}}\oplus\dots\oplus
J_{\{\lambda_r\}},\quad \lambda_1\prec\dots\prec
\lambda_r,
\end{equation*}
where $J_{\{\lambda\}}$ denotes a Jordan matrix
with a single eigenvalue $\lambda$. Then
\begin{equation*}
J^{\#}=J_{\{\lambda_1\}}^{\#}\oplus\dots \oplus
J_{\{\lambda_r\}}^{\#},\quad \Lambda(J^{\#})=
\Lambda(J_{\{\lambda_1\}}^{\#})\times \dots
\times \Lambda(J_{\{\lambda_r\}}^{\#});
\end{equation*}
the second since $SJ^{\#}=J^{\#}S$ if and only if
$S=S_1\oplus\dots\oplus S_r$ and
$S_iJ_{\{\lambda_i\}}^{\#}=
J_{\{\lambda_i\}}^{\#}S_i$.

So we may restrict ourselves to a Jordan matrix
$J_{\{\lambda\}}$ with a single eigenvalue
$\lambda$; it reduces to the form
\begin{equation}       \label{2'.3}
J_{\{\lambda\}}^{\text{\it +}} =J_{p_1}(\lambda
I_{n_1})\oplus\dots\oplus J_{p_l}(\lambda
I_{n_l}), \quad p_1>\dots> p_l.
\end{equation}

The matrix \eqref{2'.3} consists of $l$
horizontal  and $l$ vertical strips, the $i$th
strip is divided into $p_i$ substrips. We will
index the $\alpha$th substrip of the $i$th strip
by the pair $(\alpha,i)$. Permuting vertical and
horizontal substrips such that they become
lexicographically ordered with respect to these
pairs,
\begin{equation}       \label{2'.3a}
(11), (12),\dots, (1l), (21), (22),\ldots,
\end{equation}
we obtain the Weyr form $J_{\{\lambda\}}^{\#}$ of
$J_{\{\lambda\}}$ (see Example \ref{e2'.1}). The
partition into substrips is its standard
${\underline n}\times {\underline n}$ partition.

It is well known (and is proved by direct
calculations, see \cite[Sect. VIII, \S 2]{gan})
that all matrices commuting with the matrix
\eqref{2'.3} have the form
$C=[C_{ij}]_{i,j=1}^{l}$ where each $C_{ij}$ is
of the form
$$
\unitlength 0.80mm \linethickness{1.2pt}
\begin{picture}(145.00,63.00)
\put(0.00,10.00){\framebox(60.00,40.00)[cc]{}}
\put(100.00,0.00){\framebox(40.00,60.00)[cc]{}}
\put(55.00,15.00){\makebox(0,0)[cc]{$X_1$}}
\put(55.00,25.00){\makebox(0,0)[cc]{$X_2$}}
\put(55.00,36.00){\makebox(0,0)[cc]{$\vdots$}}
\put(55.00,45.00){\makebox(0,0)[cc]{$X_{p_i}$}}
\put(55.00,53.00){\makebox(0,0)[cc]{$\scriptstyle(p_jj)$}}
\put(45.00,45.00){\makebox(0,0)[cc]{$\cdots$}}
\put(43.00,37.00){\makebox(0,0)[cc]{.}}
\put(45.00,35.00){\makebox(0,0)[cc]{.}}
\put(47.00,33.00){\makebox(0,0)[cc]{.}}
\put(45.00,25.00){\makebox(0,0)[cc]{.}}
\put(35.00,35.00){\makebox(0,0)[cc]{$X_1$}}
\put(35.00,45.00){\makebox(0,0)[cc]{$X_2$}}
\put(25.00,45.00){\makebox(0,0)[cc]{$X_1$}}
\put(30.00,53.00){\makebox(0,0)[cc]{$\cdots$}}
\put(5.00,53.00){\makebox(0,0)[cc]{$
\scriptstyle(1j)$}}
\put(65.00,45.00){\makebox(0,0)[cc]{$\scriptstyle
(1i)$}}
\put(65.00,31.00){\makebox(0,0)[cc]{$\vdots$}}
\put(65.00,15.00){\makebox(0,0)[cc]{$\scriptstyle
(p_ii)$}}
\put(105.00,55.00){\makebox(0,0)[cc]{$X_1$}}
\put(115.00,55.00){\makebox(0,0)[cc]{$X_2$}}
\put(125.00,55.00){\makebox(0,0)[cc]{$\cdots$}}
\put(135.00,55.00){\makebox(0,0)[cc]{$X_{p_j}$}}
\put(115.00,45.00){\makebox(0,0)[cc]{$X_1$}}
\put(127.00,43.00){\makebox(0,0)[cc]{.}}
\put(125.00,45.00){\makebox(0,0)[cc]{.}}
\put(123.00,47.00){\makebox(0,0)[cc]{.}}
\put(135.00,46.00){\makebox(0,0)[cc]{$\vdots$}}
\put(123.00,37.00){\makebox(0,0)[cc]{.}}
\put(125.00,35.00){\makebox(0,0)[cc]{.}}
\put(127.00,33.00){\makebox(0,0)[cc]{.}}
\put(135.00,35.00){\makebox(0,0)[cc]{$X_2$}}
\put(135.00,25.00){\makebox(0,0)[cc]{$X_1$}}
\put(110.00,10.00){\makebox(0,0)[cc]{\LARGE 0}}
\put(10.00,20.00){\makebox(0,0)[cc]{\LARGE 0}}
\put(105.00,63.00){\makebox(0,0)[cc]{$\scriptstyle
(1j)$}}
\put(120.00,63.00){\makebox(0,0)[cc]{$\cdots$}}
\put(135.00,63.00){\makebox(0,0)[cc]{$\scriptstyle
(p_jj)$}}
\put(145.00,55.00){\makebox(0,0)[cc]{$\scriptstyle
(1i)$}}
\put(145.00,31.00){\makebox(0,0)[cc]{$\vdots$}}
\put(145.00,5.00){\makebox(0,0)[cc]{$\scriptstyle
(p_ii)$}} \put(43.00,27.00){\makebox(0,0)[cc]{.}}
\put(47.00,23.00){\makebox(0,0)[cc]{.}}
\put(80.00,31.00){\makebox(0,0)[cc]{\text{or}}}
\end{picture}
$$\\[-13pt]

\noindent if, respectively, $p_i\le p_j$ or
$p_i\ge p_j$. Hence, if a {\it nonzero} subblock
is located at the intersection of the
$(\alpha,i)$ horizontal substrip and the
$(\beta,j)$ vertical substrip, then either
$\alpha= \beta$ and $i\le j$, or $\alpha< \beta$.
Rating the substrips of $C$ in the lexicographic
order \eqref{2'.3a}, we obtain an upper
block-triangular ${\underline n}\times
{\underline n}$ matrix $S$ that commutes with
$J_{\{\lambda\}}^{\#}$. The matrices $S$ form the
algebra $\Lambda(J_{\{\lambda\}}^{\#})$, which is
a reduced algebra with equations (\ref{3}) of the
form $S_{ij}=S_{i'j'}$ and $S_{ij}=0$.
\end{proof}

Note that $J_{\{\lambda\}}^{\#}$ is obtained from
\begin{equation}       \label{2'.4}
J_{\{\lambda\}}=J_{k_1}(\lambda)\oplus\dots
\oplus J_{k_t}(\lambda),\quad k_1\ge\dots\ge k_t,
\end{equation}
as follows: We collect the first columns of $
J_{k_1}(\lambda), \dots, J_{k_t}(\lambda)$ on the
first $t$ columns of $ J_{\{\lambda\}}$, then
permute the rows as well. Next collect the second
columns and permute the rows as well, continue
the process until $J_{\{\lambda\}}^{\#}$ is
achieved.

\begin{remark}            \label{r2'.1}
The block-triangular form of $\Lambda(J^{\#})$ is
easily explained with the help of Jordan chains.
The matrix \eqref{2'.4} represents a linear
operator $\cal A$ in the lexicographically
ordered basis $\{e_{ij}\}_{i=1}^t{}_{j=1}^{k_i}$
such that
\begin{equation}        \label{2'.5}
{\cal A}-\lambda{\bf 1}: e_{ik_i} \mapsto \dots
\mapsto e_{i2}\mapsto e_{i1} \mapsto 0.
\end{equation}
The matrix $J_{\{\lambda\}}^{\#}$ represents the
same linear operator $\cal A$ but in the basis
$\{e_{ij}\}$, lexicographically ordered with
respect to the pairs $(j,i)$:
\begin{equation}       \label{2'.6}
e_{11},\ e_{21},\dots, e_{t1},\ e_{12},\
e_{22},\ldots
\end{equation}
Clearly, $S^{-1} J_{\{\lambda\}}^{\#} S=
J_{\{\lambda\}}^{\#}$ for a nonsingular matrix
$S$ if and only if $S$ is the transition matrix
from the basis \eqref{2'.6} to another Jordan
basis ordered like \eqref{2'.6}. This transition
can be realized by a sequence of operations of
the following form: the $i$th Jordan chain
\eqref{2'.5} is replaced with $\alpha
e_{ik_i}+\beta e_{i,k_{i'}-p} \mapsto \alpha
e_{i,k_i-1}+\beta e_{i',k_{i'}-p-1}\mapsto
\cdots$, where $\alpha,\beta\in k,\ \alpha\ne 0$,
and $p\ge\max\{0, k_{i'}-k_{i}\}$. Since a long
chain cannot be added to a shorter chain, the
matrix $S$ is block-triangular.
\end{remark}


\subsection{Algorithm}        \label{sec3}

In this section, we give an algorithm for
reducing a matrix $M$ to a canonical form under
$\Lambda$-simi\-larity with a reduced
$\underline{n}\times \underline{n}$ algebra
$\Lambda$.

We apply to $M$ the partition
$\underline{n}\times \underline{n}$:
 $$
 M=\begin{bmatrix}
                  M_{11} & \cdots  & M_{1t} \\[-4pt] \hdotsfor{3}\\
                  M_{t1} & \cdots  & M_{tt}
            \end{bmatrix}, \quad
     M_{ij} \in k^{n_i \times n_j}. $$

A block $M_{ij}$ will be called {\it stable} if
it remains invariant under $\Lambda$-similarity
transformations with $M$. Then $M_{ij}=a_{ij}I$
whenever $i\sim j$ and $M_{ij}=0$ (we put
$a_{ij}=0$) whenever $i\not\sim j$ since the
equalities $S_{ii}^{-1}M_{ij}S_{jj}=M_{ij}$ must
hold for all nonsingular block-diagonal matrices
$S=S_{11}\oplus S_{22}\oplus\dots\oplus S_{tt}$
satisfying (\ref{2}).

If all the blocks of $M$ are stable, then $M$ is
invariant under $\Lambda$-similarity, hence $M$
is canonical ($M^{\infty}=M$).

Let there exist a nonstable block. We put the
blocks of $M$ in order
\begin{equation}      \label{5}
M_{t1}<M_{t2}<\dots<M_{tt}<M_{t-1,1}<
M_{t-1,2}<\dots<M_{t-1,t}<\cdots
\end{equation}
and reduce the first (with respect to this
ordering) nonstable block $M_{lr}$. Let
$M'=S^{-1}MS$, where $S\in \Lambda^{*}$ has the
form (\ref{1}). Then the $(l,r)$ block of the
matrix $MS=SM'$ is $$
M_{l1}S_{1r}+M_{l2}S_{2r}+\dots+M_{lr}S_{rr}
=S_{ll}M'_{lr}+S_{l,l+1}M'_{l+1,r}+
\dots+S_{lt}M'_{tr}
$$ or, since all  $M_{ij}<M_{lr}$ are stable,
\begin{equation}      \label{6}
a_{l1}S_{1r}+\dots+a_{l,r-1}S_{r-1,r}+
M_{lr}S_{rr}=S_{ll}M'_{lr}+
S_{l,l+1}a_{l+1,r}+\dots+S_{lt}a_{tr}
\end{equation}
(we have removed in (\ref{6}) all summands with
$a_{ij}=0$; their sizes may differ from the size
of $M_{lr}$).

Let ${\cal I, J}\in T/\!\sim$ be the equivalence
classes such that $l \in {\cal I}$ and $r \in
{\cal J}$.

\begin{description}
\item[\it Case I:]
{\it the $q_{{\cal IJ}}$ equalities (\ref{3}) do
not imply}
\begin{equation}      \label{7}
   a_{l1}S_{1r}+a_{l2}S_{2r}+\dots+a_{l,r-1}S_{r-1,r}
   =S_{l,l+1}a_{l+1,r}+\dots+S_{lt}a_{tr}
\end{equation}
(i.e., there exists a nonzero admissible addition
to $M_{lr}$ from other blocks). Then we make
$M'_{lr}=0$ using $S \in \Lambda^{*}$ of the form
(\ref{1}) that has the diagonal $S_{ii}= I$
($i=1,\dots,t$) and fits both (\ref{3}) and
(\ref{6}) with $M'_{lr}=0$.

\item[\it Case II:]
{\it the $q_{{\cal IJ}}$ equalities (\ref{3})
imply (\ref{7}); $i\not\sim j$.} Then (\ref{6})
simplifies to
\begin{equation}      \label{8}
   M_{lr}S_{rr}= S_{ll}M'_{lr},
\end{equation}
where $S_{rr}$ and $S_{ll}$ are arbitrary
nonsingular matrices. We chose $S\in \Lambda^{*}$
such that
 $$  M'_{lr}= S_{ll}^{-1}M_{lr}S_{rr}=
      \left[ \begin{array}{cc}
                    0  &  I      \\
                    0  &  0
            \end{array} \right]. $$

\item[\it Case III:]
{\it the $q_{{\cal IJ}}$ equalities (\ref{3})
imply (\ref{7}); $i\sim j$.} Then (\ref{6})
simplifies to the form (\ref{8}) with an
arbitrary nonsingular matrix $S_{rr}=S_{ll}$; $
M'_{lr}= S_{ll}^{-1}M_{lr}S_{rr}$ is chosen as a
Weyr matrix.
\end{description}

We restrict ourselves to those admissible
transformations with $M'$ that preserve
$M'_{lr}$. Let us prove that they are the
$\Lambda'$-similarity transformations with
\begin{equation}        \label{8'}
\Lambda':=\{S\in\Lambda\,|\, SM'\equiv M'S\},
\end{equation}
where $A\equiv B$ means that $A$ and $B$ are
$\underline{n}\times \underline{n}$ matrices and
$A_{lr}=B_{lr}$ for the pair $(l,r)$. The
transformation $M'\mapsto S^{-1}M'S$, $S
\in(\Lambda')^*$, preserves $M'_{lr}$ (i.e.
$M'\equiv S^{-1}M'S$) if and only if $ SM'\equiv
M'S$ since $S$ is upper block-triangular and $M'$
coincides with $S^{-1}M'S$ on the places of all
(stable) blocks $M_{ij}<M_{lr}$. The set
$\Lambda'$ is an algebra: let $S,R\in\Lambda'$,
then $M'S$ and $SM'$ coincide on the places of
all $M_{ij}<M_{lr}$ and $R$ is upper
block-triangular, hence $M'SR\equiv SM'R$;
analogously, $SM'R\equiv SRM'$ and $SR\in
\Lambda'$. The matrix algebra $\Lambda'$ is a
reduced algebra since $\Lambda'$ consists of all
$S\in\Lambda$ satisfying the condition \eqref{6}
with $M'_{lr}$ instead of $M_{lr}$.

In Case I, $\Lambda'$ consists of all $S\in
\Lambda$ satisfying (\ref{7}) (we add it to the
system \eqref{3}). In Case II, $\Lambda'$
consists of all $S\in \Lambda$ for which
  $S_{ll}\matr{0}{I}{0}{0} =\matr{0}{I}{0}{0}S_{rr}$,
that is, $$ S_{ll}= \left[ \begin{array}{cc}
                  P_1  & P_2     \\
                    0  & P_3
            \end{array} \right],\quad
   S_{rr}= \left[ \begin{array}{cc}
                  Q_1  & Q_2     \\
                    0  & Q_3
            \end{array} \right], \quad
P_1 = Q_3.$$ In Case III,  $\Lambda'$ consists of
all $S\in \Lambda$ for which the blocks $S_{ll}$
and $S_{rr}$ are equal and commute with the Weyr
matrix $ M'_{lr}$. (It gives an additional
partition of $S\in\Lambda$ in Cases II and III;
we rewrite \eqref{2}--\eqref{3} for smaller
blocks and add the equalities that are needed for
$S_{ll}M'_{lr}=M'_{lr}S_{rr}$.)
\medskip

In this manner, for every pair $(M,\Lambda)$ we
construct a new pair $(M',\Lambda')$ with
$\Lambda'\subset \Lambda$. If $M'$ is not
invariant under $\Lambda'$-similarity, then we
repeat this construction (with an additional
partition of $M'$ in accordance with the
structure of $\Lambda'$) and obtain
$(M'',\Lambda'')$, and so on. Since at every step
we reduce a new block, this process ends with a
certain pair $(M^{(p)}, \Lambda^{(p)})$ in which
all the blocks of $M^{(p)}$ are stable (i.e.
$M^{(p)}$ is $\Lambda^{(p)}$-similar only to
itself). Putting
$(M^{\infty},\Lambda^{\infty}):=(M^{(p)},
\Lambda^{(p)})$, we get the sequence
\begin{equation}           \label{10a}
(M^0, \Lambda^0)=(M, \Lambda), \: (M',
\Lambda'),\dots, \: (M^{(p)}, \Lambda^{(p)})=
(M^{\infty}, \Lambda^{\infty}),
\end{equation}
where
\begin{equation}           \label{10b}
\Lambda^{\infty}=\{S\in\Lambda\,|\,
M^{\infty}S=SM^{\infty}\}.
\end{equation}

\begin{definition}          \label{d1.2'}
The matrix $M^{\infty}$ will be called the
$\Lambda$-{\it canonical form of $M$}.
\end{definition}

\begin{theorem}                \label{t1.1}
Let $\Lambda \subset k^{n\times n}$ be a reduced
matrix algebra. Then $M\sim_{\Lambda}M^{\infty}$
for every $M\in k^{ n\times n}$ and $M
\sim_{\Lambda} N$ if and only if
$M^{\infty}=N^{\infty}$.
\end{theorem}

\begin{proof}
Let $\Lambda$ be a reduced $\underline{n} \times
\underline{n}$ algebra, $M \sim_\Lambda N$, and
let $M_{lr}$ be the first nonstable block of $M$.
Then $M_{ij}$ and $N_{ij}$ are stable blocks
(moreover, $M_{ij}=N_{ij}$) for all
$M_{ij}<M_{lr}$. By reasons of symmetry, $N_{lr}$
is the first nonstable block of $N$; moreover,
$M_{lr}$ and $N_{lr}$ are reduced to the same
form: $M'_{lr} = N'_{lr}$. We obtain pairs
$(M',\Lambda')$ and $(N',\Lambda')$ with the same
$\Lambda'$ and $M'\sim_{\Lambda'} N'.$ Hence
$M^{(i)} \sim_{\Lambda^{(i)}} N^{(i)}$ for all
$i$, so $M^{\infty}=N^{\infty}$.
\end{proof}

\begin{example}                   \label{e1.4}
In Example \ref{e1.2} we considered the canonical
form problem for a pair of matrices under
simultaneous similarity. Suppose the first matrix
is reduced to the Weyr matrix
         $W= \left[  \begin{array}{cc}
             \lambda I_2  & I_2    \\
                       0  &  \lambda I_2
            \end{array} \right].$
Preserving $W$, we may reduce the second matrix
by transformations of $\Lambda$-similarity, where
$\Lambda$ consists of all $4\times 4$ matrices of
the form
         $  \left[  \begin{array}{cc}
                     S_1  & S_2   \\
                       0  & S_1
            \end{array} \right],\
 S_i\in k^{2\times 2}.$
For instance, one of the $\Lambda$-canonical
matrices is
\begin{equation}       \label{9}
C=  \left[  \begin{tabular}{c|c}
        $C_3$  &\!\!\!\!\!  \begin{tabular}{c|c}
       $C_6$ & $C_7$  \\  \hline
        $C_4$ &$ C_5$
 \end{tabular}\!\!\!
\\  \hline
                      $ C_1$  & $C_2$
            \end{tabular} \right]
=\left[ \begin{tabular}{c|c}
   $ \!\!\!\! \begin{array}{cc}-1&1\\
0&-1\end{array}$\!\!\!\! &
 \!\!\!\begin{tabular}{c|c}2&$\emptyset$ \\
 \hline 0&1\end{tabular}\!\!\! \\  \hline
                      $ 3I_2$  & $\emptyset$
            \end{tabular} \right],
\end{equation}
where $C_1,\dots,C_7$ are reduced blocks and
$C_q= \emptyset$ means that $C_q$ was made zero
by additions from other blocks (Case I of the
algorithm). Hence, $(W,C)$ may be considered as a
canonical pair of matrices under similarity. Note
that $\matr{W}{C}{0}{0}$ is a canonical matrix
with respect to $D$-similarity, where $D=\{
S\oplus S\,|\,S\in k^{2\times 2}\}$.

\begin{definition}          \label{d1.2a}
By the canonical form of a pair of $n\times n$
matrices $(A,B)$ under simultaneous similarity is
meant a pair $(W,C)$, where $\matr{W}{C}{0}{0}$
is the canonical form of the matrix
$\matr{A}{B}{0}{0}$ with respect to
$D$-similarity with $D=\{ S\oplus S\,|\,S\in
k^{n\times n}\}$.
\end{definition}

Clearly, each pair of matrices is similar to a
canonical pair and two pairs of matrices are
similar if and only if they reduce to the same
canonical pair. The full list of canonical pairs
of complex $4\times 4$ matrices under
simultaneous similarity was presented in
\cite{ser_gal}.
\end{example}

\begin{remark}                \label{r1.1}
Instead of (\ref{5}), we may use another linear
ordering in the set of blocks, for example, $
M_{t1}<M_{t-1,1}<
\dots<M_{11}<M_{t2}<M_{t-1,2}<\cdots$ or $
M_{t1}<M_{t-1,1}<M_{t2}<M_{t-2,1}<M_{t-1,2}
<M_{t3}<\cdots. $ It is necessary only that
$(i,j)\ll (i',j')$ implies $ M_{ij}<M_{i'j'}$,
where $(i,j)\ll (i',j')$ indicates the existence
of a nonzero addition from $M_{ij}$ to $M_{i'j'}$
and is defined as follows:
\end{remark}

\begin{definition}        \label{d1.3}
Let $\Lambda$ be a reduced $\underline{n}\times
\underline{n}$ algebra. For unequal pairs
$(i,j),(i',j')\in T\times T$ (see \eqref{0}), we
put $(i,j)\ll (i',j')$ if either $i=i'$ and there
exists $S\in\Lambda^*$ with $S_{jj'}\ne 0$, or
$j=j'$ and there exists $S\in\Lambda^*$ with
$S_{i'i}\ne 0$.
\end{definition}


\subsection{Structured $\Lambda$-canonical matrices}
\label{sec4} The structure of a
$\Lambda$-canonical matrix $M$ will be clearer if
we partition it into boxes $M_1, M_2,\dots$, as
it was made in \eqref{9}.

\begin{definition}      \label{d2.0}
Let $M=M^{(r)}$ for a certain $r\in
\{0,1,\dots,p\}$ (see \eqref{10a}). We partition
its reduced part into {\it boxes}
$M_1,M_2,\dots,M_{q_{r+1}-1}$ as follows: Let
$\Lambda^{(l)}\ (1\le l\le r)$ be a reduced
$\underline{n}^{(l)}\times\underline{n}^{(l)}$
algebra from the sequence \eqref{10a}, we denote
by $M^{(l)}_{ij}$ the blocks of $M$ under the
$\underline{n}^{(l)}\times \underline{n}^{(l)}$
partition. Then $M_{q_{l+1}}$ for $l\ne p$
denotes the first nonstable block among
$M^{(l)}_{ij}$ with respect to
$\Lambda^{(l)}$-similarity (it is reduced when
$M^{(l)}$ is transformed to $M^{(l+1)}$);
$M_{q_{_l}+1}<\dots< M_{q_{_{l+1}}-1}\ (q_0:=0)$
are  all the blocks $M_{ij}^{(l)}$ such that

(i) if $l<p$, then $M_{ij}^{(l)}<
M_{q_{_{l+1}}}$;

(ii) if $l>0$, then $M_{ij}^{(l)}$ is not
contained in the boxes $M_1,\dots,M_{q_{_l}}$.

\noindent (Note that each box $M_i$ is 0,
$\matr{0}{I}{0}{0}$, or a Weyr matrix.)
Furthermore, put
\begin{equation}       \label{9'}
\Lambda_{q_{_l}}=\Lambda_{q_{_l}+1}=
\dots=\Lambda_{q_{_{l+1}}-1} :=\Lambda^{(l)}.
\end{equation}
Generalizing the equalities \eqref{8'} and
\eqref{10b}, we obtain
\begin{equation}       \label{9''}
\Lambda_i=\{S\in\Lambda\,|\,MS\equiv_i SM\},
\end{equation}
where $MS\equiv_i SM$ means that $MS-SM$ is zero
on the places of $M_1,\dots,M_i$.
 \end{definition}

\begin{definition}      \label{d2.00}
By a {\it structured $\Lambda$-canonical matrix}
we mean a $\Lambda$-canonical matrix $M$ which is
divided into boxes $M_1, M_2,\dots,M_{q_{p+1}-1}$
and each box $M_i$ that falls into Case I from
Section \ref{sec3} (and hence is 0) is marked by
$\emptyset$ (see \eqref{9}).
\end{definition}

Now we describe the construction of
$\Lambda$-canonical matrices.

\begin{definition}       \label{d2.01}
By a {\it part} of a matrix
$M=[a_{ij}]_{i,j=1}^n$ is meant an arbitrary set
of its entries given with their indices. By a
{\it rectangular part} we mean a part of the form
$B=[a_{ij}],\ p_1 \le i \le p_2, \ q_1 \le j \le
q_2.$ We consider a partition of $M$ into
disjoint rectangular parts (which is not, in
general, a partition into substrips, see the
matrix \eqref{9}) and write, generalizing
(\ref{5}), $B<B'$ if either $p_2 = p'_2$ and $q_1
< q'_1$, or $p_2 > p'_2$.
\end{definition}

\begin{definition}              \label{d2.1}
Let $M=[M_{ij}]$ be an
$\underline{n}\times\underline{n}$ matrix
partitioned into rectangular parts
$M_1<M_2<\cdots<M_m$ such that this partition
refines the partition into the blocks $M_{ij}$,
and let each $M_i$ be equal to $0,\
\matr{0}{I}{0}{0}$, or a Weyr matrix. For every
$q\in\{0,1,\dots, m\}$, we define a subdivision
of strips into $q$-strips as follows: The 0-{\it
strips} are the strips of $M$. Let $q>0$. We make
subdivisions of $M$ into substrips that extend
the partitions of $M_1,\dots,M_q$ into cells 0,
$I$, $\lambda I$ (i.e., the new subdivisions run
the length of every boundary of the cells). If a
subdivision passes through a cell $I$ or $\lambda
I$ from $M_1,\dots,M_q$, then we construct the
perpendicular subdivision such that the cell
takes the form
$$
\begin{bmatrix} I&0\\ 0&I \end{bmatrix}
\quad {\rm or} \quad
\begin{bmatrix} \lambda I&0\\ 0&\lambda I
\end{bmatrix},$$
and repeat this construction for all new
divisions until $M_1,\dots,M_q$ are partitioned
into cells $0,\ I$, or $\lambda I$. The obtained
substrips will be called the {\it $q$-strips of
$M$}; for example, the partition into $q$-strips
of the matrix \eqref{9} has the form $$
\left[\begin{tabular}{cc|cc} -1&1&2&0\\0&-1&0&1\\
\hline 3&0&0&0\\0&3&0&0
\end{tabular}\right] \
\text{for $q=0,1,2$};\
\left[\begin{tabular}{c|c|c|c} -1&1&2&0\\ \hline
0&-1&0&1\\ \hline 3&0&0&0\\  \hline 0&3&0&0
\end{tabular}\right] \
\text{for $q=3,4,5,6,7$}. $$ We say that the
$\alpha$th $q$-strip of an $i$th (horizontal or
vertical) strip {\it is linked to} the $\beta$th
$q$-strip of an $j$th strip if (i) $\alpha=\beta$
and $i\sim j$ (including $i=j$; see \eqref{0}),
or if (ii) their intersection is a (new) cell $I$
from $M_1,\dots,M_q$, or if (iii) they are in the
transitive closure of (i) and (ii).
\end{definition}

Note that if $M$ is a $\Lambda$-canonical matrix
with the boxes $M_1,\dots, M_{q_{p+1}-1}$ (see
Definition \ref{d2.0}), then
$M_1<\dots<M_{q_{p+1}-1}$. Moreover, if
$\Lambda_q$ ($1\le q< q_{p+1}$, see \eqref{9'})
is a reduced ${\underline{n}\,}_{q}\times
{\underline{n}\,}_q$ algebra with the equivalence
relation $\sim$ (see \eqref{0}), then the
partition into $q$-strips is the
${\underline{n}\,}_q\times {\underline{n}\,}_q$
partition; the $i$th $q$-strip is linked with the
$j$th $q$-strip if and only if $i\sim j$.

 \begin{theorem}        \label{t2.1}
Let $\Lambda$ be a reduced
$\underline{n}\times\underline{n}$ algebra and
let $M$ be an arbitrary
$\underline{n}\times\underline{n}$ matrix
partitioned into rectangular parts
$M_1<M_2<\dots<M_m,$ where each $M_i$ is equal to
$\emptyset$ (a marked zero block),
$\matr{0}{I}{0}{0}$, or a Weyr matrix. Then $M$
is a structured $\Lambda$-canonical matrix with
boxes $M_1,\dots,M_m$ if and only if each $M_q$
$(1\le q\le m)$ satisfies the following
conditions:
\begin{itemize}
\item[(a)] $M_q$ is the intersection of two
$(q-1)$-strips.

\item[(b)]
Suppose there exists $M'=S^{-1}MS$ (partitioned
into rectangular parts conformal to $M$; $S\in
\Lambda^{*}$) such that
$M'_1=M_1,\dots,M'_{q-1}=M_{q-1},$ but $M'_q\ne
M_q$. Then $M_q=\emptyset$.

\item[(c)] Suppose $M'$ from (b) does not exist.
Then $M_q$ is a Weyr matrix if the horizontal and
the vertical $(q-1)$-strips of $M_q$ are linked;
$ M_q =\matr{0}{I}{0}{0}$ otherwise.
\end{itemize}
\end{theorem}

\begin{proof}
This theorem follows immediately from the
algorithm of Section \ref{sec3}.
\end{proof}


\section{Linear matrix problems}      \label{sss2}

\subsection{Introduction}           \label{sec2.1}

In Section \ref{sss2} we study a large class of
matrix problems. In the theory of representations
of finite dimensional algebras, similar classes
of matrix problems are given by vectorspace
categories \cite{rin,sim}, bocses \cite{roi,cra},
modules over aggregates \cite{gab_roi, gab_vos},
or vectroids \cite{bel_ser}.

Let us define the considered class of matrix
problems (in terms of elementary transformations
to simplify its use; a more formal definition
will be given in Section \ref{22simil}). Let
$\sim$ be an equivalence relation in
$T=\{1,\dots,t\}$. We say that a $t\times t$
matrix $A=[a_{ij}]$ {\it links an equivalence
class ${\cal I}\in T/\!\sim\ $ to an equivalence
class} ${\cal J}\in T/\!\sim\ $  if $a_{ij}\ne 0$
implies $(i,j)\in {\cal I}\times {\cal J}$.
Clearly, if $A$ links ${\cal I}$ to ${\cal J}$
and $A'$ links ${\cal I'}$ to ${\cal J'}$, then
$AA'$ links ${\cal I}$ to ${\cal J'}$ when ${\cal
J}={\cal I'}$, and $AA'=0$ when ${\cal J}\ne
{\cal I'}$.\footnote{Linking matrices behave as
mappings; one may use vector spaces $V_{\cal I}$
instead of equivalence classes ${\cal I}$ ($\dim
V_{\cal I}=\#(\cal I)$) and linear mappings of
the corresponding vector spaces instead of
linking matrices.} We also say that a sequence of
nonnegative integers
$\underline{n}=(n_1,n_2,\dots,n_t)$ is a {\it
step-sequence} if $i\sim j$ implies $n_i=n_j$.

Let $A=[a_{ij}]$ link ${\cal I}$ to ${\cal J}$,
let $\underline{n}$ be a step-sequence, and let
$(l,r)\in \{1,\dots,n_i\}\times\{1,\dots,n_j\}$
for $(i,j)\in {\cal I}\times{\cal J}$ (since
$\underline{n}$ is a step-sequence, $n_i$ and
$n_j$ do not depend on the choice of $(i,j)$);
denote by $A^{[l,r]}$ the
$\underline{n}\times\underline{n}$  matrix that
is obtained from $A$ by replacing each entry
$a_{ij}$ with the following $n_i\times n_j$ block
$A^{[l,r]}_{ij}$: if $a_{ij}=0$ then $
A^{[l,r]}_{ij}=0$, and if $a_{ij}\ne 0$ then the
$(l,r)$ entry of $A^{[l,r]}_{ij}$ is $a_{ij}$ and
the others are zeros.

Let a triple
\begin{equation}       \label{0.1}
(T/\!\sim,\ \{P_i\}_{i=1}^p,\ \{V_j\}_{j=1}^q)
\end{equation}
consist of the set of equivalence classes of
$T=\{1,\dots,t\}$, a finite or empty set of
linking nilpotent upper-triangular matrices
$P_i\in k^{t\times t}$, and a finite set of
linking matrices $V_j\in k^{t\times t}$. Denote
by $\cal P$ the product closure of
$\{P_i\}_{i=1}^p$ and by $\cal V$ the closure of
$\{V_j\}_{j=1}^q$ with respect to multiplication
by $\cal P$ (i.e., $\cal{VP}\subset\cal{V}$ and
$\cal{PV}\subset\cal{V}$). Since $P_i$ are
nilpotent upper-triangular $t\times t$ matrices,
$P_{i_1}P_{i_2}\dots P_{i_t}=0$ for all
$i_1,\dots,i_t$. Hence, ${\cal P}$ and ${\cal V}$
are finite sets consisting of linking nilpotent
upper-triangular matrices and, respectively,
linking matrices:
\begin{equation}       \label{0.2}
{\cal P}= \{P_{i_1}P_{i_2}\dots P_{i_r}\,|\,r\le
t\},\quad {\cal V}= \{PV_jP'\,|\,P,P'\in\{I_t\}
\cup{\cal P},\ 1\le j\le q \}.
\end{equation}
For every step-sequence
$\underline{n}=(n_1,\ldots,n_t)$, we denote by
${\cal M}_{\underline{n}\times\underline{n}}$ the
vector space generated by all
${\underline{n}\times\underline{n}}$ matrices of
the form $V^{[l,r]},\ 0\ne V\in\cal V$.

\begin{definition}    \label{d0.1}
A {\it linear matrix problem} given by a triple
\eqref{0.1} is the canonical form problem for
$\underline{n}\times\underline{n}$ matrices
$M=[M_{ij}]\in{\cal
M}_{\underline{n}\times\underline{n}}$ with
respect to sequences of the following
transformations:

\begin{itemize}
\item[(i)]
For each equivalence class ${\cal I}\in
T/\!\sim$, the same elementary transformations
within all the vertical strips $M_{\bullet, i},\
i\in{\cal I}$, then the inverse transformations
within the horizontal strips $M_{i, \bullet},\
i\in{\cal I}$.

\item[(ii)]
For $a\in k$ and a nonzero matrix $P=[p_{ij}]\in
\cal P$ linking $\cal I$ to $\cal J$, the
transformation $M\mapsto
(I+aP^{[l,r]})^{-1}M(I+aP^{[l,r]})$; that is, the
addition of $ap_{ij}$ times the $l$th column of
the strip $M_{\bullet,i}$ to the $r$th column of
the  strip $M_{\bullet,j}$ simultaneously for all
$(i,j)\in {\cal I} \times{\cal J}$, then the
inverse transformations with rows of $M$.
\end{itemize}
 \end{definition}

\begin{example}       \label{e0.2}
As follows from Example \ref{e0.1}, the problem
of classifying representations of the quiver
\eqref{0.01} may be given by the triple $$
(\{\{1\},\{2\},\{3,4\}\},\ \varnothing,\
\{e_{11},e_{21},e_{31},e_{41}, e_{42}, e_{43}\}),
$$ where $e_{ij}$ denotes the matrix in which the
$(i,j)$ entry is 1 and the others are 0. The
problem of classifying representations of each
quiver may be given in the same manner.

\end{example}

\begin{example}                      \label{e0.3}
Let ${\cal S}=\{p_1,\dots,p_n\}$ be a finite
partially ordered set whose elements are indexed
such that $p_i<p_j$ implies $i<j$. Its {\it
representation} is a matrix $M$ partitioned into
$n$ vertical strips $M_1,\dots,M_n$; we allow
arbitrary row-transformations, arbitrary
column-transformations within each vertical
strip, and additions of linear combinations of
columns of $M_i$ to a column of $M_j$ if  $p_i <
p_j$. (This notion is important for
representation theory and  was introduced by
Nazarova and Roiter \cite{naz_roi}, see also
\cite{gab_roi} and \cite{sim}.) The problem of
classifying representations of the poset ${\cal
S}$ may be given by the triple $$
(\{\{1\},\{2\},\dots,\{n+1\}\}, \{e_{ij}\,|\, p_i
< p_j\}, \{e_{n+1,1}, e_{n+1,2},\dots,
e_{n+1,n}\}). $$
 \end{example}

\begin{example}
Let us consider  Wasow's canonical form problem
for an analytic at the point $\varepsilon=0$
matrix
\begin{equation}       \label{0.3}
A(\varepsilon)=A_0+\varepsilon A_1 +
\varepsilon^2 A_2+\cdots,\quad A_i\in {\mathbb
C}^{n\times n},
\end{equation}
relative to analytic similarity:
\begin{equation}       \label{0.4}
A(\varepsilon)\mapsto B(\varepsilon):=
S(\varepsilon)^{-1} A(\varepsilon)
S(\varepsilon),
\end{equation}
where $S(\varepsilon)=S_0+\varepsilon S_1+\cdots$
and $S(\varepsilon)^{-1}$ are analytic matrices
at $0$. Let us restrict ourselves to the
canonical form problem for the first $t$ matrices
$A_0, A_1,\dots, A_{t-1}$ in the expansion
\eqref{0.3}. By \eqref{0.4},
$S(\varepsilon)B(\varepsilon)=  A(\varepsilon)
S(\varepsilon)$, that is $S_0B_0=A_0S_0,\dots,
S_0B_{t-1}+ S_1B_{t-2}+ \dots + S_{t-1}B_0=
A_0S_{t-1}+ A_1S_{t-2}+ \dots + A_{t-1}S_0,$ or
in the matrix form
\begin{multline*}
\begin{bmatrix}
S_0&S_1&\cdots&S_{t-1}\\&S_0&\ddots&\vdots\\
&&\ddots&S_1\\{\text{\LARGE 0}}&&&S_0
\end{bmatrix}
\begin{bmatrix}
B_0&B_1&\cdots&B_{t-1}\\&B_0&\ddots&\vdots\\
&&\ddots&B_1\\{\text{\LARGE 0}}&&&B_0
\end{bmatrix}
=\\
\begin{bmatrix}
A_0&A_1&\cdots&A_{t-1}\\&A_0&\ddots&\vdots\\
&&\ddots&A_1\\{\text{\LARGE 0}}&&&A_0
\end{bmatrix}
\begin{bmatrix}
S_0&S_1&\cdots&S_{t-1}\\&S_0&\ddots&\vdots\\
&&\ddots&S_1\\{\text{\LARGE 0}}&&&S_0
\end{bmatrix}.
\end{multline*}
Hence this problem may be given by the following
triple of one-element sets: $$ (\{T\},\ \{J_t\},\
\{I_t\}), $$ where $J_t= e_{12}+e_{23}+\dots
+e_{t-1,t}$ is the nilpotent Jordan block. Then
all elements of $T=\{1,2,\dots,t\}$ are
equivalent, ${\cal P}=\{ J_t, J_t^2,\dots,
J_t^{t-1}\}$ and ${\cal V}=\{ I_t, J_t, \dots,
J_t^{t-1}\}$. This problem is wild even if $t=2$,
see \cite{fri1, ser1}. I am thankful to S.
Friedland for this example.
\end{example}

In Section \ref{22simil} we give a definition of
the linear matrix problems in a form, which is
more similar to Gabriel and Roiter's definition
(see Example \ref{e0.1a}) and is better suited
for Belitski\u\i's algorithm.

In Section \ref{s3.3} we prove that every
canonical matrix may be decomposed into a direct
sum of indecomposable canonical matrices by
permutations of its rows and columns. We also
investigate the canonical form problem for upper
triangular matrices under upper triangular
similarity (see \cite{thi}).

In Section \ref{s3.3'} we consider a canonical
matrix as a parametric matrix whose parameters
are eigenvalues of its Jordan blocks. It enables
us to describe a set of canonical matrices having
the same structure.

In Section \ref{s3.4} we consider linear matrix
problems that give matrix problems with
independent row and column transformations and
prove that the problem of classifying modules
over a finite-dimensional algebra may be reduced
to such a matrix problem. The reduction is a
modification of Drozd's reduction of the problem
of classifying modules over an algebra to the
problem of classifying representations of bocses
\cite{dro1} (see also Crawley-Boevey \cite{cra}).
Another reduction of the problem of classifying
modules over an algebra to a matrix problem with
arbitrary row transformations was given in
\cite{gab_vos}.


\subsection{Linear matrix problems and
$\Lambda$-similarity} \label{22simil}

In this section we give another definition of the
linear matrix problems, which is equivalent to
the Definition \ref{d0.1} but is often more
convenient. The set of admissible transformations
will be formulated in terms of
$\Lambda$-similarity; it simplifies the use of
Belitski\u\i's algorithm.

\begin{definition}    \label{d3.1}
An algebra $\varGamma \subset k^{t\times t}$ of
upper triangular matrices will be called a {\it
basic matrix algebra} if $$
\begin{bmatrix}
                       a_{11}&\cdots&a_{1t} \\
                       &\ddots&\vdots \\
                       \text{\Large 0} & & a_{tt}
            \end{bmatrix}\in\varGamma
\quad  {\rm implies} \quad
           \begin{bmatrix}
                       a_{11}  & & \text{\Large 0} \\
                       &\ddots&   \\
                       \text{\Large 0}  & & a_{tt}
            \end{bmatrix}\in\varGamma.
$$
\end{definition}

\begin{lemma} \label{l3.2}
(a) Let $\varGamma\subset k^{t\times t}$ be a
basic matrix algebra, $\cal D$ be the set of its
diagonal matrices, and $\cal R$ be the set of its
matrices with zero diagonal. Then there exists a
basis $E_1,\dots,E_r$ of $\cal D$ over $k$ such
that all entries of its matrices are 0 and 1,
moreover
\begin{equation}       \label{3.3}
E_1+\dots+E_r=I_t,\ \ E_{\alpha}E_{\beta}=0\
({\alpha} \ne {\beta}),\ \ E_{\alpha}^2=
E_{\alpha}.
\end{equation}
These equations imply the following decomposition
of $\varGamma$ (as a vector space over $k$) into
a direct sum of subspaces:
\begin{equation}       \label{3.3a}
\varGamma={\cal D}\oplus{\cal R}=
\Bigl(\bigoplus_{{\alpha}=
1}^rkE_{\alpha}\Bigr)\oplus \Bigl(\bigoplus_{
{\alpha},{\beta}=1}^rE_{\alpha}{\cal
R}E_{\beta}\Bigr).
\end{equation}

(b) The set of basic $t\times t$ algebras is the
set of reduced ${\underline{1}\times
\underline{1}}$ algebras, where
$\underline{1}:=(1,1,\dots,1)$. A basic $t\times
t$ algebra $\varGamma$ is the reduced
${\underline{1}\times \underline{1}}$ algebra
given by
\begin{itemize}
  \item $T/\!\sim \,=\{{\cal I}_1,\dots, {\cal I}_r\}$
where ${\cal I}_{\alpha}$ is the set of indices
defined by $E_{\alpha}=\sum_{i\in{\cal
I}_{\alpha}} e_{ii},$ see \eqref{3.3}, and
  \item a family of systems of the form \eqref{00}
such that for every $\alpha,\beta\in
\{1,\dots,r\}$ the solutions of its $({\cal
I}_{\alpha}, {\cal I}_{\beta})$ system form the
space $E_{\alpha}{\cal R}E_{\beta}$.
\end{itemize}
\end{lemma}

\begin{proof}
(a) By Definition \ref{d3.1}, $\varGamma$ is the
direct sum of vector spaces ${\cal D}$ and ${\cal
R}$. Denote by $\cal F$ the set of diagonal
${t\times t}$  matrices with entries in
$\{0,1\}$. Let $D\in \cal D$, then
$D=a_1F_1+\dots +a_lF_l$, where $a_1,\dots, a_l$
are distinct nonzero elements of $k$ and
$F_1,\dots, F_l$ are such matrices from $\cal F$
that $F_iF_j=0$ whenever $i\ne j$. The vectors
$(a_1,\dots, a_l)$, $(a_1^2,\dots, a_l^2),\dots,
(a_1^l,\dots, a_l^l)$ are linearly independent
(they form a Vandermonde determinant), hence
there exist $b_1,\dots, b_l\in k$ such that
$F_1=b_1D+ b_2D^2 +\dots + b_lD^l\in \cal D$,
analogously $F_2,\dots, F_l\in \cal D$. It
follows that ${\cal D}=kE_1\oplus \dots \oplus
kE_r$, where $E_1,\dots,E_r\in\cal F$ and satisfy
\eqref{3.3}. Therefore, ${\cal R}=(E_1+ \dots
+E_r){ \cal R}(E_1+ \dots +E_r)=
\bigoplus_{{\alpha},{\beta}}E_{\alpha} {\cal
R}E_{\beta}$, we get the decomposition
\eqref{3.3a}. (Note that \eqref{3.3} is a
decomposition of the identity of $\varGamma$ into
a sum of minimal orthogonal idempotents and
\eqref{3.3a} is the Peirce decomposition of
$\varGamma$, see \cite{dr_ki}.)
\end{proof}

\begin{definition}    \label{d3.2}
A {\it linear matrix problem given by a pair}
\begin{equation}       \label{3.4}
(\varGamma,\cal M), \quad \varGamma {\cal
M}\subset {\cal M},\ {\cal M}\varGamma \subset
{\cal M},
\end{equation}
consisting of a basic $t\times t$ algebra
$\varGamma$ and a vector space ${\cal M}\subset
k^{t\times t}$, is the canonical form problem for
matrices $M\in {\cal
M}_{\underline{n}\times\underline{n}}$ with
respect to $\varGamma_{\underline{n}
\times\underline{n}}$-similarity transformations
$$ M\mapsto S^{-1}MS,\quad S\in
\varGamma_{\underline{n}\times\underline{n}}^*,
$$ where
$\varGamma_{\underline{n}\times\underline{n}}$
and ${\cal M}_{\underline{n}\times\underline{n}}$
consist of $\underline{n}\times\underline{n}$
matrices whose blocks satisfy the same linear
relations as the entries of all $t\times t$
matrices from $\varGamma$ and $\cal M$
respectively.

More exactly,
$\varGamma_{\underline{n}\times\underline{n}}$ is
the reduced ${\underline{n}\times\underline{n}}$
algebra given by the same system (\ref{00}) and
$T/\!\sim\;=\{{\cal I}_1,\dots, {\cal I}_r\}$ as
$\varGamma$ (see Lemma \ref{l3.2}(b)).%
\footnote{%
If $n_1 >0,\dots,n_t>0$, then  $\varGamma
_{\underline{n}\times\underline{n}}$ is Morita
equivalent to $\varGamma$; moreover, $\varGamma$
is the basic algebra for
$\varGamma_{\underline{n}\times\underline{n}}$ in
terms of the theory of algebras, see
\cite{dr_ki}.

} Next,
\begin{equation}       \label{3.4a}
{\cal M}=\Bigl(\sum_{\alpha=1}^rE_{\alpha}
\Bigr){\cal M} \Bigl(\sum_{\beta=1}^rE_{\beta}
\Bigr)= \bigoplus_{\alpha,\beta=1}^r
E_{\alpha}{\cal M}E_{\beta}
\end{equation}
(see \eqref{3.3}), hence there is a system of
linear equations
\begin{equation}       \label{3.5}
\sum_{(i,j)\in{\cal I}_{\alpha} \times{\cal
I}_{\beta}} d_{ij}^{(l)}x_{ij} = 0, \quad 1 \le l
\le p_{\alpha\beta}, \quad {\cal
I}_{\alpha},{\cal I}_{\beta}\in T/\!\sim,
\end{equation}
such that $\cal M$ consists of all matrices
$[m_{ij}]_{i,j=1}^t$ whose entries satisfy the
system \eqref{3.5}. Then ${\cal
M}_{\underline{n}\times\underline{n}}$
($\underline{n}$ is a step-sequence) denotes the
vector space of all
${\underline{n}\times\underline{n}}$ matrices
$[M_{ij}]_{i,j=1}^t$ whose blocks satisfy the
system \eqref{3.5}: $$ \sum_{(i,j)\in{\cal
I}_{\alpha} \times{\cal
I}_{\beta}}d_{ij}^{(l)}M_{ij} = 0,\quad 1 \le
l\le p_{\alpha\beta}, \quad {\cal
I}_{\alpha},{\cal I}_{\beta}\in T/\!\sim . $$
\end{definition}

\begin{theorem}
Definitions \ref{d0.1} and \ref{d3.2} determine
the same class of matrix problems:

(a) The linear matrix problem given by a triple
$(T/\!\sim,\ \{P_i\}_{i=1}^p,\ \{V_j\}_{j=1}^q)$
may be also given by the pair $(\varGamma,{ \cal
M})$, where $\varGamma$ is the basic matrix
algebra generated by $P_1,\dots, P_p$ and all
matrices $E_{\cal I}=\sum_{j\in {\cal I}} e_{jj}$
$({\cal I}\in T/\!\sim)$ and $\cal M$ is the
minimal vector space of matrices containing
$V_1,\dots,V_q$ and closed with respect to
multiplication by $P_1,\dots, P_p$.

(b) The linear matrix problem given by a  pair
$(\varGamma,{ \cal M})$ may be also given by a
triple $(T/\!\sim,\ \{P_i\}_{i=1}^p,\
\{V_j\}_{j=1}^q)$, where $T/\!\sim \,=\{{\cal
I}_1,\dots, {\cal I}_r\}$ (see Lemma
\ref{l3.2}(b)), $\{P_i\}_{i=1}^p$ is the union of
bases for the spaces $E_{\alpha}{\cal
R}E_{\beta}$ (see (\ref{3.3a})), and
$\{V_j\}_{j=1}^q$ is the union of bases for the
spaces $E_{\alpha}{\cal M}E_{\beta}$ (see
(\ref{3.4a})).
\end{theorem}

\begin{proof}
(a) Let $\underline{n}$ be a step-sequence. We
first prove that the set of admissible
transformations is the same for both the matrix
problems; that is, there exists a  sequence of
transformations (i)--(ii) from Definition
\ref{d0.1} transforming $M$ to $N$ (then we write
$M \simeq N$) if and only if they are
$\Lambda$-similar with $\Lambda:=
\varGamma_{\underline{n}\times \underline{n}}$.

By Definition \ref{d0.1}, $M \simeq N$ if and
only if $S^{-1}MS=N$, where $S$ is a product of
matrices of the form
\begin{equation}       \label{3.2}
I+aE^{[l,r]}_{\cal I} \ (a\ne -1\ \text{if}\
l=r),\ \ I+bP^{[l,r]},
\end{equation}
where $a,b\in k$, ${\cal I}\in T/\!\sim$ and
$0\ne P\in {\cal P}$. Since $S\in\Lambda$, $M
\simeq N$ implies $M \sim_{\Lambda}N$.

Let $M \sim_{\Lambda}N$, that is $SMS^{-1}=N$ for
a nonsingular $S\in\Lambda$. To prove $M \simeq
N$, we must expand $S^{-1}$ into factors of the
form \eqref{3.2}; it suffices to reduce $S$ to
$I$ multiplying by matrices \eqref{3.2}. The
matrix $S$ has the form \eqref{1} with $S_{ii} =
S_{jj}$ whenever $i \sim j$; we reduce $S$ to the
form \eqref{1} with $S_{ii} =I_{n_i}$ for all $i$
multiplying by matrices $I+aE^{[l,r]}_{\cal I}$.
Denote by $\cal Q$ the set of all
$\underline{n}\times \underline{n}$ matrices of
the form $P^{[l,r]},$ $P\in {\cal P}$. Since
${\cal Q}\cup \{ E^{[l,r]}_{\cal I}\}_{{\cal
I}\in T/\!\sim}$ is product closed, it generates
$\Lambda$ as a vector space. Therefore,  $S=
I+\sum_{Q\in \cal Q}a_QQ\ (a_Q\in k)$. Put ${\cal
Q}_l=\{Q\in{\cal Q}\,|\, Q^l=0\},$ then ${\cal
Q}_0=\varnothing$ and ${\cal Q}_t={\cal Q}$.
Multiplying $S$ by $\prod_{Q\in{\cal
Q}}(I-a_QQ)=I-\sum_{Q\in \cal Q}a_QQ+\cdots$, we
make $S=I+\cdots$, where the points denote a
linear combination of products of matrices from
$\cal Q$ and each product consists of at least 2
matrices (so its degree of nilpotency is at most
$t-1$). Each product is contained in ${\cal
Q}_{t-1}$ since $\cal Q$ is product closed, hence
$S= I+\sum_{Q\in {\cal Q}_{t-1}}b_QQ$. In the
same way we get $S= I+\sum_{Q\in {\cal
Q}_{t-2}}c_QQ$, and so on until obtain $S= I$.

Clearly, the set of reduced $\underline{n}\times
\underline{n}$ matrices ${\cal
M}_{\underline{n}\times \underline{n}}$ is the
same for both the matrix problems.
\end{proof}

Hereafter we shall use only Definition \ref{d3.2}
of linear matrix problems.


\subsection{Krull--Schmidt theorem}
\label{s3.3} In this section we study
decompositions of a canonical matrix into a
direct sum of indecomposable canonical matrices.

Let a linear matrix problem be given by a pair
$(\varGamma,{ \cal M})$. By the {\it canonical
matrices} is meant the
$\varGamma_{\underline{n}\times
\underline{n}}$-canonical matrices $M\in{\cal
M}_{\underline{n} \times\underline{n}}$ for
step-sequences $\underline{n}$. We say that
$\underline{n}\times \underline{n}$ matrices $M$
and $N$ are {\it equivalent} and write $M \simeq
N$ if they are $\varGamma_{\underline{n}
\times\underline{n}}$-similar. The {\it
block-direct sum} of an
$\underline{m}\times\underline{m}$ matrix
$M=[M_{ij}]_{i,j=1}^t$ and an
$\underline{n}\times\underline{n}$ matrix
$N=[N_{ij}]_{i,j=1}^t$ is the
$(\underline{m}+\underline{n})\times(\underline{m}+
\underline{n})$ matrix $$ M\uplus N=[M_{ij}\oplus
N_{ij}]_{i,j=1}^t. $$ A matrix $M\in{\cal
M}_{\underline{n} \times\underline{n}}$ is said
to be {\it indecomposable} if $\underline{n}\neq
0$ and $M\simeq M_1\uplus M_2$ implies that $M_1$
or $M_2$ has size  $0\times 0$.

\begin{theorem}          \label{t3.1}
For every canonical
$\underline{n}\times\underline{n}$ matrix $M$,
there exists a permutation matrix $P\in
\varGamma_{\underline{n}\times\underline{n}}$
such that
\begin{equation}          \label{9a}
P^{-1}MP= \underbrace{M_1\uplus\dots\uplus
M_1}_{\mbox{$q_1$ copies}} \uplus\dots\uplus
\underbrace{M_l\uplus\dots\uplus
M_l}_{\mbox{$q_l$ copies}}
\end{equation}
where $M_i$ are distinct indecomposable canonical
matrices. The decomposition (\ref{9a}) is
determined by $M$ uniquely up to permutation of
summands. \end{theorem}

\begin{proof}
Let $M$ be a canonical $\underline{n}\times
\underline{n}$ matrix. The repeated application
of Belitski\u\i's algorithm produces the sequence
\eqref{10a}: $(M, \Lambda), \: (M',
\Lambda'),\dots, \: (M^{(p)}, \Lambda^{(p)}),$
where $\Lambda=
\varGamma_{\underline{n}\times\underline{n}}$ and
$\Lambda^{(p)}= \{S\in \Lambda\, |\, MS =SM\}$
(see \eqref{10b}) are reduced
$\underline{n}\times\underline{n}$ and
$\underline{m}\times\underline{m}$ algebras; by
Definition \ref{d1.1}(a) $\Lambda$ and
$\Lambda^{(p)}$ determine equivalence relations
$\sim$ in $T=\{1,\dots,t\}$ and $\approx$ in
$T^{(p)}=\{1,\dots,r\}$. Since $M$ is canonical,
$M^{(i)}$ differs from $M^{(i+1)}$ only by
additional subdivisions. The strips with respect
to the $\underline{m}\times\underline{m}$
partition will be called the {\it substrips}.

Denote by $\Lambda^{(p)}_0$ the subalgebra of
$\Lambda^{(p)}$ consisting of its block-diagonal
$\underline{m}\times\underline{m}$ matrices, and
let $S\in\Lambda^{(p)}_0$. Then it has the form
$$
S=C_1\oplus\dots\oplus C_r, \quad
C_{\alpha}=C_{\beta} \text{ if } {\alpha}\approx
{\beta}.
$$
It may be also considered as a
block-diagonal $\underline{n}\times\underline{n}$
matrix $S=S_1\oplus\dots\oplus S_t$ from
$\Lambda$ (since $\Lambda^{(p)} \subset\Lambda$);
each block $S_i$ is a direct sum of subblocks
$C_{\alpha}$.

Let ${\cal I}$ be an equivalence class from
$T^{(p)}/\!\approx$. In each $S_i$, we permute
its subblocks $C_{\alpha}$ with ${\alpha}\in \cal
I$ into the first subblocks: $$ {\bar
S}_i=C_{\alpha_1}\oplus\dots\oplus
C_{\alpha_p}\oplus C_{\beta_1}\oplus\dots\oplus
C_{\beta_q},\quad \alpha_1<\dots< \alpha_p,\ \
\beta_1<\dots< \beta_q, $$ where $\alpha_1,\dots,
\alpha_p \in {\cal I}$ and $\beta_1,\dots,
\beta_q \notin {\cal I}$ (note that
$C_{\alpha_1}=\dots= C_{\alpha_p}$); it gives the
matrix ${\bar S}=Q^{-1}SQ$, where
$Q=Q_1\oplus\dots\oplus Q_t$ and $Q_i$ are
permutation matrices. Let $i\sim j$, then
$S_i=S_j$ (for all $S\in\Lambda$), hence the
permutations within $S_i$ and $S_j$ are the same.
We have $Q_i=Q_j$ if $i\sim j$, therefore
$Q\in\Lambda$.

Making the same permutations of
substrips within each strip of
$M$, we get ${\bar M}=Q^{-1}MQ$.
Let $M=[M_{ij}]_{i,j=1}^t$
relatively to the
$\underline{n}\times\underline{n}$
partition, and let $M=[N_{
\alpha\beta}]_{\alpha,\beta
=1}^r$ relatively to the
$\underline{m}\times\underline{m}$
partition. Since $M$ is
canonical, all $N_{\alpha\beta}$
are reduced, hence
$N_{\alpha\beta}=0$ if $\alpha
\not\approx\beta$ and
$N_{\alpha\beta}$ is a scalar
square matrix if $\alpha
\approx\beta$. The $\bar M$ is
obtained from $M$ by gathering
all subblocks $N_{
\alpha\beta}$, $(\alpha,\beta)
\in {\cal I}\times{\cal I}$, in
the left upper cover of every
block $M_{ij}$, hence ${\bar
M}_{ij}= A_{ij}\oplus B_{ij}$,
where $A_{ij}$ consists of
subblocks $N_{ \alpha\beta}$,
$\alpha,\beta \in {\cal I}$, and
$B_{ij}$ consists of subblocks
$N_{ \alpha\beta}$,
$\alpha,\beta \notin {\cal I}$.
We have ${\bar M}= A_1\uplus B$,
where $A_1=[A_{ij}]$ and
$B=[B_{ij}]$. Next apply the
same procedure to $B$; continue
the process until get $$
P^{-1}MP=A_1\uplus\dots\uplus
A_l, $$ where $P\in\Lambda$ is a
permutation matrix and the
summands $A_i$ correspond to the
equivalence classes of
$T^{(p)}/\!\approx$.

The matrix $A_1$ is canonical. Indeed, $M$ is a
canonical matrix, by Definition \ref{d2.0}, each
box $X$ of $M$ has the form $\emptyset$,
$\matr{0}{I}{0}{0}$, or a Weyr matrix. It may be
proved that the part of $X$ at the intersection
of substrips with indices in $\cal I$ has the
same form and this part is a box of $A_1$.
Furthermore, the matrix $A_1$ consists of
subblocks $N_{ \alpha\beta}$, $(\alpha,\beta) \in
{\cal I}\times{\cal I}$, that are scalar matrices
of the same size $t_1\times t_1$. Hence, $A_1=
M_1\uplus\dots\uplus M_1$ ($t_1$ times), where
$M_1$ is canonical. Analogously, $A_i=
M_i\uplus\dots\uplus M_i$ for all $i$ and the
matrices $M_i$ are canonical.
 \end{proof}

\begin{corollary}[Krull--Schmidt theorem]
For every matrix $M\in{\cal M}_{\underline{n}
\times\underline{n}}$, there exists its
decomposition $$ M\simeq M_1\uplus\dots\uplus M_r
$$ into a block-direct sum of indecomposable
matrices $M_i\in{\cal M}_{ \underline{n}_i
\times\underline{n}_i}$. Moreover, if $$ M\simeq
N_1\uplus\dots\uplus N_s $$ is another
decomposition into a block-direct sum of
indecomposable matrices, then $r=s$ and, after a
suitable reindexing, $M_1\simeq N_1,\dots,
M_r\simeq N_r$.
\end{corollary}

\begin{proof}
This statement follows from Theorems \ref{t1.1}
and \ref{t3.1}. Note that this statement is a
partial case of the Krull--Schmidt theorem
\cite{bas} for additive categories; namely, for
the category of matrices $\cup\,{\cal
M}_{\underline{n} \times\underline{n}}$ (the
union over all step-sequences ${\underline{n}}$)
whose morphisms from $M\in{\cal M}_{\underline{m}
\times\underline{m}}$ to $N\in{\cal
M}_{\underline{n} \times\underline{n}}$ are the
matrices $S\in{\cal M}_{\underline{m}
\times\underline{n}}$ such that $MS=SN$. (The set
${\cal M}_{\underline{m} \times\underline{n}}$ of
${\underline{m} \times\underline{n}}$ matrices is
defined like ${\cal M}_{\underline{n}
\times\underline{n}}$.)
\end{proof}

\begin{example}
Let us consider the {\it canonical form problem
for upper triangular matrices under upper
triangular similarity} (see \cite{thi} and the
references given there). The set $\varGamma^t$ of
all upper triangular $t\times t$ matrices is a
reduced $\underline{1} \times\underline{1}$
algebra, so every $A\in\varGamma^t$ is reduced to
the $\varGamma^t$-canonical form $A^{\infty}$ by
Belitski\u\i's algorithm; moreover, in this case
the algorithm is very simplified: All diagonal
entries of $A=[a_{ij}]$ are not changed by
transformations; the over-diagonal entries are
reduced starting with the last but one row: $$
a_{t-1,t};\ a_{t-2,t-1},\ a_{t-2,t};\
a_{t-3,t-2},\, a_{t-3,t-1},\, a_{t-3,t};\ldots\,
. $$ Let $a_{pq}$ be the first that changes by
admissible transformations. If there is a nonzero
admissible addition, we make $a_{pq}=0$;
otherwise $a_{pq}$ is reduced by transformations
of equivalence or similarity, in the first case
me make $a_{pq}\in \{0,1\}$, in the second case
$a_{pq}$ is not changed. Then we restrict the set
of admissible transformations to those that
preserve the reduced $a_{pq}$, and so on. Note
that this reduction is possible for an arbitrary
field $k$, which does not need to be
algebraically closed.

Furthermore, $\varGamma^t$ is a basic $t\times t$
algebra, so we may consider $A^{\infty}$ as a
canonical matrix for the linear matrix problem
given by the pair $(\varGamma^t,\varGamma^t)$. By
Theorem \ref{t3.1} and since a permutation
$t\times t$ matrix $P$ belongs to $\varGamma^t$
only if $P=I$, there exists a unique
decomposition
$$
A^{\infty}=A_1\uplus\dots\uplus A_r
$$
where each $A_i$ is an indecomposable canonical
$\underline{n}_i\times \underline{n}_i$ matrix,
$\underline{n}_i\in \{0,1\}^t$. Let $t_i\times
t_i$ be the size of $A_i$, then
$\varGamma^t_{\underline{n}_i\times
\underline{n}_i}$  may be identified with
$\varGamma^{t_i}$ and $A_i$ may be considered as
a $\varGamma^{t_i}$-canonical matrix.

Let $A^{\infty}=[a_{ij}]_{i,j=1}^t$, define the
graph $G_A$ with vertices $1,\dots,t$ having the
edge $i$---$j$ $(i<j)$ if and only if both
$a_{ij}=1$ and $a_{ij}$ was reduced by
equivalence transformations. Then $G_A$ is a
union of trees; moreover, $G_A$ is a tree if and
only if $A^{\infty}$ is indecomposable (compare
with \cite{ser2}).

The Krull--Schmidt theorem for this case and a
description of nonequivalent indecomposable
$t\times t$ matrices for $t\le 6$ was given by
Thijsse \cite{thi}.
\end{example}


\subsection{Parametric canonical matrices}
 \label{s3.3'}

Let a linear matrix problem be given by a pair
$(\varGamma,\cal M)$. The set ${\cal M}$ may be
presented as the matrix space of all solutions
$[m_{ij}]_{i,j=1}^t$ of the system \eqref{3.5} in
which the unknowns $x_{ij}$ are disposed like the
blocks \eqref{5}: $x_{t1}\prec
x_{t2}\prec\cdots$. The Gauss-Jordan elimination
procedure to the system \eqref{3.5} starting with
the last unknown reduces the system to the form
\begin{equation}       \label{3'.1}
x_{lr}=\sum_{(i,j)\in {\cal
N}_f}c_{ij}^{(l,r)}x_{ij}, \quad (l,r)\in {\cal
N}_d,
\end{equation}
where ${\cal N}_d$ and ${\cal N}_f$ are such that
${\cal N}_d\cup{\cal N}_f= \{1,\dots,t\} \times
\{1,\dots,t\}$ and ${\cal N}_d\cap{\cal N}_f=
\varnothing$; the inequality $c_{ij}^{(l,r)}\ne
0$ implies $i\sim l$, $j\sim r$ and $x_{ij}\prec
x_{lr}$ (i.e., every unknown $x_{lr}$ with
$(l,r)\in {\cal N}_d\cap ({\cal I}\times{\cal
J})$ is a linear combination of the preceding
unknowns with indices in ${\cal N}_f\cap ({\cal
I}\times{\cal J})$).

A block $M_{ij}$ of $M\in{\cal M}_{\underline{n}
\times \underline{n}}$ will be called {\it free}
if $(i,j)\in{\cal N}_f$, {\it dependent} if
$(i,j)\in{\cal N}_d$. A box $M_i$ will be called
{\it free} ({\it dependent}) if it is a part of a
free (dependent) block.

\begin{lemma}  \label{l3.1'}
The vector space ${\cal M}_{\underline{n}\times
\underline{n}}$ consists of all
$\underline{n}\times\underline{n}$ matrices
$[M_{ij}]_{i,j=1}^t$ whose free blocks are
arbitrary and the dependent blocks are their
linear combinations given by (\ref{3'.1}):
\begin{equation}       \label{3'.1'}
M_{lr}=\sum_{(i,j)\in {\cal
N}_f}c_{ij}^{(l,r)}M_{ij}, \quad (l,r)\in {\cal
N}_d.
\end{equation}
On each step of Belitski\u\i's algorithm, the
reduced subblock of $M\in{\cal
M}_{\underline{n}\times \underline{n}}$ belongs
to a free block (i.e., all boxes $M_{q_1},
M_{q_2},\dots$ from Definition \ref{d2.0} are
subblocks of free blocks).
\end{lemma}

\begin{proof}
Let us prove the second statement. On the $l$th
step of Belitski\u\i's algorithm, we reduce the
first nonstable block $M^{(l)}_{\alpha\beta}$ of
the matrix $M^{(l)}= [M^{(l)}_{ij}]$ with respect
to $\Lambda^{(l)}$-similarity. If
$M^{(l)}_{\alpha\beta}$ is a subblock of a
dependent block $M_{ij}$, then
$M^{(l)}_{\alpha\beta}$ is a linear combination
of already reduced subblocks of blocks preceding
to $M_{ij}$, hence $M^{(l)}_{\alpha\beta}$ is
stable, a contradiction.
\end{proof}

We now describe a set of canonical matrices
having `the same form'.

\begin{definition}    \label{d2.2}
Let $M$ be a structured (see Definition
\ref{d2.00}) canonical
$\underline{n}\times\underline{n}$ matrix, let
$M_{r_1} <\dots< M_{r_s}$ be those of its free
boxes that are Weyr matrices (Case III of
Belitski\u\i's algorithm), and let
$\lambda_{t_{i-1}+1} \prec \cdots \prec
\lambda_{t_i}$ be the distinct eigenvalues of
$M_{r_i}$. Considering some of $\lambda_i$ (resp.
all $\lambda_i$) as parameters, we obtain a
parametric matrix $M(\vec\lambda),\
\vec\lambda:=(\lambda_{i_1},\dots,
\lambda_{i_p})$ (resp. $\vec\lambda:=
(\lambda_1,\dots, \lambda_p),\ p:=t_s$), which
will be called a {\it semi-parametric} (resp.
{\it parametric}) {\it canonical matrix}. Its
{\it  domain of parameters} is the set of all
$\vec a\in k^p$ such that $M(\vec a)$ is a
structured canonical
$\underline{n}\times\underline{n}$ matrix with
the same disposition of the boxes $\emptyset $ as
in $M$.
\end{definition}

\begin{theorem}      \label{t2.2}
The domain of parameters $\cal D$ of a parametric
canonical $\underline{n}\times\underline{n}$
matrix $M(\vec\lambda)$ is given by a system of
equations and inequalities of the following three
types:

     (i) $f(\vec\lambda)=0$,

     (ii) $(d_1(\vec\lambda),\dots,
d_n(\vec\lambda))\ne(0,\dots,0)$,

     (iii)  $\lambda_i \prec \lambda_{i+1}$,

\noindent where $f,d_j \in k[x_1,\dots,x_p]$.
\end{theorem}

\begin{proof}  Let $M_1 <\dots<M_m$ be all the
boxes of $M(\vec\lambda)$. Put ${\cal A}_0:=k^p$
and denote by ${\cal A}_q$ $(1\le q\le m)$ the
set of all $\vec a \in k^p$ such that $M(\vec a)$
coincides with $M(\vec a)^{\infty}$ on
$M_1,\dots,M_q$. Denote by $\Lambda_q(\vec a)\
(1\le q\le m,\ \vec a\in {\cal A}_q)$ the
subalgebra of $\Lambda:=
\varGamma_{\underline{n}\times\underline{n}}$
consisting of all $S\in \Lambda$ such that
$SM(\vec a)$ coincides with $M(\vec a)S$ on the
places of $M_1,\dots,M_q$.

We prove that there is a system ${\cal S}_q(\vec
\lambda)$ of equations of the form (\ref{2}) and
(\ref{3}) (in which every $c_{ij}^{(l)}$ is an
element of $k$ or a parameter $\lambda_i$ from
$M_1,\dots,M_q$) satisfying the following two
conditions for every $\vec \lambda=\vec{a}\in
{\cal A}_q$:

(a) the equations of each $({\cal I},\ {\cal }J)$
subsystem of (\ref{3}) are linearly independent,
and

(b) $\Lambda_q(\vec a)$ is a reduced
$\underline{n}_q\times \underline{n}_q$ algebra
given by ${\cal S}_q(\vec a)$.

\noindent This is obvious for $\Lambda_0(\vec
a):=\Lambda (\vec a)$. Let it hold for $q-1,$ we
prove it for $q$.

We may assume that $M_q$ is a free box since
otherwise ${\cal A}_{q-1}={\cal A}_q$ and
$\Lambda_q(\vec a)= \Lambda_{q-1}(\vec a)$ for
all $\vec{a}\in {\cal A}_{q-1}$. Let $(l,r)$ be
the indices of $M_q$ as a block of the
$\underline{n}_{q-1}\times \underline{n}_{q-1}$
matrix $M$ (i.e. $M_q=M_{lr}$). In accordance
with the algorithm of Section \ref{sec3}, we
consider two cases:

{\it Case 1:} $M_q = \emptyset .$ Then the
equality (\ref{7}) is not implied by the system
${\cal S}_{q-1}(\vec a)$ (more exactly, by its
$({\cal I},{\cal J})$ subsystem with ${\cal I}
\times {\cal J}\ni (l,r)$, see \eqref{3}) for all
$\vec a\in {\cal A}_q$. It means that there is a
nonzero determinant formed by columns of
coefficients of the system
$(\ref{3})\cup(\ref{7})$. Hence, ${\cal A}_q$
consists of all $\vec a\in {\cal A}_{q-1}$ that
satisfy the condition (ii), where
$d_1(\vec\lambda),\dots, d_n(\vec\lambda)$ are
all such determinants; we have ${\cal S}_q(\vec
\lambda) ={\cal S}_{q-1}(\vec \lambda)\cup
(\ref{7})$.

{\it Case 2:} $M_q\neq \emptyset .$ Then
(\ref{7}) is implied by the system ${\cal
S}_{q-1}(\vec a)$ for all $\vec a\in {\cal A}_q$.
Hence, ${\cal A}_q$ consists of all $\vec a\in
{\cal A}_{q-1}$ that satisfy the conditions
$d_1(\vec a)=0,\dots,d_n(\vec a)=0$ of the form
(i) and (if $M_q$ is a Weyr matrix with the
parameters
$\lambda_{t_{q-1}+1},\dots,\lambda_{t_q}$) the
conditions $\lambda_{t_{q-1}+1}\prec \cdots \prec
\lambda_{t_q}$ of the form (iii). The system
${\cal S}_{q}(\vec \lambda)$ is obtained from
${\cal S}_{q-1}(\vec \lambda)$  as follows: we
rewrite (\ref{2})--(\ref{3}) for smaller blocks
of $\Lambda_q$ (every system (\ref{3}) with
${\cal I}\ni l$ or ${\cal J}\ni r$ gives several
systems with the same coefficients, each of them
connects equally disposed subblocks of the blocks
$S_{ij}$ with $(i,j)\in \cal I\times\cal J$) and
add the equations needed for
$S_{ll}M_{lr}=M_{lr}S_{rr}$.

Since ${\cal A}_0=k^p,\ {\cal A}_q\ (1\le q\le
m)$ consists of all $\vec a\in {\cal A}_{q-1}$
that satisfy a certain system of conditions
(i)--(iii) and ${\cal D}:={\cal A}_m$ is the
domain of parameters of $M(\vec\lambda)$.
\end{proof}

\begin{example}           \label{e2.1}
The canonical pair of matrices from Example
\ref{e1.4} has the parametric form
        $$\left( \left[  \begin{tabular}{c|c}
                  $\!\!\!  \begin{array}{cc}
          \lambda_1 & 1  \\  0 &\lambda_1
          \end{array}$\!\!\! &{\LARGE 0}\\      \hline
       {\LARGE 0} & $\!\!\! \begin{array}{cc}
           \lambda_2 & 1\\  0 &\lambda_2
                    \end{array}\!\!\!$
                         \end{tabular}\right], \
              \left[  \begin{tabular}{c|c}
             $\!\!\! \begin{array}{cc}
               \mu_2 &  1 \\  0  & \mu_2
                  \end{array}\!\!\!$ &
           \!\!\!\!\! \begin{tabular}{c|c}
            $\mu_5$  &  $\emptyset$ \\ \hline
            $\mu_3$  & $\mu_4$
          \end{tabular}\!\!\! \\  \hline
              $\!\!\! \begin{array}{cc}
              \mu_1 & 0 \\   0  & \mu_1
                  \end{array}\!\!\! $
                               & $\emptyset$
            \end{tabular} \right]\right). $$
Its domain of parameters is given by the
conditions $\lambda_1 \prec \lambda_2, \
\mu_1\neq 0,\ \mu_3 =0$, and $\mu_4\neq\mu_5.$
\end{example}

\begin{remark} \label{r2.1}
The number of parametric canonical
$\underline{n}\times \underline{n}$ matrices is
finite for every $\underline{n}$ since there
exists a finite number of partitions into boxes,
and each box is $\emptyset $,
$\matr{0}{I}{0}{0},$ or a Weyr matrix (consisting
of 0, 1, and parameters). Therefore, a linear
matrix problem for matrices of size
$\underline{n}\times \underline{n}$ is reduced to
the problem of finding a finite set of parametric
canonical matrices and their domains of
parameters. Each domain of parameters is given by
a system of polynomial equations and inequalities
(of the types (i)--(iii)), so it is a
semi-algebraic set; moreover, it is locally
closed up to the conditions (iii).
\end{remark}


\subsection{Modules over finite-dimensional
algebras}     \label{s3.4}

In this section, we consider matrix problems with
independent row and column transformations (such
problems are called {\it separated} in
\cite{gab_roi}) and reduce to them the problem of
classifying modules over algebras.

\begin{lemma} \label{l3.4}
Let $\varGamma\subset k^{m\times m}$ and
$\Delta\subset k^{n\times n}$ be two basic matrix
algebras and let ${\cal N} \subset k^{m\times n}$
be a vector space such that $\varGamma{\cal N}
\subset {\cal N}$ and ${\cal N} \Delta \subset
{\cal N}$. Denote by $0\diagdown {\cal N}$ the
vector space of ${(m+n)\times (m+n)}$ matrices of
the form $\matr{0}{N}{0}{0}$, $N\in{\cal N}$.
Then the pair
\begin{equation*}      
 (\varGamma\oplus \Delta ,\ 0\diagdown {\cal N})
\end{equation*}
determines the canonical form problem for
matrices $N\in {\cal
N}_{\underline{m}\times\underline{n}}$ in which
the row transformations are given by $\varGamma$
and the column transformations are given by
$\Delta $: $$ N\mapsto CNS,\quad C\in
\varGamma_{\underline{m}\times \underline{m}}^*,\
S\in \Delta
_{\underline{n}\times\underline{n}}^*. $$
\end{lemma}

\begin{proof}
Put $M=\matr{0}{N}{0}{0}$ and apply Definition
\ref{d3.2}.
\end{proof}

In particular, if $\varGamma=k$, then the row
transformations are arbitrary; this
classification problem is studied intensively in
representation theory where it is given by a
vectorspace category \cite{rin, sim}, by a module
over an aggregate \cite{gab_roi, gab_vos}, or by
a vectroid \cite{bel_ser}.

\medskip

The next theorem shows that the problem of
classifying modules over a finite dimensional
algebra $\varGamma$ may be reduced to a linear
matrix problem. If the reader is not familiar
with the theory of modules (the used results can
be found in \cite{dr_ki}), he may omit this
theorem since it is not used in the next
sections. The algebra $\varGamma$ is isomorphic
to a matrix algebra, so by Theorem \ref{r1.0} we
may assume that $\varGamma$ is a reduced matrix
algebra. Moreover, by the Morita theorem
\cite{dr_ki}, the category of modules over
$\varGamma$ is equivalent to the category of
modules over its basic algebra, hence we may
assume that $\varGamma$ is a basic matrix
algebra. All modules are taken to be right
finite-dimensional.

\begin{theorem}  \label{t3.2}
For every basic $t\times t$ algebra $\varGamma$,
there is a natural bijection between:

(i) the set of isoclasses of indecomposable
modules over $\varGamma$ and

(ii) the set of indecomposable $(\varGamma\oplus
\varGamma,\ 0\diagdown {\cal R})$ canonical
matrices without zero ${\underline n}\times
{\underline n}$ matrices with ${\underline
n}=(0,\dots,0,n_{t+1},\dots,n_{2t})$, where
${\cal R}=\rad{\varGamma}$ (it consists of the
matrices from $\varGamma$ with zero diagonal).
 \end{theorem}

\begin{proof}
We will successively reduce
\begin{itemize}
\item [(a)]
the problem of classifying, up to isomorphism,
modules over a basic matrix algebra
$\varGamma\subset k^{t\times t}$
\end{itemize}
to a linear matrix problem.

Drozd \cite{dro1} (see also Crawley-Boevey
\cite{cra}) proposed a method for reducing the
problem (a) (with an arbitrary finite-dimensional
algebra $\varGamma$) to a matrix problem. His
method was founded on the following well-known
property of projective modules \cite[p.
156]{dr_ki}:

For every module $M$ over $\varGamma$, there
exists an exact sequence
\begin{gather}
P\stackrel{\varphi}{\longrightarrow} Q
\stackrel{\psi}{\longrightarrow} M\longrightarrow
0, \label{3.11}\\ \Ker\varphi\subset \rad P,
\quad \im\varphi\subset \rad Q, \label{3.11a}
\end{gather}
where $P$ and $Q$ are projective modules.
Moreover, if $$
P'\stackrel{\varphi'}{\longrightarrow} Q
'\stackrel{\psi'}{\longrightarrow}
M'\longrightarrow 0 $$ is another exact sequence
with these properties, then $M$ is isomorphic to
$M'$ if and only if there exist isomorphisms
$f:P\to P'$ and $g:Q\to Q'$ such that
$g\varphi=\varphi'f$.

Hence, the problem (a) reduces to
\begin{itemize}
\item [(b)]
the problem of classifying triples
$(P,Q,\varphi)$, where $P$ and $Q$ are projective
modules over a basic matrix algebra $\varGamma$
and $\varphi:P\to Q$ is a homomorphism satisfying
\eqref{3.11a}, up to isomorphisms $(f,g):
(P,Q,\varphi)\to (P',Q',\varphi')$ given by pairs
of isomorphisms $f:P\to P'$ and $g:Q\to Q'$ such
that $g\varphi=\varphi'f$.
\end{itemize}

By Lemma \ref{l3.2}, $\varGamma$ is a reduced
algebra, it defines an equivalence relation
$\sim$ in $T=\{1,\dots,t\}$ (see (\ref{0})).
Moreover, if $T/\!\sim\ =\{{\cal I}_1,\dots,
{\cal I}_r\}$, then the matrices
$E_{\alpha}=\sum_{i\in{\cal I}_{\alpha}} e_{ii}$
$(\alpha=1,\dots,r)$ form a decomposition
\eqref{3.3} of the identity of $\varGamma$ into a
sum of minimal orthogonal idempotents, and
$P_1=E_1\varGamma,\dots, P_r=E_r\varGamma$ are
all nonisomorphic indecomposable projective
modules over $\varGamma$.

Let $\varphi\in\Hom_{\varGamma}(P_{\beta},
P_{\alpha})$, then $\varphi$ is given by $F:=
\varphi (E_{\beta})$. Since $F\in P_{\alpha},\
F=E_{\alpha}F$. Since $\varphi$ is a
homomorphism, $\varphi (E_{\beta}G)=0$ implies
$FG=0$ for every $G\in {\varGamma}$. Taking
$G=I-E_{\beta}$, we have $F(I-E_{\beta})=0$, so
$F=F E_{\beta}=E_{\alpha}F E_{\beta}$. Hence we
may identify $\Hom_{\varGamma}(P_{\beta},
P_{\alpha})$ and $E_{\alpha}\varGamma E_{\beta}$:
\begin{equation}       \label{3.11b}
\Hom_{\varGamma}(P_{\beta}, P_{\alpha})=
\varGamma_{\alpha\beta}:= E_{\alpha}\varGamma
E_{\beta}.
\end{equation}
The set $\cal R$ of all matrices from $\varGamma$
with zero diagonal is the radical of $\varGamma$;
$\rad P_{\alpha}= P_{\alpha}{\cal R}=
E_{\alpha}{\cal R}$. Hence $\varphi\in
\Hom_{\varGamma}(P_{\beta}, P_{\alpha})$
satisfies $\im\varphi\subset \rad P_{\alpha}$ if
and only if $\varphi(E_{\beta})\in {\cal
R}_{\alpha \beta}:= E_{\alpha}{\cal R}
E_{\beta}$.

Let $$ P=P_1^{(p_1)}\oplus\dots\oplus
P_r^{(p_r)},\ \ Q=Q_1^{(q_1)}\oplus\dots\oplus
Q_r^{(q_r)} $$ be two projective modules, where
$X^{(i)}:=X\oplus\dots\oplus X$ ($i$ times); we
may identify $\Hom_{\varGamma}(P, Q)$ with the
set of block matrices
$\Phi=[\Phi_{\alpha\beta}]_{\alpha,\beta=1}^r$,
where $\Phi_{\alpha\beta}\in
\varGamma_{\alpha\beta}^{q_{\alpha}\times
p_{\beta}}$ is a ${q_{\alpha}\times p_{\beta}}$
block with entries in $\varGamma_{\alpha\beta}$.
Moreover, $\im\Phi\subset\rad Q$ if and only if
$\Phi_{\alpha\beta}\in {\cal
R}_{\alpha\beta}^{q_{\alpha}\times p_{\beta}}$
for all $\alpha,\beta$. The condition $\Ker
\varphi\subset \rad P$ means that there exists no
decomposition $P=P'\oplus P''$ such that $P''\ne
0$ and $\varphi(P'')=0$.

Hence, the problem (b) reduces to
\begin{itemize}
\item [(c)]
the problem of classifying $\underline{q}\times
\underline{p}$ matrices
$\Phi=[\Phi_{\alpha\beta}]_{ \alpha,\beta=1}^r$,
$\Phi_{\alpha\beta}\in {\cal
R}_{\alpha\beta}^{q_{\alpha}\times p_{\beta}}$,
up to transformations
\begin{equation}       \label{3.12}
\Phi\longmapsto C\Phi S,
\end{equation}
where $C=[C_{\alpha\beta}]_{\alpha,\beta=1}^r$
and $S=[S_{\alpha\beta}]_{\alpha,\beta=1}^r$ are
invertible $\underline{q}\times \underline{q}$
and $\underline{p}\times \underline{p}$ matrices,
$C_{\alpha\beta}\in \varGamma_{\alpha\beta}^{
q_{\alpha}\times q_{\beta}}$, and
$S_{\alpha\beta}\in \varGamma_{\alpha\beta}^{
p_{\alpha}\times p_{\beta}}$. The matrices $\Phi$
must satisfy the condition: there exists no
transformation \eqref{3.12} making a zero column
in $\Phi$.
\end{itemize}

Every element of $\varGamma_{\alpha\beta}$ is an
upper triangular matrix $a=[a_{ij}]_{i,j=1}^t$;
define its submatrix $\bar{a}=[a_{ij}]_{(i,j)\in
{\cal I}_{\alpha}\times{\cal I}_{\beta}}$ (by
\eqref{3.11b}, $a_{ij}=0$ if $(i,j)\notin {\cal
I}_{\alpha}\times{\cal I}_{\beta}$). Let $\Phi =
[\Phi_{\alpha\beta} ]_{\alpha,\beta=1}^r$ with
$\Phi_{\alpha\beta}\in {\cal
R}_{\alpha\beta}^{q_{\alpha}\times p_{\beta}}$;
replacing every entry $a$ of $\Phi_{\alpha\beta}$
by the matrix $\bar{a}$ and permuting rows and
columns to order them in accordance with their
position in $\varGamma$, we obtain a matrix
$\bar{\Phi}$ from $\varGamma_{\underline{m}\times
\underline{n}}$, where $m_i:=q_{\alpha}$ if
$i\in{\cal I}_{\alpha}$ and $n_j:=p_{\beta}$ if
$j\in{\cal I}_{\beta}$. It reduces the problem
(c) to
\begin{itemize}
\item [(d)]
the problem of classifying $\underline{m}\times
\underline{n}$ matrices $N\in{\cal
R}_{\underline{m}\times \underline{n}}$
($\underline{m}$ and $\underline{n}$ are
step-sequences) up to transformations
\begin{equation}       \label{3.13}
N\mapsto CNS,\quad C\in
\varGamma_{\underline{m}\times \underline{m}}^*,\
S\in
\varGamma_{\underline{n}\times\underline{n}}^*.
\end{equation}
The matrices $N$ must satisfy the condition: for
each equivalence class ${\cal I}\in T/\!\sim$,
there is no transformation \eqref{3.13} making
zero the first column in all the $i$th vertical
strips with $i\in\cal I$.
\end{itemize}

By Lemma \ref{l3.4}, the problem (d) is the
linear matrix problem given by the pair
$(\varGamma\oplus \varGamma,\ 0\diagdown {\cal
R})$ with an additional condition on the
transformed matrices: they do not reduce to a
block-direct sum with a zero summand whose size
has the form $\underline{n}\times\underline{n}$,
$\underline{n}=(0,\dots,0,n_{t+1},\dots,n_{2t})$.
\end{proof}

\begin{corollary}
The following three statements are equivalent:

(i) The number of nonisomorphic indecomposable
modules over an algebra $\varGamma$ is finite.

(ii) The set of nonequivalent ${n\times n}$
matrices over $\varGamma$ is finite for every
integer $n$.

(iii) The set of nonequivalent elements is finite
in every algebra $\Lambda$ that is Morita
equivalent \cite{dr_ki} to $\varGamma$ (two
elements $a,b\in\Lambda $ are said to be
equivalent if $a=xby$ for invertible $x,y\in
\Lambda$).
\end{corollary}
The corollary follows from the proof of Theorem
\ref{t3.2} and the second Brauer--Thrall
conjecture \cite{gab_roi}: the number of
nonisomorphic indecomposable modules over an
algebra $\Lambda$ is infinite if and only if
there exist infinitely many nonisomorphic
indecomposable $\Lambda$-modules of the same
dimension. The condition \eqref {3.11a} does not
change the finiteness since every exact sequence
\eqref{3.11} is the direct sum of an exact
sequence $P_1\to Q_1 \to M\to 0$ that satisfies
this condition and exact sequences of the form
$e_i\varGamma \to e_i\varGamma \to 0\to 0$ and
$e_i\varGamma \to 0 \to 0\to 0$, where
$1=e_1+\dots+e_r$ is a decomposition of
$1\in\Gamma$ into a sum of minimal orthogonal
idempotents.


\section{Tame and wild matrix problems}
\label{sss3}

\subsection{Introduction}  \label{s4.1}

In this section, we prove the Tame--Wild Theorem
in a form approaching to the Third main theorem
from \cite{gab_vos}.

Generalizing the notion of a quiver and its
representations, Roiter \cite{roi} introduced the
notions of a bocs (=bimodule over category with
coalgebra structure) and its representations. For
each free triangular bocs, Drozd \cite{dro1} (see
also \cite{dro,dro2}) proved that the problem of
classifying its representations satisfies one and
only one of the following two conditions
(respectively,  is of {\it tame} or {\it wild}
type): (a) all but a finite number of
nonisomorphic indecomposable representations of
the same dimension belong to a finite number of
one-parameter families, (b) this problem
`contains' the problem of classifying pairs of
matrices up to simultaneous similarity. It
confirmed a conjecture due to Donovan and
Freislich \cite{don} states that every finite
dimensional algebra is either tame or wild.
Drozd's proof was interpreted by Crawley-Boevey
\cite{cra,cra1}. The authors of \cite{gab_vos}
got a new proof of the Tame--Wild Theorem for
matrix problems given by modules over aggregates
and studied a geometric structure of the set of
nonisomorphic indecomposable matrices.

The problem of classifying pairs of matrices up
to simultaneous similarity (i.e. representations
of the quiver \!\!
\unitlength 0.5mm \linethickness{0.4pt}
\begin{picture}(15,0)(-7,-1.5)
\put(0.00,0.00){\circle*{0.8}}
\put(-5.53,0.00){\oval(4.00,4.00)[l]}
\bezier{16}(-5.47,1.93)(-3.67,2.00)(-1.27,0.93)
\put(-1.20,-0.67){\vector(2,1){0.2}}
\bezier{20}(-5.47,-2.00)(-3.20,-1.87)(-1.20,-0.67)
\put(5.53,0.00){\oval(4.00,4.00)[r]}
\bezier{16}(5.47,1.93)(3.67,2.00)(1.27,0.93)
\put(1.20,-0.67){\vector(-2,1){0.2}}
\bezier{20}(5.47,-2.00)(3.20,-1.87)(1.20,-0.67)
\end{picture}
\!) is used as a measure of complexity since it
`contains' a lot of matrix problems, in
particular, the problem of classifying
representations of every quiver. For instance,
the classes of isomorphic representations of the
quiver \eqref{0.01} correspond, in a one-to-one
manner, to the classes of similar pairs of the
form
\begin{equation}       \label{0.02}
\left(\begin{bmatrix} I&0&0&0\\ 0&2I&0&0\\
0&0&3I&0\\ 0&0&0&4I
\end{bmatrix},
\begin{bmatrix}
A_{\alpha}&0&0&0\\ A_{\beta}&0&0&0\\
A_{\gamma}&0&0&0  \\ A_{\delta}&
A_{\varepsilon}&I& A_{\zeta}
\end{bmatrix}\right).
\end{equation}
Indeed, if $(J,A)$ and $(J,A')$ are two similar
pairs of the form \eqref{0.02}, then
$S^{-1}JS=J,\ S^{-1}AS=A'$, the first equality
implies $S=S_1\oplus S_2\oplus S_3\oplus S_4$ and
equating the (4,3) blocks in the second equality
gives $S_3=S_4$ (compare with Example
\ref{e0.1}).

Let $A_1,\dots,A_p\in k^{m\times m}$. For a
parametric matrix
$M(\lambda_1,\dots,\lambda_p)=[a_{ij}+b_{ij}
\lambda_1+\dots+d_{ij}\lambda_p]$
($a_{ij},b_{ij},\dots,d_{ij}\in k$), the matrix
that is obtained by replacement of its entries
with $a_{ij}I_m+b_{ij} A_1+\dots+d_{ij}A_p$ will
be denoted by $M(A_1,\dots,A_p)$.

In this section, we get the following
strengthened form of the Tame--Wild Theorem,
which is based on an explicit description of the
set of canonical matrices.

\begin{theorem}      \label{t0.1}
Every linear matrix problem satisfies one and
only one of the following two conditions
(respectively, is of tame or wild type):
\begin{itemize}
\item[(I)] For every step-sequence $\underline{n}$,
the set of indecomposable canonical matrices in
the affine space of $\underline{n} \times
\underline{n}$ matrices consists of a finite
number of points and straight
lines%
\footnote{ Contrary to \cite{gab_vos}, these
lines are unpunched. Thomas Br\"ustle and the
author proved in [\emph{Linear Algebra Appl.} 365 (2003) 115--133] that the number of points and lines
is bounded by $4^d$, where $d={\text{dim}}({\cal
M}_{\underline{n} \times \underline{n}})$. This
estimate is based on an explicit form of
canonical matrices given in the proof of Theorem
\ref{t0.1} and is an essential improvement of the
estimate \cite{bru}, which started from the
article \cite{gab_vos}. } of the form
$\{L(J_m(\lambda))\,|\, \lambda\in k\}$, where
$L(x)=[a_{ij}+xb_{ij}]$ is a one-parameter
${\underline l}\times {\underline l}$ matrix
($a_{ij},b_{ij}\in k$, ${\underline
l}={\underline n}/m$) and $ J_m(\lambda)$ is the
Jordan cell. Changing $m$ gives a new line of
indecomposable canonical matrices
$L(J_{m'}(\lambda))$; there exists an integer $p$
such that the number of points of
intersections%
\footnote{ Hypothesis: this number is equal to 0.
} of the line $L(J_m(\lambda))$ with other lines
is $p$ if $m>1$ and $p$ or $p+1$ if $m=1$.

\item[(II)] There exists a two-parameter
$\underline{n}\times \underline{n}$ matrix
$P(x,y)= [a_{ij}+xb_{ij}+ yc_{ij}]$
($a_{ij},b_{ij},c_{ij}\in k$) such that the plane
$\{P(a,b)\,|\, a,b\in k\}$ consists only of
indecomposable canonical matrices. Moreover, a
pair $(A,B)$ of $m\times m$ matrices is in the
canonical form with respect to simultaneous
similarity if and only if $P(A,B)$ is a canonical
$m\underline{n} \times m\underline{n}$ matrix.
\end{itemize}

\end{theorem}

We will prove Theorem \ref{t0.1} analogously to
the proof of the Tame--Wild Theorem in
\cite{dro1}: We reduce an indecomposable
canonical matrix $M$ to canonical form (making
additional partitions into blocks) and meet a
free (in the sense of Section \ref{s3.3'}) block
$P$ that is reduced by similarity
transformations. If there exist infinitely many
values of eigenvalues of $P$ for which we cannot
simultaneously make zero all free blocks after
$P$, then the matrix problem satisfies the
condition (II). If there is no matrix $M$ with
such a block $P$, then the matrix problem
satisfies the condition (I). We will consider the
first case in Section \ref{s4.3} and the second
case in Section \ref{s4.4}. Two technical lemmas
are proved in Section \ref{s4.2}.


\subsection{Two technical lemmas}   \label{s4.2}

In this section we get two lemmas, which will be
used in the proof of Theorem \ref{t0.1}.

\begin{lemma}            \label{l4.1}
Given two matrices $L$ and $R$ of the form $
L=\lambda I_m+F$ and $R=\mu I_n+G$ where $F$ and
$G$ are nilpotent upper triangular matrices.
Define
\begin{equation}      \label{4.1}
 A^f=\sum_{ij} a_{ij}L^iAR^j
\end{equation}
for every $A\in k^{m\times n}$ and
$f(x,y)=\sum_{i,j\ge 0} a_{ij}x^iy^j\in k[x,y]$.
Then
\begin{itemize}
\item[(i)] $(A^f)^g=A^{fg}=(A^g)^f$;

\item[(ii)] $A^f=\sum b_{ij}F^iAG^j,$  where
$b_{00}=f(\lambda,\mu),\ b_{01}=\frac{\partial
f}{\partial y} (\lambda,\mu),\dots$;

\item[(iii)]
if $f(\lambda,\mu)=0$, then the left lower entry
of $A^f$ is $0$;

\item[(iv)]
if $f(\lambda,\mu)\neq 0$, then for every
$m\times n$ matrix $B$ there exists a unique $A$
such that $A^f=B$ (in particular, $B=0$ implies
$A=0$).
\end{itemize}
\end{lemma}

\begin{proof}
(ii) $A^f=\sum a_{ij}(\lambda I+F)^iA(\mu I+G)^j=
\sum a_{ij}\lambda^i\mu^jA+\sum a_{ij}\lambda^i
j\mu^{j-1}AG+\cdots.$

(iii) It follows from (ii).

(iv) Let $f(\lambda,\mu)\ne 0$ and $A\in
k^{m\times n}$. By (ii), $B:=A^f= \sum
b_{ij}F^iAG^j,$  where $b_{00}= f(\lambda,\mu)$.
Then $A=b_{00}^{-1}[B-\sum_{i+j\ge 1}
b_{ij}F^iAG^j].$ Substituting this equality in
its right-hand side gives $$ A=b_{00}^{-1}B-
b_{00}^{-2}[\sum_{i+j\ge 1} b_{ij}F^iBG^j-
\sum_{i+j\ge 2} c_{ij}F^iAG^j]. $$ Repeating this
substitution $m+n$ times, we eliminate $A$ on the
right since $F^m=G^n=0$ (recall that $F$ and $G$
are nilpotent).
 \end{proof}

\begin{lemma}            \label{l4.2}
Given a polynomial $p\times t$ matrix $[f_{ij}]$,
$f_{ij}\in k[x,y]$, and an infinite set
${D}\subset k\times k$. For every
$l\in\{0,1,\dots, p\}$, $(\lambda,\mu)\in {D}$,
and ${\cal F}_l=\{m,n,F,G,N_1,\dots,N_l\},$ where
$F\in k^{m\times m}$ and $G\in k^{n\times n}$ are
nilpotent upper triangular matrices and
$N_1,\dots,N_l\in k^{m\times n}$, we define a
system of matrix equations
\begin{equation}                 \label{4.2}
{\cal S}_l={\cal S}_l(\lambda,\mu,{\cal
F}_l):\quad X_1^{f_{i1}}+\dots+
X_t^{f_{it}}=N_i,\quad i=1,\dots,l,
\end{equation}
(see \eqref{4.1}) that is empty if $l=0$.
Suppose, for every $(\lambda,\mu)\in {D}$ there
exists ${\cal F}_p$ such that the system ${\cal
S}_p$ is unsolvable.

Then there exist an infinite set ${D}' \subset
{D}$, a polynomial $d\in k[x,y]$ that is zero on
${D}'$,  a nonnegative integer $w\le \min
(p-1,t)$, and pairwise distinct
$j_1,\dots,j_{t-w}\in\{1,\dots, t\}$ satisfying
the conditions:
\begin{itemize}
\item[(i)]
For each $(\lambda,\mu)\in {D}'$ and ${\cal
F}_w$, the system ${\cal S}_w(\lambda,\mu,{\cal
F}_w)$ is solvable and every $(t-w)$-tuple
$S_{j_1},\, S_{j_2},\dots, S_{j_{t-w}}\in
k^{m\times n}$ is uniquely completed to its
solution $(S_1,\dots, S_t)$.

\item[(ii)]
For each $(\lambda,\mu)\in {D}'$, ${\cal F}_w^0
=\{m,n,F,G,0,\dots,0\}$, and for every solution
$(S_1,\dots, S_t)$ of ${\cal S}_w(\lambda,\mu,
{\cal F}_w^0)$, there exists a matrix $S$ such
that
\begin{equation}      \label{4.5'}
S_1^{f_{w+1,1}}+\dots+ S_t^{f_{{w+1},t}} = S^d.
 \end{equation}
\end{itemize}
\end{lemma}

\begin{proof}
Step-by-step, we will simplify the system ${\cal
S}_p(\lambda,\mu, {\cal F}_p)$ with
$(\lambda,\mu)\in D$.

{\it The first step.} Let there exist a
polynomial $f_{1j}$, say $f_{1t}$, that is
nonzero on an infinite set ${D}_1\subset{D}$. By
Lemma \ref{l4.1}(iv), for each
$(\lambda,\mu)\in{D}_1$ and every
$X_1,\dots,X_{t-1}$ there exists a unique $X_t$
such that the first equation of \eqref{4.2}
holds. Subtracting the $f_{it}$th power of the
first equation of (\ref{4.2}) from the $f_{1t}$th
power of the $i$th equation of (\ref{4.2}) for
all $i>1$, we obtain the system
\begin{equation}      \label{4.6}
X_1^{g_{i1}}+\dots+
X_{t-1}^{g_{i,t-1}}=N^{f_{1t}}_i- N_1^{f_{it}},
\quad 2\le i\le l,
\end{equation}
where $g_{ij}=f_{ij}f_{1t}-f_{1j}f_{it}.$ By
Lemma 4.1(iv), the system ${\cal S}_p$ and the
system \eqref{4.6} supplemented by the first
equation of ${\cal S}_p$ have the same set of
solutions for all $(\lambda,\mu)\in {D}_1$ and
all ${\cal F}_p$.

{\it The second step.} Let there exist a
polynomial $g_{2j}$, say $g_{2,t-1}$, that is
nonzero on an infinite set ${D}_2\subset{D}_1$.
We eliminate $X_{t-1}$ from the equations
(\ref{4.6}) with $3\le i\le l$.

{\it The last step.} After the $w$th step, we
obtain a system
\begin{equation}      \label{4.7}
\left.\begin{matrix} X_{j_1}^{r_1}+\dots+
X_{j_{t-w}}^{r_{t-w}}= N\\ \hdotsfor[1.5]{1}
\end{matrix}\ \right\}
\end{equation}
(empty if $w=t$) and an infinite set ${D}_w$ such
that the projection
$$
(S_1,\dots,S_t)\mapsto (S_{j_1},\dots
S_{j_{t-w}})
$$
is a bijection of the set of solutions of the
system ${\cal S}_p(\lambda,\mu, {\cal F}_p)$ into
the set of solutions of the system \eqref{4.7}
for every $(\lambda,\mu)\in {D}_w$.

Since for every $(\lambda,\mu)\in {D}$ there
exists ${\cal F}_p$ such that the system ${\cal
S}_p$ is unsolvable, the process stops on the
system \eqref{4.7} with $w<p$ for which either

(a) there exists $r_i\ne 0$ and
$r_1(\lambda,\mu)=\dots= r_{t-w}(\lambda,\mu)=0$
for almost all  $(\lambda,\mu)\in{D}_w$, or

(b) $r_1=\dots=r_{t-w}=0$ or  $w=t$.

We add the $(w+1)$st equation
$$
X_1^{f_{w+1,1}}+\dots+ X_t^{f_{w+1,t}}=
X_1^{f_{w+1,1}}+\dots+ X_t^{f_{w+1,t}}
$$
to the system ${\cal S}_w(\lambda,\mu, {\cal
F}_w^0)$ with $(\lambda,\mu)\in{D}_w$ and ${\cal
F}_w^0 =\{m,n,F,G,0,\dots,0\}$ and apply the $w$
steps; we obtain the equation
\begin{equation}      \label{4.5}
X_{j_1}^{r_1}+\dots+ X_{j_{t-w}}^{r_{t-w}}
=(X_1^{f_{w+1,1}}+\dots+ X_t^{f_{w+1,t}})^
{\varphi}
\end{equation}
where $r_1,\dots,r_{t-w}$ are the same as in
\eqref{4.7} and $\varphi(\lambda,\mu)\ne 0$.
Clearly, the solutions $(S_1,\dots,S_t)$ of
${\cal S}_w(\lambda,\mu, {\cal F}_w^0)$ satisfy
\eqref{4.5}; moreover,
\begin{equation}      \label{4.5a}
(S_{j_1}^{\rho_1}+\dots+
S_{j_{t-w}}^{\rho_{t-w}})^d
=(S_1^{f_{w+1,1}}+\dots+ S_t^{f_{w+1,t}})^
{\varphi}
 \end{equation}
for $(\lambda,\mu)\in{D}'$, where
${\rho_1},\dots,{\rho_{t-w}}, d\in k[x,y]$ and
${D}'$ define as follows: In the case (a),
$r_1,\dots, r_{t-w}$ have a common divisor
$d(x,y)$ with infinitely many roots in ${D}_w$
(we use the following form of the Bezout theorem
\cite[Sect. 1.3]{gri}: two relatively prime
polynomials $f_1,f_2\in k[x,y]$ of degrees $d_1$
and $d_2$ have no more than $d_1d_2$ common
roots); we put $\rho_i=r_i/d$ and
${D}'=\{(\lambda,\mu)\in{D}_w\,|\,
d(\lambda,\mu)=0\}$. In the case (b), the
left-hand side of \eqref{4.5} is zero; we put
$\rho_1=\dots=\rho_{t-w}=0$ (if $w<t$), $d=0$,
and ${D}'={D}_w$.

We take $(\lambda,\mu)\in{D}'$ and put ${\bar
\varphi}(x,y)=\varphi(x+ \lambda,y+ \mu)$. Since
${\bar \varphi}(0,0)= \varphi(\lambda,\mu)\ne 0$,
there exists $\bar\psi\in k[x,y]$ for which
${\bar\varphi}{\bar\psi}\equiv 1 \mod (x^s,y^s)$,
where $s$ is such that $F^s=G^s=0$. We put
$\psi(x,y)={\bar\psi(x-\lambda, y-\mu)}$, then
$A^{\varphi\psi}=A$ for every $m\times n$ matrix
$A$. By \eqref{4.5a},
$$
S_1^{f_{w+1,1}}+\dots+ S_t^{f_{w+1,t}}=
(S_{j_1}^{\rho_1}+\dots+
S_{j_{t-w}}^{\rho_{t-w}})^{\psi d};
$$
it proves \eqref{4.5'}.
\end{proof}

\subsection{Proof of Theorem \ref{t0.1} for
wild problems} \label{s4.3}

A subblock of a free (dependent) block will be
named a {free $($dependent$)$ subblock}. In this
section, we consider a matrix problem given by a
pair $(\varGamma,\, \cal{M})$ such that there
exists a semi-parametric canonical matrix $M\in
{\cal M}_{\underline{n}\times \underline{n}}$
having a free box $M_q\neq \emptyset $ with the
following property:
\begin{equation}                 \label{4.8}
\parbox{25em}
{The horizontal or the vertical $(q-1)$-strip of
$M_q$ is linked (see Definition \ref{d2.1}) to a
$(q-1)$-strip containing an {\it infinite
parameter} from a free box $M_v,\ v<q$,  (i.e.,
the domain of parameters contains infinitely many
vectors with distinct values of this parameter).}
\end{equation}

We choose such $M\in {\cal
M}_{\underline{n}\times \underline{n}}$ having
the smallest $\sum \underline{n}=n_1+n_2+\cdots$
and take its free box $M_q\neq \emptyset $ that
is the first with the property \eqref{4.8}. Then
each $(q-1)$-strip of $M$ is linked to the
horizontal or the vertical $(q-1)$-strip
containing $M_q$. Our purpose is to prove that
the matrix problem satisfies the condition (II)
of Theorem \ref{t0.1}.   Let each of  the boxes
$M_q,\ M_{q+1},\dots$ that is free be replaced by
$0$, and let as many as possible parameters in
the boxes $M_1,\dots, M_{q-1}$ be replaced by
elements of $k$ (correspondingly we retouch
dependent boxes and narrow down the domain of
parameters ${\cal D}$) such that the property
\eqref{4.8} still stands (note that all the
parameters of a ``new" semi-parametric canonical
matrix $M$ are infinite and that $M_q=0$ but
$M_q\neq \emptyset $). The following three cases
are possible: {\it
     \begin{description}
     \item[\it Case 1:]
The horizontal and the vertical $(q-1)$-strips of
$M_q$ are linked to $(q-1)$-strips containing
distinct parameters $\lambda_l$ and $\lambda_r$
respectively.
      \item[\it Case 2:]
The horizontal or the vertical $(q-1)$-strip of
$M_q$ is linked to no $(q-1)$-strips containing
parameters.
      \item[\it Case 3:]
The horizontal and the vertical $(q-1)$-strips of
$M_q$ are linked to $(q-1)$-strips containing the
same parameter $\lambda$.
       \end{description}}


\subsubsection{Study Case 1}  \label{s4.3.1}

By Theorem \ref{t3.1}, the minimality of $\sum
\underline{n}$, and since each $(q-1)$-strip of
$M$ is linked to a $(q-1)$-strip containing
$M_q$, we have that $M$ is a two-parameter matrix
(hence $l,r\in\{1,2\}$) and, up to permutation of
$(q-1)$-strips, it has the form
$\hat{H}_l\oplus\hat{H}_r$, where
$\hat{H}_l=H_l(J_{s_l}(\lambda_l I))$ and
$\hat{H}_r= H_r(J_{s_r}(\lambda_r I))$ lie in the
intersection of all $(q-1)$-strips linked to the
horizontal and, respectively, the vertical
$(q-1)$-strips of $M_q$, $H_l(a)$ and $H_r(a)$
are indecomposable canonical matrices for all
$a\in k$, and
$$ J_{s}(\lambda I):= \begin{bmatrix} \lambda
I&I&&\text{\Large 0}\\
 &\lambda I&\ddots& \\ &&\ddots&I  \\
\text{\Large 0}&&&\lambda I
               \end{bmatrix}.
$$

We will assume that the parameters $\lambda_1$
and $\lambda_2$ are enumerated such that the free
boxes $M_u$ and $M_v$ containing $\lambda_1$ and,
respectively, $\lambda_2$ satisfy $u\le v$
(clearly, $M_u$ and $M_v$ are Weyr matrices).

Let first $u<v$. Then
\begin{equation}       \label{4.8'}
M_u=A\oplus J_{s_1}(\lambda_1 I) \oplus B,
 \quad
M_v= J_{s_2}(\lambda_2 I),
\end{equation}
where $A$ and $B$ lie in $\hat{H}_2$ ($M_v$ does
not contain summands from $\hat{H}_1$ since every
box $M_i$ with $i>u$ that is reduced by
similarity transformations belongs to $\hat{H}_1$
or $\hat{H}_2$).

By the
$\underline{n}^{\star}\times\underline{n}^{\star}$
{\it partition} of $M$ into blocks
$M_{ij}^{\star}$ (which will be called
$\star$-{\it blocks} and the corresponding strips
will be called $\star$-{\it strips}), we mean the
partition obtained from the partition into
$(v-1)$-strips by removing the divisions inside
of $J_{s_1}(\lambda_1 I)$ and the corresponding
divisions inside of the horizontal and vertical
$(u-1)$-strips of $M_u$ and inside of all
$(u-1)$-strips that are linked with them.
Clearly, $ J_{s_1}(\lambda_1 I)$ and
$J_{s_2}(\lambda_2 I)$ are free $\star$-blocks,
the other $\star$-blocks are zero or scalar
matrices, and $M_q$ is a part of a $\star$-block.
Denote by ${\cal I}$ (resp. ${\cal J}$)  the set
of indices of $\star$-strips of $\hat{H}_l$
(resp.  $\hat{H}_r$) in
$M=[M_{ij}^{\star}]_{i,j=1}^e$, then ${\cal
I}\cup {\cal J}= \{1,\dots,e\}$ and ${\cal I}\cap
{\cal J}=\varnothing$.

\begin{step}[a selection of
$M^{\star}_{\zeta \eta}$]  \label{s1}
 On this step we will select both a free
$\star$-block $M^{\star}_{\zeta \eta}>M_v$ with
$(\zeta ,\eta) \in{\cal I}\times{\cal J}$ and an
infinite set of $(a,b)\in {\cal D}$ such that
$M^{\star}_{\zeta \eta}$  cannot be made
arbitrary by transformations of $M(a,b)$
preserving all $M_1,\dots,M_v$ and all
$M^{\star}_{ij}<M^{\star}_{\zeta \eta}$. Such
$M^{\star}_{\zeta \eta}$ exists since $M_q\neq
\emptyset $ is a part of a free $M^{\star}_{ij}$
with $(i,j)\in {\cal I}\times{\cal J}$.

Denote by $\Lambda_0$ the algebra of all $S$ from
$\Lambda:={\varGamma}_{\underline{n}\times
\underline{n}}$ for which $MS$ and $SM$ are
coincident on the places of the boxes
$M_1,\dots,M_v$ (see \eqref{9''}). Then the
transformations
\begin{equation}        \label{4.10}
M \longmapsto M'=SMS^{-1},\quad S\in \Lambda_0^*,
\end{equation}
preserve $M_1,\dots,M_v$. Note that $\Lambda_0$
is an algebra of upper block-triangular
${\underline{n}^{\star}\times
\underline{n}}^{\star}$ (and even
${\underline{n}_v\times \underline{n}}_v$)
matrices.

Let $M^{\star}_{\zeta \eta}$ be selected and let
$S\in\Lambda_0^*$ be such that the transformation
\eqref{4.10} preserves all
$M^{\star}_{ij}<M^{\star}_{\zeta \eta}$. Equating
the $(\zeta ,\eta)$ $\star$-blocks in the
equality $M'S=SM$ gives
\begin{equation}       \label{4.15}
M_{\zeta 1}^{\star}S_{1\eta}^{\star}+\dots+
M_{\zeta ,\eta-1}^{\star}S_{\eta-1,\eta}^{\star}+
M_{\zeta \eta}^{\star\prime}S_{\eta\eta}^{\star}
= S_{\zeta \zeta }^{\star}M_{\zeta
\eta}^{\star}+\dots +S_{\zeta e}^{\star}M_{e
\eta}^{\star},
\end{equation}
where $e\times e$ is the number of $\star$-blocks
in $M$. Since
\begin{equation}       \label{4.16}
M_{ij}^{\star}\ne 0 \quad {\rm implies}\quad
(i,j)\in ({\cal I} \times {\cal I})\cup ({\cal
J}\times{\cal J}),
\end{equation}
the equality \eqref{4.15} may contain
$S_{ij}^{\star}$ only if $(i,j)\in {\cal I}
\times {\cal J}$ or $(i,j)=(\eta, \eta)$, hence $
M_{\zeta \eta}^{\star\prime }$ is fully
determined by $M$, $S_{\eta\eta}^{\star}$ and the
family of ${\star}$-blocks $$ S_{\cal{I
J}}^{\star}:= \{S_{ij}^{\star}\,|\,(i,j)\in{\cal
I}\times {\cal J} \}.$$

We will select $M^{\star}_{\zeta \eta}$ in the
sequence
\begin{equation}       \label{7.2}
F_1<F_2<\dots<F_{\delta}
\end{equation}
of all free $M^{\star}_{ij}$ such that $(i,j)\in
{\cal I}\times{\cal J}$ and
$M^{\star}_{ij}\not\subset M_1\cup\dots \cup
M_v$. For $\alpha\in \{1,\dots,\delta\}$ denote
by $\Lambda_{\alpha}$ the algebra of all
$S\in\Lambda_0$ for which $MS$ and $SM$ coincide
on the places of all $M^{\star}_{ij}\le
F_{\alpha}$. Then the transformations
\begin{equation}        \label{7.1}
M \longmapsto M'=SMS^{-1},\quad S\in
\Lambda_{\alpha}^*,
\end{equation}
preserve $M_1,\dots,M_v$ and all
$M^{\star}_{ij}\le F_{\alpha}$.

Let us investigate the family $S_{\cal{I
J}}^{\star}$ for each $S\in \Lambda_{\alpha}^*$.

The algebra
$\Lambda=\varGamma_{\underline{n}\times
\underline{n}}$ consists of all
${\underline{n}\times \underline{n}}$ matrices
$S=[S_{ij}]$ whose blocks satisfy a system of
linear equations of the form \eqref{2}--\eqref{3}
completed by $S_{ij}=0$ for all $i>j$. Let us
rewrite this system for smaller $\star$-blocks
$S_{ij}^{\star}$. The equations that contain
blocks from the family $S_{\cal{I J}}^{\star}$
contain no blocks $S_{ij}^{\star}\notin S_{\cal{I
J}}^{\star}$. (Indeed, by the definition of Case
1, the $(q-1)$-strips of $\hat{H}_1$ are not
linked to the $(q-1)$-strips of $\hat{H}_2$, so
the partition of $T$ into ${\cal I}$ and ${\cal
J}$ is in agreement with its partition $T/\!\sim$
into equivalence classes; see \eqref{0} and
Definition \ref{d2.1}.) Hence the family
$S_{\cal{I J}}^{\star}$ for $S\in
\Lambda^{\star}$ is given by a system of
equations of the form
\begin{equation}       \label{4.17}
\sum_{(i,j)\in{\cal I}\times{\cal J}}
\alpha_{ij}^{(\tau)}S_{ij}^{\star}=0,\quad \tau
=1,\dots,w_1.
\end{equation}

Denote by ${\cal B}_u$ (resp. ${\cal B}_{uv}$)
the part of $M$ consisting of all entries that
are in the intersection of $\cup_{i\le u}M_i$
(resp. $\cup_{u<i\le v}M_i$; by the union of
boxes we mean the part of the matrix formed by
these boxes) and $\cup_{(i,j)\in{\cal
I}\times{\cal J}}M_{ij}^{\star}$. Let us prove
that ${\cal B}_u$ and ${\cal B}_{uv}$ are unions
of $\star$-blocks $M_{ij}^{\star},\ (i,j)\in{\cal
I}\times{\cal J}$. It is clear for ${\cal B}_u$
since the partition into $\star$-strips is a
refinement of the partition into $(u-1)$-strips.
It is also true for ${\cal B}_{uv}$ since ${\cal
B}_{uv}$ is partitioned into rectangular parts
(see Definition \ref{d2.01}) of the form
$[M_{\tau+1}|M_{\tau+2}|\cdots| M_{\tau+s_1}]$ if
$r=1$ and $[M_{\tau_1}^T|M_{\tau_2}^T|\cdots|
M_{\tau_{{\scriptstyle s}_1}}^T]^T$ if $l=1$ (the
indices $l$ and $r$ were defined in the
formulation of Case 1); recall that all $M_i$ are
boxes and $M_u$ has the form \eqref{4.8'}.

By the definition of the algebra  $\Lambda_0$, it
consists of all $ S\in\Lambda$ such that $MS-SM$
is zero on the places of the boxes $M_i\le M_v$.
To obtain the conditions on the family
$S_{\cal{IJ}}^{\star}$ of blocks of $S\in
\Lambda_0$, by virtue of the statement
\eqref{4.16}, it suffices to equate zero the
blocks of $MS-SM$ on the places of all free
$\star$-blocks $M_{ij}^{\star}$ from ${\cal B}_u$
and ${\cal B}_{uv}$ (note that some of them may
satisfy $M_{ij}^{\star}>M_{\zeta \eta}^{\star}$).
Since all free $\star$-blocks of $M$ except for
$J_{s_1}(\lambda_1 I)$ and $J_{s_2}(\lambda_2 I)$
are scalar or zero matrices, we obtain a system
of equalities of the form
\begin{equation}       \label{4.18}
\sum_{(i,j)\in{\cal I}\times{\cal J}}
\alpha_{ij}^{(\tau)}S_{ij}^{\star}=0,\quad
\tau=w_1+1,\dots,w_2,
\end{equation}
for the places from ${\cal B}_u$ and
\begin{equation}       \label{4.19}
\sum_{(i,j)\in{\cal I}\times{\cal J}}
(S_{ij}^{\star})^{g_{ij}^{(\nu)}}=0,\quad
\nu=1,\dots,w_3,
\end{equation}
for the places from ${\cal B}_{uv}$, where
$(S_{ij}^{\star})^{g_{ij}^{(\nu)}}$,
$g_{ij}^{(\nu)}\in k[x,y]$ (more precisely,
$g_{ij}^{(\nu)}\in k[x]$ if $l=1$ and
$g_{ij}^{(\nu)}\in k[y]$ if $r=1$), are given by
(\ref{4.1}) with $L=J_{s_l}(\lambda_l I)$ and
$R=J_{s_r}(\lambda_r I)$.

Applying the Gauss-Jordan elimination algorithm
to the system \eqref{4.17}--\eqref{4.18}, we
choose $S_1,\dots,S_t\in S_{\cal{IJ}}^{\star}$
such that they are arbitrary and the other
$S_{ij}^{\star}\in S_{\cal{IJ}}^{\star}$ are
their linear combinations. Rewriting the system
\eqref{4.19} for $S_1,\dots,S_t$, we obtain a
system of equalities of the form
\begin{equation}       \label{4.20}
S_1^{f_{i1}}+\dots+ S_t^{f_{it}}=0,\quad
i=1,\dots,w_3.
\end{equation}

The algebra $\Lambda_{\alpha}$ $(1\le \alpha \le
\delta)$ consists of all $S\in \Lambda_{0}$ such
that $SM$ and $MS$ have the same blocks on the
places of all free $M_{xy}^{\star}\le
F_{\alpha}$:
\begin{equation}          \label{4.21}
M_{ x1}^{\star}S_{1y}^{\star}+\dots+ M_{
xy}^{\star}S_{yy}^{\star} =
S_{xx}^{\star}M_{xy}^{\star}+
\dots+S_{xe}^{\star}M_{ey}^{\star}.
\end{equation}
We may omit the equalities \eqref{4.21} for all
$(x,y)$ such that $M_{xy}^{\star}$ is contained
in $M_1,\dots,M_v$ (by the definition of
$\Lambda_{0}$), or  $(x,y)\notin \cal I\times
\cal J$ (by \eqref{4.16}, the equality
\eqref{4.21} contains $S_{ij}^{\star}\in
S_{\cal{IJ}}^{\star}$ only if $(x,y)\in \cal
I\times \cal J$). The remaining equalities
\eqref{4.21} correspond to (zero)
$M_{xy}^{\star}\in\{F_1,\dots, F_{\alpha}\}$ and
take the form
\begin{equation}       \label{4.22}
S_1^{f_{i1}}+\dots+ S_t^{f_{it}}=0,\quad
i=w_3+1,\dots,w_3+\alpha.
\end{equation}

It follows from the preceding that any sequence
of matrices $S_1,\dots,S_t$ is the sequence of
corresponding blocks of a matrix $S\in
\Lambda_{\alpha}$ if and only if the system
\eqref{4.20}$\cup$\eqref{4.22} holds for
$S_1,\dots,S_t$.

Put $\alpha=\delta$  (see \eqref{7.2}),
$p=w_3+\delta$, and $D=\{(a_l,a_r)\,|\,
(a_1,a_2)\in {\cal D}\}$. Since $M_q\neq
\emptyset $ is a part of a free $M^{\star}_{ij}$
with $(i,j)\in {\cal I}\times{\cal J}$, for every
$(a_l,a_r)\in  D$ we may change the right-hand
part of the system \eqref{4.20}$\cup$\eqref{4.22}
to obtain an unsolvable system. Applying Lemma
\ref{l4.2} to the system
\eqref{4.20}$\cup$\eqref{4.22}, we get an
infinite ${D}' \subset {D}$, a polynomial $d\in
k[x,y]$ that is zero on ${D}'$, a nonnegative
integer $w\le \min(p-1,t)$, and pairwise distinct
$j_1,\dots,j_{t-w}\in\{1,\dots, t\}$ satisfying
the conditions (i)--(ii) of Lemma \ref{l4.2}.  We
take $F_{w+1-w_3}$ as the desired block
$M^{\star}_{\zeta \eta}$. Since $M_q$ is the
first among free boxes $\neq \emptyset$ with the
property \eqref{4.8}, $M^{\star}_{\zeta
\eta}>M_v$. The equality \eqref{4.15} takes the
form
\begin{equation}       \label{4.23}
S_1^{f_{w+1,1}}+\dots+ S_t^{f_{w+1,t}}=
 S_{\zeta \zeta }^{\star}M_{\zeta \eta}^{\star}-
M_{\zeta \eta}^{\star\prime}
S_{\eta\eta}^{\star}.
\end{equation}
\end{step}

\begin{step}[a construction of $P(x,y)$]
   \label{s2}
On this step, we construct the two-para\-meter
matrix $P(x,y)$ from the condition (II) of
Theorem \ref{t0.1}.

Let us fix a pair $(a_l,a_r)\in {D}'$ in the
following manner. If the polynomial $d\in k[x,y]$
is zero, then $(a_l,a_r)$ is an arbitrary pair
from ${D}'$. Let $d\ne 0$; if $d$ is reducible,
we replace it by its irreducible factor. Since
$d$ is zero on the infinite set $D'$ that does
not contain infinitely many pairs $(a_l,a_r)$
with the same $a_l$ (otherwise, the $l$th
parameter can be replaced with $a_l$, but we have
already replaced as many as possible parameters
by elements of $k$ such that the property
\eqref{4.8} still stands), it follows $d\notin
k[x]$ and so $d'_y:=\partial d/\partial y \ne 0$.
Since $d$ is an irreducible polynomial,
$(d,d'_y)=1$; by the Bezout theorem (see the
proof of Lemma \ref{l4.2}), we may chose
$(a_l,a_r)\in{D}'$ such that
\begin{equation}      \label{4.24}
    d(a_l,a_r)=0, \quad d'_y(a_l,a_r)\neq 0.
\end{equation}

Denote by $P(x,y)$ the matrix that is obtained
from $M$ by replacement of its $\star$-blocks
$J_{s_l}(\lambda_l I)$ and $J_{s_r}(\lambda_r I)$
with Weyr matrices
\begin{equation}       \label{4.12}
L:={\varPi} (J_1(a_l)\oplus J_3(a_l) \oplus
J_5(a_l)\oplus J_7(a_l)\oplus J_9(a_l))
{\varPi}^{-1},\ R:=J_5(a_rI_2)
\end{equation}
(where ${\varPi}$ is a permutation matrix, see
Theorem \ref{t2'.1}) and the $\star$-block
$M_{\zeta \eta}^{\star}$ with
\begin{equation}       \label{4.13}
P_{\zeta  \eta}^{\star}= {\varPi}\begin{bmatrix}
Q_1 \\ Q_2
\\ Q_3
\\ Q_4 \\ Q_5
\end{bmatrix}, \ \
Q_i=
\begin{bmatrix}
0&0&0&0&0 \\[-7pt] \hdotsfor{5} \\ 0&0&0&0&0
 \\ T_i&0&0&0&0 \\
0&0&0&0&0 \\[-7pt] \hdotsfor{5} \\ 0&0&0&0&0
\end{bmatrix}, \ \
T=\begin{bmatrix} T_1 \\ T_2 \\ T_3 \\ T_4 \\ T_5
\end{bmatrix}=
\begin{bmatrix}
1&y \\ 1&x \\ 1&1 \\ 1&0 \\ 0&1
\end{bmatrix},
\end{equation}
where $Q_i$ is $(2i-1)$-by-$10$ (its zero blocks
are 1-by-2) and $T_i$ is in the middle row. (Each
nonzero free $\star$-block $M_{ij}^{\star}$ of
$M$, except for $J_{s_l}(\lambda_l I)$ and
$J_{s_r}(\lambda_r I)$, is a scalar matrix with
$(i,j)\in({\cal I}\times{\cal I})\cup({\cal
J}\times{\cal J})$; it is replaced by the scalar
matrix $P_{ij}^{\star}$ with the same diagonal
having the size $(1+3+5+7+9)\times (1+3+5+7+9)$
if $(i,j)\in {\cal I}\times\cal I$ and $10\times
10$ if $(i,j)\in {\cal J}\times\cal J$.) The
dependent blocks are respectively corrected by
formulas \eqref{3'.1'}.

Let us enumerate the rows and columns of $J
=J_1(a_l)\oplus J_3(a_l)\oplus J_5(a_l)\oplus
J_7(a_l)\oplus J_9(a_l)$ and the rows of
$[Q_i]_{i=1}^5$ by the pairs of numbers $\langle
1,1\rangle;$ $ \langle 3,1\rangle,$ $\langle
3,2\rangle,$ $ \langle 3,3\rangle;$ $ \langle
5,1\rangle,$ $ \langle 5,2\rangle,\dots, \langle
5,5\rangle;\dots; \langle 9,1\rangle,$ $ \langle
9,2\rangle,\dots, \langle 9,9\rangle $. Going
over to the matrix $P$, we have permuted them in
$L={\varPi}J{\varPi}^{-1}$ and $P_{\zeta
\eta}^{\star}={\varPi}[Q_i]$ in the following
order:
\begin{equation}       \label{4.14}
\begin{matrix}
\langle 9,1\rangle, \langle 7,1\rangle, \langle
5,1\rangle, \langle 3,1\rangle, \langle
1,1\rangle, \langle 9,2\rangle, \langle
7,2\rangle, \langle 5,2\rangle, \langle
3,2\rangle,
\\ \langle 9,3\rangle, \langle 7,3\rangle, \langle
5,3\rangle, \langle 3,3\rangle, \langle
9,4\rangle, \langle 7,4\rangle, \langle
5,4\rangle, \\ \langle 9,5\rangle, \langle
7,5\rangle, \langle 5,5\rangle, \langle
9,6\rangle, \langle 7,6\rangle, \langle
9,7\rangle, \langle 7,7\rangle, \langle
9,8\rangle, \langle 9,9\rangle
\end{matrix}
\end{equation}
(see Section \ref{ss2}). In the same manner, we
will enumerate the rows and columns in every
$i$th $\star$-strip $(i\in \cal I)$ of $P(x,y)$.

We will prove that $P(x,y)$ satisfies the
condition (II) of Theorem \ref{t0.1}. Let $(W,B)$
be a canonical pair of $m\times m$ matrices under
simultaneous similarity; put
\begin{equation}       \label{7.3}
K=P(W,B)
\end{equation}
and denote by $\bar Q_i,\ \bar T_i,\ \bar L,\
\bar R$ the blocks of $K$ that correspond to
$Q_i,\ T_i,\ L,\ R$ (see \eqref{4.12}) from
$P(x,y)$:
\begin{gather}
\bar L={\bar \varPi} \bar J {\bar
\varPi}^{-1},\quad \bar R=J_5(a_rI_{2m}),
\label{7.3'} \\ \bar J:=J_1(a_l I_m)\oplus
J_3(a_l I_m) \oplus J_5(a_l I_m) \oplus J_7(a_l
I_m) \oplus J_9(a_l I_m), \label{7.3'a}
\end{gather}
where ${\bar \varPi}$ is a permutation matrix. It
suffices to show that $K$ is a canonical matrix
(i.e., $K$ is stable relatively to the algorithm
of Section \ref{sec3}). To prove it, we will
construct the partition of $K$ into boxes.

Clearly, the boxes $M_1,\dots,M_u$ of $M$ convert
to the boxes $K_1,\dots,K_u$ of $K$. The box
$M_v$ of $M$ is replaced by the box $K_{\bar v}$
of $K$. The numbers $v$ and $\bar v$ may be
distinct since $M_u$ and $K_u$ may have distinct
numbers of cells. The part $K_1\cup\dots\cup
K_{\bar v}$ of $K$ is in canonical form. The
partition of $K$ obtained after reduction of
$K_1,\dots,K_{\bar v}$ is the partition into
$\bar v$-strips; the corresponding blocks will be
called $\bar v$-{\it blocks}; for instance,
${\bar T}_1,\dots, {\bar T}_5$ are $\bar
v$-blocks.

The transformations of $K$ that preserve the
boxes $K_1,\dots,K_{\bar v}$ are
\begin{equation}       \label{4.25}
K \longmapsto K'=SKS^{-1},\quad S\in
\bar\Lambda_{0}^*.
\end{equation}
For every matrix $S$ from the algebra
$\bar\Lambda_0$, the family ${\cal S}_{\cal I\cal
J}^{\star}$ of its $\star$-blocks satisfies the
system \eqref{4.17}--\eqref{4.19}, so
$S_1,\dots,S_t\in S_{\cal I \cal J}^{\star}$
(which correspond to $S_1,\dots,S_t$ for
$S\in\Lambda_0$) are arbitrary satisfying the
equations \eqref{4.20} and the other
$S_{ij}^{\star}\in S_{\cal I \cal J}^{\star}$ are
their linear combinations.
\end{step}

\begin{step}                       \label{s4}
 {\it We prove the following statement:}
\begin{equation}       \label{4.26}
\parbox{25em}
{Let $p\in\{1,\dots,5\}$ and let the matrix $K$
be reduced by those transformations \eqref{4.25}
that preserve all $\bar v$-blocks preceding $\bar
T_p$. Then $\bar T_p$ is transformed into $\bar
T_p'=A_p\bar T_pB,$ where $A_p$ is an arbitrary
nonsingular matrix and $B$ is a nonsingular
matrix for which there exist nonsingular matrices
$A_{p+1},\dots,A_5$ satisfying $\bar
T_{p+1}=A_{p+1}\bar T_{p+1}B,\dots, \bar
T_5=A_5\bar T_5B $.}
\end{equation}

The rows and columns of $P(x,y)$ convert to the
{\it substrips} of $K=P(W,B)$. For every
$i\in\cal I,$ we have enumerated the rows and
columns in the $i$th $\star$-strip of $P(x,y)$ by
the pairs \eqref{4.14}; we will use the same
indexing for the substrips in the $i$th
$\star$-strip of $K$.

By analogy with \eqref{4.15}, equating in
$K'S=SK$ (see \eqref{4.25}) the blocks on the
place of $K_{\zeta \eta}^{\star}$ gives
\begin{equation}       \label{4.27}
K^{\star\prime}_{\zeta 1}S^{\star}_{1\eta}+\dots+
K^{\star\prime}_{\zeta \eta}S^{\star}_{\eta\eta}=
S^{\star}_{\zeta \zeta }K^{\star}_{\zeta
\eta}+\dots+ S^{\star}_{\zeta e}K^{\star}_{e\eta}
\end{equation}
For $p$ from \eqref{4.26} and $ i\in\cal I$, we
denote by $ C_{\zeta i},\ \hat K'_{\zeta i},\
\hat K_{\zeta i}$ (resp. $D_{\zeta i})$ the
matrices that are obtained from $S^{\star}_{\zeta
i},\  K^{\star\prime}_{\zeta i},\
K^{\star}_{\zeta i}$ (resp. $K_{\zeta
i}^{\star}$) by deletion of all horizontal
(resp., horizontal and vertical) substrips except
for the substrips indexed by $\langle
2p-1,p\rangle,$ $ \langle 2p-1,p+1\rangle,\dots,
\langle 2p-1,2p-1\rangle $. Then \eqref{4.27}
implies
\begin{equation}       \label{4.28}
\hat K'_{\zeta 1}S^{\star}_{1\eta}+\dots+ \hat
K'_{\zeta \eta}S^{\star }_{\eta\eta}= C_{\zeta
\zeta }K^{\star}_{\zeta \eta}+\dots+ C_{\zeta
e}K^{\star}_{e\eta}.
\end{equation}

The considered in \eqref{4.26} transformations
\eqref{4.25} preserve all $\bar v$-blocks
preceding $\bar T_p$. Since $\bar T_p$ is a $\bar
v$-block from the $\langle 2p-1,p\rangle $
substrip of the $\zeta $th horizontal
$\star$-strip whose substrips are ordered by
\eqref{4.14}, the block $\hat K_{\zeta i}$
($i<\eta$) is located in a part of $K$ preserved
by these transformations, that is $\hat K'_{\zeta
i}=\hat K_{\zeta i}$. If $\eta>i\in{\cal I}$,
then $\hat K'_{\zeta i} S^{\star}_{i\eta}=
D_{\zeta i}C_{i\eta}$ since $K_{\zeta i}^{\star}$
is a scalar matrix or $\bar L$ (see
\eqref{7.3'}). If $\eta>i\in{\cal J}$, then $\hat
K'_{\zeta i}= \hat K_{\zeta i}=0$. So the
equality \eqref{4.28} is presented in the form
\begin{equation}       \label{4.29}
\sum_{i=1}^{\eta-1} D_{\zeta i}C_{i\eta}+ \hat
K'_{\zeta \eta}S^{\star }_{\eta\eta}= C_{\zeta
\zeta }K^{\star}_{\zeta \eta}+ \sum_{i=\zeta
+1}^e C_{\zeta i}K^{\star}_{i\eta}.
\end{equation}

The equality \eqref{4.29} contains $C_{ij}$ only
if $(i,j)\in {\cal I}\times\cal J$, so each of
them is a part of $S^{\star}_{ij}\in
S^{\star}_{\cal I\cal J}$.  We have chosen
$S_1,\dots,S_t$ in $S^{\star}_{\cal I\cal J}$
such that they are arbitrary and the others are
their linear combinations; let $C_1,\dots,C_t$ be
the corresponding parts of $S_1,\dots,S_t$. It is
easy to show that $C_1,\dots,C_t$ satisfy the
system that is obtained from
\eqref{4.20}$\cup$\eqref{4.22} with $w_3+\alpha
=w+1$ by replacing $S_1,\dots,S_t$ with
$C_1,\dots,C_t$. Each $D_{\zeta i}$ in
\eqref{4.29} is a scalar or zero matrix if $
K^{\star}_{\zeta i}$ is not $\bar L$ and
$D_{\zeta i}=J_p(a_lI_m)$ otherwise, each $
K^{\star}_{i \eta}$ ($i<\zeta$) is a scalar or
zero matrix or $\bar R=J_5(a_rI_{2m})$, so the
equality \eqref{4.29} may be rewritten in the
form
\begin{equation}       \label{4.30}
C_1^{f_{w+1,1}}+\dots+ C_t^{f_{w+1,t}}= C_{\zeta
\zeta }K^{\star}_{\zeta \eta}- \hat K'_{\zeta
\eta}S^{\star}_{\eta\eta},
\end{equation}
where $f_{w+1,j}$ are the same as in \eqref{4.23}
and $C_i^{f_{w+1,i}}$ is defined by (\ref{4.1})
with $L=J_p(a_l I_m)$ and $R=J_5(a_rI_{2m})$. By
\eqref{4.5'}, the left-hand side of \eqref{4.30}
has the form $C^d$, so
\begin{equation}       \label{4.32}
C^d=C_{\zeta \zeta }K^{\star}_{\zeta \eta}- \hat
K'_{\zeta \eta}S^{\star}_{\eta\eta}.
\end{equation}

Let us study the right-hand side of \eqref{4.32}.
Since $\zeta \in \cal I$ and $\eta\in \cal J$,
the blocks $S_{\zeta  \zeta }^{\star}$ and
$S_{\eta \eta}^{\star}$ are arbitrary matrices
satisfying
\begin{equation}       \label{4.33}
S_{\zeta  \zeta }^{\star} {\bar L}={\bar L}
S_{\zeta  \zeta }^{\star} , \quad
S_{\eta\eta}^{\star}{\bar R}={\bar R}
S_{\eta\eta}^{\star}.
\end{equation}
By \eqref{7.3'} and \eqref{4.33}, $Z:=\bar
\varPi^{-1}S^{\star}_{\zeta \zeta }\bar \varPi $
commutes with $\bar J$. Let us partition $Z$ into
blocks $Z_{ij}\ (i,j=1,\dots,5)$ and $X:=\bar
\varPi ^{-1}(S_{\zeta \zeta
}^{\star}K^{\star}_{\zeta \eta}- K_{\zeta
\eta}^{\star\prime} S^{\star}_{\eta\eta})= Z[\bar
Q_i]- [\bar Q_i']S_{\eta\eta}^{\star}$ (recall
that $K_{\zeta \eta}^{\star}=\bar \varPi[\bar
Q_i]$) into horizontal strips $X_1,\dots, X_5$ in
accordance with the partition of $\bar J$ into
diagonal blocks $J_1(a_lI_m),$ $J_3(a_l
I_m),\dots,J_9(a_l I_m)$ (see \eqref{7.3'a}).
Then
\begin{equation*}      
X_p=Z_{p1}\bar Q_1+\dots+Z_{p5}\bar Q_5-Q_p'
S_{\eta\eta}^{\star}.
\end{equation*}

Since $Z$ commutes with $\bar J$,
$Z_{pi}J_{2i-1}(a_l I_m)= J_{2p-1}(a_l I_m)
Z_{pi}$. Hence $Z_{pi}$ has the form $$
\begin{bmatrix}
A_i&&&\text{\LARGE *}\\ &A_i&&\\&&\ddots&\\
&&&A_i\\ &&&\\ \text{\LARGE 0}&&&
\end{bmatrix}
\quad  {\rm or}\quad
\begin{bmatrix}
&&A_i&&&\text{\LARGE *}\\ &&&A_i&&\\&&&&\ddots&\\
\text{\LARGE 0}&&&&&A_i
\end{bmatrix}$$
if $p\ge i$ or $p\le i$ respectively. We look at
$$ \bar Q_i=
\begin{bmatrix}
0&0&0&0&0\\[-7pt] \hdotsfor{5}\\0&0&0&0&0\\
\bar T_i&0&0&0&0\\ 0&0&0&0&0\\[-7pt]
\hdotsfor{5}\\0&0&0&0&0\\
\end{bmatrix},
 \ \text{to get }\
X_p=
\begin{bmatrix}
*&*&*&*&*\\[-7pt] \hdotsfor{5}\\ *&*&*&*&*\\
A_p\bar T_p-\bar T'_pB&*&*&*&*\\
0&0&0&0&0\\[-7pt] \hdotsfor{5}\\0&0&0&0&0\\
\end{bmatrix},
$$ where $A_p$ is the diagonal $m\times m$ block of
$Z_{pp}$ and $B$ is the diagonal $2m\times 2m$
block of $S_{\eta\eta}^{\star}$ (recall that
$S_{\eta\eta}^{\star}$ commutes with
$J_5(a_rI_{2m})$). Since $X_p$ is formed by the
substrips of $S_{\zeta \zeta
}^{\star}K^{\star}_{\zeta \eta}-  K_{\zeta
\eta}^{\star\prime}S^{\star}_{\eta\eta}$ indexed
by the pairs $\langle 2p-1,1\rangle,\dots,
\langle 2p-1,2p-1\rangle $, the equality
\eqref{4.32} implies
\begin{equation}       \label{4.34}
R:=C^d=
 \begin{bmatrix}
A_p\bar T_p-\bar T'_pB&*&*&*&*\\
0&0&0&0&0\\[-7pt] \hdotsfor{5}\\0&0&0&0&0\\
\end{bmatrix}.
\end{equation}

Let us prove that
\begin{equation}       \label{4.35}
A_p\bar T_p=\bar T'_pB.
\end{equation}
If $d=0$, then the equality \eqref{4.35}  follows
from \eqref{4.34}. Let $d\ne 0$. We partition $C$
and $R=C^d$ into $p\times 5$ blocks $C_{ij}$ and
$R_{ij}$ conformal to the block form of the
matrix on the right-hand side of \eqref{4.34}. By
Lemma \ref{l4.1}(ii) and \eqref{4.34},
\begin{gather}       \label{4.36}
R=C^d=\sum_{ij}b_{ij}J_p(0_m)^iCJ_5(0_{2m})^j,
\\ R_{ij}=0 \ \ {\rm if}\ \  i>1, \label{4.36a}
\end{gather}
where $b_{00}=d(a_l,a_r)=0$ and
$b_{01}=d_y'(a_l,a_r)\ne 0$ (see \eqref{4.24}).
Hence $R_{p1}=0$, it proves \eqref{4.35} for
$p=1$. Let $p\ge 2$, then $R_{p2}=b_{01}C_{p1}=0$
by \eqref{4.36} and $C_{p1}=0$  by \eqref{4.36a}.
Next,  $R_{p3}=b_{01}C_{p2}=0$ by \eqref{4.36}
and $C_{p2}=0$ by \eqref{4.36a}, and so on until
obtain $C_{p1}=\dots=C_{p4}=0$. By \eqref{4.36},
$R_{p-1,1}=0$, it proves \eqref{4.35} for $p=2$.
Let $p\ge 3$, then $R_{p-1,2}=b_{01}C_{p-1,1}=0$
and $C_{p-1,1}=0$; further,
$R_{p-1,3}=b_{01}C_{p-1,2}=0$ and $C_{p-1,2}=0$,
and so on until obtain
$C_{p-1,1}=\dots=C_{p-1,3}=0$. Therefore,
$R_{p-2,1}=0$; we have \eqref{4.35} for  $p=3$
and $C_{p-2,1}=C_{p-2,2}=0$ otherwise.
Analogously, we get \eqref{4.35} for  $p=4$ and
$C_{p-3,1}= 0$ otherwise, and, at last,
\eqref{4.35} for  $p=5$.

By \eqref{4.26}, the considered transformation
preserves all $\bar v$-blocks preceding $\bar
T_l$, so we may repeat this reasoning for each
$l\in\{p+1,\dots,5\}$ instead of $p$ and obtain $
A_l\bar T_l=\bar T_lB$. It proves \eqref{4.26}.
\end{step}

\begin{step}[a construction of $ K_{\bar v+1},
 K_{\bar v+2},\ldots$]
          \label{s5}
The boxes $K_1,\dots, K_{\bar v}$ were
constructed at the end of Step \ref{s2}. The
first nonzero free $\bar v$-block of $K$ that is
not contained in $K_1\cup\dots \cup K_{\bar v}$
is $\bar T_5=[0_m\,I_m]$. The $\bar v$-blocks
that preceding $\bar T_5$ and are not contained
in $K_1\cup\dots\cup K_{\bar v}$ are zero, so
they are the boxes $K_{\bar v+1},\dots,
K_{v_1-1}$ for a certain $v_1\in\mathbb N$. By
the statement \eqref{4.26}, the admissible
transformations with $K$ that preserve the boxes
$K_1,\dots, K_{v_1-1}$ reduce, for $\bar T_5$, to
the equivalence transformations; therefore, $\bar
T_5=[0_m\,I_m]$ is canonical and $K_{v_1}=\bar
T_5$.

Conformal to the block form of
$K_{v_1}=[0_m\,I_m]$, we divide each $\bar
v$-block of $K$ into two $v_1$-blocks. The first
nonzero free $v_1$-block that is not contained in
$K_1\cup\dots \cup K_{v_1}$ is $I_m$ from $\bar
T_4=[I_m\,0_m]$. The $v_1$-blocks that preceding
it and are not contained in $K_1\cup\dots \cup
K_{v_1}$ are the boxes $K_{v_1+1},K_{v_2},\dots,
K_{v_2-1}$ for a certain $v_2\in\mathbb N$. By
the statement \eqref{4.26}, the admissible
transformations with $K$ that preserve the boxes
$K_1,\dots, K_{v_2-1}$ reduce, for $\bar T_4$, to
the transformations of the form $$ \bar
T_4\longmapsto A\bar T_4
\begin{bmatrix}
B&C\\ 0&B
\end{bmatrix}
$$ with nonsingular $m\times m$ matrices $A$ and $B$.
Since the block $\bar T_4=[I_m\,0_m]$ is
canonical under these transformations, we have
$\bar T_4=[I_m\,0_m]=[K_{v_2}|K_{v_2+1}]$; and so
on until we get the partition of $K$ into boxes.
\end{step}

It remains to consider the case $u=v$; in this
case the parameters $\lambda_1$ and $\lambda_2$
are parameters of a certain free box $M_v$. Since
$\lambda_1$ and $\lambda_2$ are distinct (by
prescribing of Case 1) parameters of the same
Weyr matrix $M_v$, $a_1\ne a_2$ for all $(a_1,
a_2)$ from the domain of parameters ${\cal
D}\subset k^2$. We will assume that the
parameters $\lambda_1$ and $\lambda_2$ are
enumerated such that there exists $(a_1, a_2)\in
{\cal D}$ with $a_1\prec a_2$, then by Definition
\ref{d2'.1} of Weyr matrices $a_1\prec a_2$ for
all $(a_1, a_2)\in {\cal D}$.   By the minimality
of $\sum \underline{n}$, $M_v= J_{s_1}(\lambda_1
I)\oplus J_{s_2}(\lambda_2 I)$, all
$(v-1)$-strips are linked, and $M=H(M_v)=
\hat{H}_l\oplus\hat{H}_r$, where $H(a)$ is an
indecomposable canonical matrix for all $a\in k$,
$\hat{H}_l:=H(J_{s_l}(\lambda_l I))$ and
$\hat{H}_r:= H(J_{s_r}(\lambda_r I))$ (i.e.
$H=H_l=H_r$, see the beginning of Section
\ref{s4.3.1}). By the
$\underline{n}^{\star}\times\underline{n}^{\star}$
{\it partition} of $M$ into blocks
$M_{ij}^{\star}$, we mean the partition into
$(v-1)$-strips supplemented by the division of
every $(v-1)$-strip into two substrips in
accordance with the partition of $M_u$ into
subblocks $J_{s_1}(\lambda_1 I)$ and
$J_{s_2}(\lambda_2 I)$. Then $ J_{s_1}(\lambda_1
I)$ and $J_{s_2}(\lambda_2 I)$ are free
$\star$-blocks, the other $\star$-blocks are zero
or scalar matrices, and $M_q$ is a part of a
$\star$-block. The reasoning in this case is the
same as in the case $u<v$ (but with ${\cal
B}_{uv}=\varnothing$).


\subsubsection{Study Case 2}  \label{s4.3.2}

In this case, $M=M(\lambda)$ is a one-parameter
matrix with an infinite domain of parameters
${\cal D}\subset k$. Up to permutation of
$(q-1)$-strips, $M$ has the form
$\hat{H}_1\oplus\hat{H}_2$, where $H_1(a)$ and
$H_2$ are indecomposable canonical matrices for
all $a\in k$, $\hat{H}_1:=H_1(J_{s_1}(\lambda
I))$, and $\hat{H}_2$ is obtained from $H_2$ by
replacement of its elements $h_{ij}$ with
$h_{ij}I_{s_2}$. The matrix $ J_{s_1}(\lambda I)$
is a part of $M_v$ (see \eqref{4.8}). Let $l,r\in
\{1,2\}$ be such that the horizontal
$(q-1)$-strip of $M_q$ crosses $\hat{H}_l$ and
its vertical $(q-1)$-strip crosses $\hat{H}_r$.
Under the ${\star}$-partition of $M$, we mean the
partition obtained from the $(q-1)$-partition by
removing the divisions inside of $
J_{s_1}(\lambda I)$ and the corresponding
divisions inside of the horizontal and vertical
$(v-1)$-strips of $M_v$ and all $(v-1)$-strips
that are linked with them; then $M_v$ is a
$\star$-block. Denote by ${\cal I}$ (resp. ${\cal
J}$)  the set of indices of $\star$-strips of
$\hat{H}_l$ (resp.  $\hat{H}_r$) in $M$.

Let $M_z$ be the last nonzero free box of $M$
(clearly, $z\ge v$). Denote by ${\cal B}$ the
part of $M$ consisting of all entries that are in
the intersection of $\cup_{i\le z}M_i$  and
$\cup_{(i,j)\in{\cal I}\times{\cal
J}}M_{ij}^{\star}$. By analogy with Case 1,
${\cal B}$ is a union of $\star$-blocks
$M_{ij}^{\star}$ for some $(i,j)\in{\cal
I}\times{\cal J}$.

Let $\Lambda_0$ be the algebra of all
$S\in\Lambda$ such that $MS-SM$ is zero on the
places of the boxes $M_i\le M_z$. Equating zero
the blocks of $MS-SM$ on the places of all free
$\star$-blocks $M_{ij}^{\star}$ from ${\cal B}$,
we obtain a system of equalities of the form
\eqref{4.18}--\eqref{4.19} with
$g_{ij}^{(\tau)}\in k[x]$ if $l=1$ and
$g_{ij}^{(\tau)}\in k[y]$ if $l=2$ for
$\star$-blocks of $S=[S_{ij}^{\star}]\in
\Lambda_0$ from the family $S_{\cal{I
J}}^{\star}:= \{S_{ij}^{\star}\,|\,(i,j)\in{\cal
I}\times {\cal J} \}$. Solving the system
\eqref{4.17}--\eqref{4.18}, we choose
$S_1,\dots,S_t\in S_{\cal{IJ}}^{\star}$  such
that they are arbitrary and the others are their
linear combinations, then we present the system
\eqref{4.19} in the form \eqref{4.20}.

Let $ F_1<F_2<\dots<F_{\delta}$ be the sequence
of all free $M^{\star}_{ij}$ such that
$M^{\star}_{ij}\not\subset M_1\cup\dots \cup M_z$
and $(i,j)\in {\cal I}\times{\cal J}$. Denote by
$\Lambda_{\alpha}$ $(\alpha\in
\{1,\dots,\delta\})$ the algebra of all
$S\in\Lambda_0$ for which $MS$ and $SM$ are
coincident on the places of all
$M^{\star}_{ij}\le F_{\alpha}$; it gives
additional conditions \eqref{4.22} on $
S_{\cal{IJ}}^{\star}$.

By analogy with Case 1, the transformation
\eqref{7.1} preserves all $M_i$ with $i\le z$ and
all $M^{\star}_{ij}\le F_{\alpha}$; moreover, any
sequence of matrices $S_1,\dots,S_t$ is the
sequence of the corresponding blocks of a matrix
$S\in \Lambda_{\alpha}$ if and only if the system
\eqref{4.20}$\cup$\eqref{4.22} holds.

Putting $\alpha=\delta$, $p=w_3+\delta$,
$D=\{(a,a)\,|\,a\in{\cal D}\}$  and applying
Lemma \ref{l4.2} to
\eqref{4.20}$\cup$\eqref{4.22} (note that
$f_{ij}\in k[x]$ or $f_{ij}\in k[y]$), we get an
infinite set ${D}' \subset {D}$, a polynomial
$d$, an integer $w\le \min(p-1, t)$, and
$j_1,\dots,j_{t-w}\in\{1,\dots, t\}$ satisfying
the conditions (i)--(ii) of Lemma \ref{l4.2}. The
polynomial $d\in k[x]\cup k[y]$ is zero since it
is zero on the infinite set $\{a\,|\,(a,a)\in
D'\}$.

Let us fix $a_1,\dots,a_5\in {D}',\ a_1\prec
a_2\prec\dots \prec a_5$ (with respect to the
ordering in $k$, see the beginning of Section
\ref{ss2}), and denote by $P(x,y)$ the matrix
that is obtained from $M$ by replacement of

(i) its $\star$-block $J_{s_1}(\lambda I)$ with
$\diag(a_1,a_2,\dots,a_5)$,

(ii) all entries $h_{ij}I_{s_2}$ of $\hat{H}_2$
with $h_{ij}I_2$, and

(iii) $M_{\zeta  \eta}^{\star}$ with $T$ (see
\eqref{4.13}) if $l=1$ and with $$
\begin{bmatrix}
1&0&1&x&y\\ 0&1&1&1&1
\end{bmatrix} \ \ \text{if}\ \ l=2,$$
and by the corresponding correction of dependent
blocks. As in Case 1, we can prove that $P(x,y)$
satisfies the condition (II) of Theorem
\ref{t0.1}.


\subsubsection{Study Case 3}  \label{s4.3.3}

The free box $M_v$ is a Weyr matrix that is
similar to $J_{s_1}(\lambda I)\oplus
J_{s_2}(\lambda I)$ ($s_1\ne s_2$) or $
J_{s}(\lambda I)$, hence it has the form
$M_v=\lambda I+F$, where $F$ is a nilpotent upper
triangular matrix. Clearly, $M=M(\lambda)$ is a
one-parameter matrix with an infinite domain of
parameters ${\cal D}\subset k$; moreover,
$M=H(M_v)$, where $H(a)\ (a\in k)$ is an
indecomposable canonical matrix. Under the
$\star$-partition we mean the partition into
$(v-1)$-strips (then $M_v$ is a $\star$-block).

\setcounter{step}{0}
\begin{step}[a construction of  $P(x,y)$]
      \label{s5.1}

Let $\Lambda_{-1}$ (resp. $\Lambda_{0}$) be the
algebra of all $S\in\Lambda$ such that $MS-SM$ is
zero on the places of the boxes $M_i< M_{v}$
(resp. $M_i\le M_{v}$). Then $\Lambda_{-1}$ is a
reduced
$\underline{n}^{\star}\times\underline{n}^{\star}$
algebra whose equivalence relation \eqref{0} in
$T^{\star}=\{1,\dots,e\}$ is full (i.e. every two
elements are equivalent). The blocks of
$S\in\Lambda_{-1}$ satisfy a system of equations
of the form
\begin{gather}
S_{11}^{\star}= S_{22}^{\star}=\dots=
S_{ee}^{\star}, \label{5.1}\\ \sum_{i < j}
c_{ij}^{(l)}S_{ij}^{\star}= 0,\quad l=1,2,\dots,
q_{_{T^{\star} T^{\star}}} \label{5.2}
\end{gather}
(see \eqref{3}). Solving the system \eqref{5.2},
we choose $S_1,\dots,S_t\in \{ S_{ij}^{\star}\,|
\, i<j\}$ such that they are arbitrary and the
other $ S_{ij}^{\star}\ (i<j)$ are their linear
combinations. The algebra $\Lambda_0$ consists of
all $S\in \Lambda_{-1}$ for with
$S^{\star}_{11}M_v= M_v S^{\star}_{11}$.

Let $F_1<F_2<\dots<F_{\delta}$ be the sequence of
all free $M^{\star}_{ij}\not\subset M_1\cup\dots
\cup M_v$, and let $\Lambda_{\alpha}$ $(\alpha\in
\{1,\dots,\delta\})$ denote the algebra of all
$S\in\Lambda_0$ for which $MS$ and $SM$ are
coincident on the places of all
$M^{\star}_{ij}\le F_{\alpha}$; it gives
conditions on $S_i$ of the form
\begin{equation}       \label{4.22a}
S_1^{f_{i1}}+\dots+ S_t^{f_{it}}=0,\quad
i=1,\dots,\alpha,
\end{equation}
where $ f_{ij}\in k[x,y]$ and $S_i^{f_{ij}}$ is
defined by (\ref{4.1}) with $L=R=M_v$.

Putting $p=\delta$, $D=\{(a,a)\,|\, a\in{\cal
D}\}$ and applying Lemma \ref{l4.2} to
\eqref{4.22a} with $\alpha:=\delta$, we get an
infinite ${D}' \subset {D}$, $d\in k[x,y]$, $w\le
\min(p-1, t)$, and $j_1,\dots,j_{t-w}
\in\{1,\dots, t\}$. Since $d(a,a)=0$ for all
$(a,a)\in {D}'$, $d(x,y)$ is divisible by $x-y$
by the Bezout theorem (see the proof of Lemma
\ref{l4.2}). We may take
\begin{equation}       \label{5.4'}
d(x,y)=x-y.
\end{equation}

Let us fix an arbitrary $a\in {D}'$ and denote by
$P(x,y)$ the matrix that is obtained from $M$ by
replacement of its $\star$-blocks $M_v$ and
$M_{\zeta \eta}^{\star}$ with
\begin{equation}       \label{5.4}
P_v=\begin{bmatrix} aI_2&0&I_2&0\\ 0&aI_1&0&0\\
0&0&aI_2&I_2\\ 0&0&0&aI_2
\end{bmatrix}\ \ \text{and}\ \
P_{\zeta  \eta}^{\star}=\begin{bmatrix} 0&0&0&0\\
0&0&0&0\\ T&0&0&0\\0&Q&0&0
\end{bmatrix},
\end{equation}
where
\begin{equation}       \label{5.5}
Q=\begin{bmatrix} 1\\ 0 \end{bmatrix},\quad
T=\begin{bmatrix} x &y\\ 1&0 \end{bmatrix},
\end{equation}
and by the corresponding correction of dependent
blocks. ($P_v$ is a Weyr matrix that is similar
to $J_1(a)\oplus J_3(aI_2)$.) We prove that
$P(x,y)$ satisfies the condition (II) of Theorem
\ref{t0.1}. Let $(W,B)$ be a canonical pair of
$m\times m$ matrices under simultaneous
similarity, put $K=P(W,B)$ and denote by $\bar Q$
and $\bar T$ the blocks of $K$ that correspond to
$Q$ and $T$. It suffices to show that $K$ is a
canonical matrix.
\end{step}

\begin{step}[a construction of $K_1,\dots, K_{v_1}$]
      \label{s5.2}
The boxes $M_1,\dots,M_v$ of $M$ become the boxes
$K_1,\dots,K_v$ of $K$.

Let us consider the algebra $\bar\Lambda_{-1}$
for  the matrix $K$. For each
$S\in\bar\Lambda_{-1}$, its $\star$-blocks
satisfy the system \eqref{5.2}, so  we may choose
$S_1,\dots,S_t\in \{S_{ij}^{\star}\, |\, i<j\}$
(on the same places as for $\Lambda_{-1}$) that
are arbitrary and the other $ S_{ij}^{\star}\
(i<j)$ are their linear combinations. A matrix
$S\in\bar\Lambda_{-1}$ belongs to $\bar\Lambda_0$
if and only if the matrix $S_{11}^{\star}=
S_{22}^{\star}=\cdots$ (see \eqref{5.1}) commutes
with $K_v$, that is
\begin{equation}       \label{5.6}
S_{11}^{\star}= S_{22}^{\star}=\cdots=
S_{ee}^{\star}=
\begin{bmatrix}
A_0&B_2&A_1&A_2\\ &B_0&0&B_1\\ &&A_0&A_1\\
\text{\LARGE 0}&&& A_0
\end{bmatrix}
\end{equation}
by \eqref{5.4} and by analogy with Example
\ref{e2'.1}.

The first nonzero free $v$-block of $K$ that is
not contained in $K_1\cup\dots\cup K_v$ is $\bar
Q$ (see \eqref{5.5}). The $v$-blocks that
preceding $\bar Q$ and are not contained in
$K_1\cup\dots\cup K_v$ are the boxes
$K_{v+1},\dots, K_{v_1-1}$ for a certain
$v_1\in\mathbb N$.

The block $\bar Q$ is reduced by the
transformations
\begin{equation}       \label{5.7}
K \longmapsto K'=SKS^{-1},\quad S\in
\bar\Lambda_0^*
\end{equation}
with the matrix $K$; these transformations
preserve the boxes $K_1,\dots,K_{v}$ of $K$. Each
$\star$-strip of $P(x,y)$ consists of 7 rows or
columns (since $P_v\in k^{7\times 7}$, see
\eqref{5.4}); they become the {\it substrips} of
the corresponding $\star$-strip of $K$. Denote by
$ C_{ij},\ \hat K'_{ij},\ \hat K_{ij}$ (resp.
$D_{ij})$ the matrices that are obtained from
$S^{\star}_{ij},\ K^{\star\prime}_{ij},\
K^{\star}_{ij}$ (resp. $K_{ij}^{\star}$) by
elimination of the first 5 horizontal (resp.,
horizontal and vertical) substrips; note that
$\bar Q$ is contained in the remaining 6th and
7th substrips of $K_{\zeta \eta}^{\star}$. The
equation \eqref{4.27} implies \eqref{4.28}. Since
all $K_{ij}^{\star}<K_{\zeta \eta}^{\star}$ are
upper triangular, the equation \eqref{4.28}
implies \eqref{4.29}.

The equality \eqref{4.29} takes the form
\eqref{4.30}, where $C_1,\dots,C_t$ are the
corresponding parts of $S_1,\dots,S_t;\
f_{w+1,j}$ are the same as in \eqref{4.22a} and
$C_j^{f_{w+1,j}}$ is defined by (\ref{4.1}) with
$L=aI_{2m}$ (a part of $K_v$) and $R=K_v$.

By \eqref{4.5'} and \eqref{5.4'},
\begin{equation*}
C_1^{f_{w+1,1}}+\dots+ C_t^{f_{w+1,t}}=C^{x-y};
\end{equation*}
by \eqref{4.30},
\begin{equation}       \label{4.32a}
C^{x-y}=C_{\zeta \zeta }K^{\star}_{\zeta \eta}-
\hat K'_{\zeta \eta}S^{\star}_{\eta\eta}.
\end{equation}
As follows from the form of the second matrix in
\eqref{5.4} and from \eqref{5.6},
\begin{equation}       \label{5.8}
S_{\zeta \zeta }^{\star}K_{\zeta \eta}^{\star}-
K_{\zeta \eta}^{\star\prime}
S_{\eta\eta}^{\star}=
\begin{bmatrix}
*&*&*&*\\ *&*&*&*\\ A_0\bar T-\bar T'A_0&*&*&* \\
0&A_0\bar Q-\bar Q'B_0&0&*
\end{bmatrix}.
\end{equation}
Looking at the form of the matrix $K_v$ (see
\eqref{5.4}), we have
\begin{equation}       \label{5.9}
K_vD-DK_v=
\begin{bmatrix}
*&*&*&*\\ *&*&*&*\\ D_{41}&*&*&* \\ 0&0&-D_{41}&*
\end{bmatrix}
\end{equation}
for an arbitrary block matrix $D=[D_{ij}]$. So
the equality \eqref{4.32a} can be presented in
the form
\begin{equation}       \label{5.9a}
\begin{bmatrix}
0&0&-D_{41}&*
\end{bmatrix}=\begin{bmatrix}
0&A_0\bar Q-\bar Q'B_0&0&*
\end{bmatrix},
\end{equation}
where $C=[D_{41}\, D_{42}\, D_{43}\, D_{44}]$. It
follows $ A_0\bar Q-\bar Q'B_0=0$ and $\bar
Q'=A_0\bar QB_0^{-1}$. Therefore, the block $\bar
Q$ is reduced by elementary transformations.
Since $\bar
Q=\left[\genfrac{}{}{0pt}{}{I}{0}\right]$ is
canonical, $K_{v_1}:=\bar Q$ is a box.
\end{step}

\begin{step}[a construction of
$K_{v_1+1},\dots, K_{v_2}$]
  \label{s5.3}
The partition into $v_1$-strips coincides with
the partition into substrips, so the $v_1$-blocks
are the subblocks of $K$ corresponding to the
entries of $P$. The first nonzero free subblock
of $K$ that is not contained in $K_1\cup\dots\cup
K_{v_1}$ is $\bar T_{21}=I_m$ from $\bar T=[{\bar
T}_{ij} ]_{i,j=1}^2$. The subblocks that
preceding $\bar T_{21}$ and are not contained in
$K_1\cup\dots\cup K_{v_1}$ are the boxes
$K_{v_1+1},\dots, K_{v_2-1}$ for a certain
$v_2\in\mathbb N$.

Let a transformation \eqref{5.7} preserve the
boxes $K_1,\dots,K_{v_2-1}$. Denote by $ C_{ij},\
\hat K'_{ij},\ \hat K_{ij}$ (resp. $D_{ij})$ the
matrices that are obtained from $ S^{\star}_{ij},
K^{\star\prime}_{ij},\ K^{\star}_{ij}$ (resp.
$K_{ij}^{\star}$) by elimination of the first 4
horizontal (resp., horizontal and vertical)
substrips; note that $\bar T_{21}=I_m$ is
contained in the 5th horizontal substrip of
$K_{\zeta \eta}^{\star}$. Let $C_1,\dots,C_t$ be
the corresponding parts of $S_1,\dots,S_t$.
Similar to Step \ref{s5.2}, we have the
equalities  \eqref{4.30} and \eqref{4.32a}. As
follows from \eqref{5.8} and \eqref{5.9}, the
equality \eqref{4.32a} may be presented in the
form
\begin{equation}       \label{5.10}
\begin{bmatrix}
(D_{41})_2&*&*&* \\ 0&0&-D_{41}&*
\end{bmatrix}=\begin{bmatrix}
(A_0\bar T-\bar T'A_0)_2&*&*&* \\ 0&A_0\bar
Q-\bar Q'B_0&0&*
\end{bmatrix}
\end{equation}
(compare with \eqref{5.9a}), where $(D_{41})_2$
and $(A_0\bar T-\bar T'A_0)_2$ are the lower
substrips of $D_{41}$ and $A_0\bar T-\bar T'A_0.$
It follows that $A_0\bar Q-\bar Q'B_0=0$,
$D_{41}=0$, and so $( A_0\bar T-\bar T'A_0)_2=0$.
But $\bar Q=\bar Q' =
\left[\genfrac{}{}{0pt}{}{I}{0}\right]$, hence
\begin{equation}       \label{5.11}
A_0=\begin{bmatrix} A_{11}& A_{12}\\ 0& A_{22}
\end{bmatrix},
\end{equation}
and we have $A_{22}\bar T_{21}-\bar
T'_{21}A_{11}=0$, so $\bar T_{21}$ is reduced by
equivalence transformations. Therefore, $\bar
T_{21}=I_m$ is canonical and $K_{v_2}=\bar
T_{21}=I_m$.
\end{step}

\begin{step}[a construction of
$K_{v_2},\ K_{v_2+1},\ldots$]
   \label{s5.4}
The partition into $v_2$-strips coincides with
the partition into substrips. The first nonzero
free subblock of $K$ that is not contained in
$K_1\cup \dots\cup K_{v_2}$ is $\bar T_{11}=W$
from $\bar T$. The subblocks that preceding $\bar
T_{11}$ and are not contained in
$K_1\cup\dots\cup K_{v_2}$ are the boxes
$K_{v_2+1},\dots, K_{v_3-1}$ for a certain
$v_3\in\mathbb N$.

Let a transformation \eqref{5.7} preserve the
boxes $K_1,\dots,K_{v_3-1}$. Denote by $ C_{ij},\
\hat K'_{ij},\ \hat K_{ij}$ (resp. $D_{ij})$ the
matrices that are obtained from $
S^{\star}_{ij},\ K^{\star\prime}_{ij},\
K^{\star}_{ij}$ (resp. $K_{ij}^{\star}$) by
elimination of the first 3 horizontal (resp.,
horizontal and vertical) substrips. In this case,
instead of \eqref{5.10} we get the equality $$
\begin{bmatrix}
D_{41}&*&*&* \\ 0&0&-D_{41}&*
\end{bmatrix}= \begin{bmatrix}
A_0\bar T-\bar T'A_0&*&*&* \\ 0&A_0\bar Q-\bar
Q'B_0&0&*
\end{bmatrix},$$
so $ A_0\bar T-\bar T'A_0=0$, where $A_0$ is of
the form \eqref{5.11}. Since $[\bar T_{21}\, \bar
T_{22}]=[\bar T_{21}'\, \bar T_{22}']=[I_m\,
0_m]$, we have $A_{11}=A_{22}$ and $A_{12}=0$, so
$A_{11}\bar T_{11}-\bar T_{11}'A_{11}=0$ and
$\bar T_{11}$ is reduced by similarity
transformations. Since $\bar T_{11}=W$ is a Weyr
matrix, it is canonical and $K_{v_3}=W$.

Furthermore, $A_{11}\bar T_{12}-\bar
T_{12}'A_{11}=0$, where $A_{11}$ commutes with
$W$, hence $\bar T_{12}=B$ is canonical too. It
proves that $K$ is a canonical matrix.
\end{step}


\subsection{Proof of Theorem \ref{t0.1} for
tame problems}
 \label{s4.4}

In this section, we consider a matrix problem
given by $(\varGamma,\, \cal{M})$ for which there
exists no semi-parametric canonical matrix $M$
having a free box $M_q\neq \emptyset$ with the
property \eqref{4.8}. Our purpose is to prove
that the matrix problem satisfies the condition
(I) of Theorem \ref{t0.1}.

Let $\underline{n}$ be a step-sequence. By Remark
\ref{r2.1}, the number of parametric canonical
${\underline{n}\times \underline{n}}$ matrices is
finite. Let $M$ be a parametric canonical
${\underline{n}\times \underline{n}}$ matrix and
one of its parameters is a finite parameter
$\lambda$; that is, the set of
$\lambda$-components in the domain of parameters
is a finite set $\{a_1,\dots,a_r\}.$ Putting
$\lambda=a_1,\dots,a_r$ gives $r$ semi-parametric
canonical matrices. Repeating this procedure, we
obtain a finite number of semi-parametric
canonical ${\underline{n}\times \underline{n}}$
matrices having only infinite parameters or
having no parameters.

Let $M$ be an indecomposable semi-parametric
canonical ${\underline{n}\times \underline{n}}$
matrix that has no finite parameters but has
infinite parameters, and let $M_v$ be the first
among its boxes with parameters (then $M_v$ is
free). By the property \eqref{4.8}, if a
$v$-strip is linked with a $v$-strip containing a
parameter $\lambda$ from $M_v$, then it does not
contain a free box $M_i>M_v$ such that $M_i\neq
\emptyset$. Since $M$ is indecomposable, it
follows that all its $v$-strips are linked, all
free boxes $M_i>M_v$ are equal to $\emptyset$,
and $M_v=J_m(\lambda)$. Hence, all free
$v$-blocks excepting $M_v$ are scalar matrices
and $M=L(J_m(\lambda))$, where
$L(\lambda)=[a_{ij}+\lambda b_{ij}]$ is a
semi-parametric canonical matrix with a free
$1\times 1$ box $M_v=[\lambda]$ and all free
boxes after it are $1\times 1$ matrices of the
form $\emptyset$.

Let ${\cal D}_m\subset k$ be the domain of
parameters of $M$. By the property \eqref{4.8},
${\cal D}_m$ is a cofinite set (i.e. $k\setminus
{\cal D}_m$ is finite).

If $a\notin {\cal D}_m$, then the matrix $M(a)$
is canonical and there exists a free box
$M_q>M_v$ such that $M_q\neq \emptyset$. This box
$M_q$ is the zero $1\times 1$ matrix. Since $M$
is indecomposable, all its rows and columns are
linked, so $M_q$ is reduced by similarity
transformations. Replacing it by the parametric
box $[\mu]$, we obtain a straight line of
indecomposable canonical matrices that intersects
$\{M(\lambda)\,|\,\lambda\in k\}$ at the point
$M(a)$. Hence, each $M(a)$, $a\notin {\cal D}_m$,
is a point of intersection of $\{M(\lambda)\,|\,
\lambda\in k\}$ with a straight line of
indecomposable canonical matrices.

Let $M(a)$, $a\in {\cal D}_m$, be a point of
intersection too; that is, there exists a line
$\{N(\mu)\,|\,\mu\in k\}$ of indecomposable
canonical matrices such that $M(a)=N(b)$ for a
certain $b\in k$. Then $M(\lambda)$ has a free
box $M_u$ ($u<v$) that is a Weyr matrix, $b$ is
its eigenvalue, and $N(\mu)$ is obtained from
$M(a)$ by replacement of $b$ with $\mu$. Since
$M(\lambda)$ and $N(\mu)$ coincide on
$M_1\cup\dots\cup M_{u-1}$, $M_u=N_u$ for
$\mu=b$. By analogy with the structure of
$M(\lambda)$, all free boxes $N_i>N_u$  are zero,
hence $M_v=0$ if $\lambda=a$. Since
$M_v=J_m(\lambda)$, $M(a)$ with $a\in {\cal D}_m$
can be a point of intersection only if $m=1$ and
$\lambda=0$.

Replacing $m$ by an arbitrary integer $n$ gives a
new semi-parametric canonical matrix
$L(J_n(\lambda))$ with the domain of parameters
${\cal D}_n$. To prove that the condition (I) of
Theorem \ref{t0.1} holds, it suffices to show
that ${\cal D}_m={\cal D}_n$. Moreover, it
suffices to show that ${\cal D}_m={\cal D}_1$.

Let first $a\in{\cal D}_1$. By analogy with
Section \ref{s4.3.3}, under the $\star$-partition
we mean the partition into $(v-1)$-strips. Then
$a\in{\cal D}_m$ if and only if all free
$\star$-blocks after $M_v$ in $M(a)$ are
$\emptyset$. The $\star$-blocks of every $S\in
\Lambda_{-1}^*$ (see Section \ref{s4.3.3})
satisfy the system \eqref{5.1}--\eqref{5.2},
where $c_{ij}^l$ do not depend on $m$ and $a$.
Solving the system \eqref{5.2}, we choose
$S_1,\dots,S_t\in \{ S_{ij}^{\star}\,|\, i<j\}$
such that they are arbitrary and the other $
S_{ij}^{\star}\ (i<j)$ are their linear
combinations.

Let $F_1<F_2<\dots<F_{\delta}$ be the sequence of
all free $M^{\star}_{ij}\not\subset M_1\cup\dots
\cup M_v$ and let $K$ be obtained from $M$ by
replacing $F_1,\dots,F_{\delta}$ with arbitrary
$m\times m$ matrices $G_1,\dots,G_{\delta}$. To
prove that $a\in{\cal D}_m$, we must show that
$F_1=\dots=F_{\delta}=\emptyset$ for $M(a)$; that
is, there exists $S\in \Lambda_0^*$ such that
$G'_1=\dots=G'_{\delta}=0$ in $K':= SKS^{-1}$. It
suffices to consider the case
$G_{1}=\dots=G_{q-1}=0\ne G_{q}$ ($q\in
\{1,\dots,\delta\}$) and to show that there
exists $S\in \Lambda_{-1}^*$ with
$S_{11}^{\star}= S_{22}^{\star}=\dots =I_m$ (then
$S\in \Lambda_{0}^*$) such that
$G'_{1}=\dots=G'_{q-1}= G'_{q}=0$. It means that
the $\star$-blocks $S_1,\dots, S_t$ of $S$
satisfy the system of equations that is obtained
by equating in $K'S=SK$ the blocks on the places
of $G_{1},\dots,G_{q}$:
\begin{align}
S_1^{f_{l1}}+\dots+ S_t^{f_{lt}} &=0,\quad
l=1,\dots,q-1, \label{5.12}\\ S_1^{f_{q1}}+\dots+
S_t^{f_{qt}} &= G_{q}^{\varphi},   \label{5.13}
\end{align}
where ${\varphi}(a,a)\ne 0$ and $S_j^{f_{ij}}$ is
defined by (\ref{4.1}) with $L=R= J_m(a)$. Note
that the polynomials $f_{ij}$ are the same for
all $m\in\mathbb N$ and $a$.

Taking 1 instead of $m$, we obtain the system
 \begin{align*} f_{l1}(a,a)s_1+\dots+
f_{lt}(a,a)s_t &=0,\quad l=1,\dots,q-1,
\label{5.14}\\ f_{q1}(a,a)s_1+\dots+
f_{qt}(a,a)s_t &= g.
 \end{align*}
Since $a\in{\cal D}_1$, this system is solvable
with respect to $s_1,\dots,s_t$ for all $ g\in
k$. It holds for all $q$, so the rows of
$F:=[f_{ij}(a,a)]$ are linearly independent.

Let $S_r=[s_{ij}^{(r)}]_{i,j=1}^m$ and
$G_{q}^{\varphi}=[g_{ij}]_{i,j=1}^m$. Since $L=R=
J_m(a)$, the system of $q$ matrix equations
\eqref{5.12}--\eqref{5.13} is equivalent to the
$m^2$ systems of $q$ linear equations relatively
to the entries of $S_1,\dots,S_t$, each of them
is obtained by equating the $(i,j)$ entries for
the corresponding $i,j\in\{1,\dots,m\}$ and has
the form:
\begin{equation}  \label{5.16}
f_{l1}(a,a)s^{(1)}_{ij}+\dots+f_{lt}(a,a)s^{(t)}_{ij}
  =d^{(l)}_{ij}, \quad l=1,\dots,q,
\end{equation}
where $d^{(l)}_{ij}$ is a linear combination of
$s^{(1)}_{i'j'},\dots, s^{(t)}_{i'j'}$,
$(i',j')\in \{(1,j),\dots,(i-1,j)\}\cup
\{(i,j+1),(i,j+2), \ldots\}$, and (only if $l=q$)
$g_{ij}$. Since the rows of $F=[f_{ij}(a,a)]$ are
linearly independent, the system \eqref{5.16} for
$(i,j)=(m,1)$ is solvable. Let $\bar s_{m1}=
(\bar{s}^{(1)}_{m1}, \dots, \bar{s}^{(t)}_{m1})$
be its solution. Knowing $\bar{s}_{m1}$, we
calculate $d^{(l)}_{m-1,1}$ and $d^{(l)}_{m2}$,
then solve the system \eqref{5.16}  for
$(i,j)=(m-1,1)$ and for $(i,j)=(m,2).$ We next
calculate $d^{(l)}_{ij},\ i-j=m-2$, and solve
\eqref{5.16} for
$(i,j)=(m-2,1),\,(m-1,2),\,(m,3),$ and so on,
until we obtain a solution $\bar{S}_1,\dots,
\bar{S}_t$ of \eqref{5.12}, a contradiction.
Hence $a\in{\cal D}_m$, which clearly implies
$a\in{\cal D}_1$. It proves Theorem \ref{t0.1}.

\begin{remark}
We can give a more precise description of the set
of canonical matrices based on the proof of
Theorem \ref{t0.1}. For simplicity, we restrict
ourselves to the case ${\cal M}=k^{t\times t}$.

Namely, a linear matrix problem given by a pair
$(\varGamma, k^{t\times t})$ satisfies one and
only one of the following two conditions
(respectively, is of tame or wild type):
  \begin{itemize}
  \item[(I)]
For every step-sequence $\underline{n}$, there
exists a finite set of semi-parametric canonical
$\underline{n}\times\underline{n}$ matrices
$M_{\underline{n},i}(\lambda)$, $i=1,\dots,
t_{\underline n},$ whose domains of parameters
${\cal D}_{\underline{n},i}$ are cofinite subsets
in $k$ and
\begin{itemize}
  \item[(a)]
for every $m\ge 1$,
$M_{\underline{n},i}(J_m(\lambda))$ is a
semi-parametric canonical matrix with the same
domain of parameters ${\cal D}_{\underline{n},i}$
and the following partition into boxes:
$J_m(\lambda)$ is a box, all boxes preceding it
are the scalar matrices $B_1\otimes
I_m,\dots,B_l\otimes I_m$ (where $B_1,\dots,B_l$
are the boxes of $M_{\underline{n},i}(\lambda)$
preceding $[\lambda]$), and all boxes after it
are the $1\times 1$ matrices $\emptyset$;

\item[(b)]
for every $\underline{n}'$, the set of matrices
of the form $M_{\underline{n},i}(J_m(a))$,
$m\underline{n}=\underline{n}'$, $a\in {\cal
D}_{\underline{n},i}$, is a cofinite subset in
the set of indecomposable canonical
$\underline{n}'\times\underline{n}'$ matrices.
\end{itemize}

\item[(II)]
There exists a semi-parametric canonical
$\underline{n}\times \underline{n}$ matrix
$P(\alpha,\beta)$ (in which two entries are the
parameters $\alpha$ and $\beta$ and the other
entries are elements of $k$) such that

\begin{itemize}
\item[(a)]
two pairs of $m\times m$ matrices $(A,B)$ and
$(C,D)$ are similar if and only if $P(A,B) \simeq
P(C,D)$; moreover,
\item[(b)]
a pair of $m\times m$ matrices $(A,\,B)$ is
canonical under similarity (see Definition
\ref{d1.2a}) if and only if the
$m\underline{n}\times m\underline{n}$ matrix
$P(A,\,B)$ is canonical.
 \end{itemize}
\end{itemize}
\end{remark}

\section*{Acknowledgements}
I wish to thank P. Gabriel, L. A. Nazarova, A. V.
Roiter, and D. Vossieck; a joint work on the
article \cite{gab_vos} was a great inspiration
for me and introduced me to the theory of tame
and wild matrix problems. The idea to prove the
Tame--Wild Theorem using Belitski\u{\i}'s
algorithm came from this work and was first
discussed in my talks at the Zurich University in
1993 by invitation of P. Gabriel.

I wish to thank C. M. Ringel for the invitations
to speak on the contents of this paper at the
University of Bielefeld in 1998--1999 and
stimulating discussions, and for the publication
of the preprint \cite{ser_p}.

I am grateful to G. R. Belitski\u\i, T.
Br\"ustle, Yu. A. Drozd, S. Friedland, D. I.
Merino, and the referee for helpful suggestions
and comments.

The work was partially supported by the Long-Term
Research Grant No. U6E000 of the International
Science Foundation and by Grant No. UM1-314 of
the U.S. Civilian Research and Development
Foundation for the Independent States of the
Former Soviet Union.

\end{document}